\documentclass[12pt,reqno]{amsart}

\usepackage[utf8]{inputenc}
\usepackage{amsmath} 
\usepackage{amsthm} 
\usepackage{amssymb} 
\usepackage{amscd} 

\usepackage{emptypage}
\usepackage{enumitem} 
\usepackage{mathtools} 
\usepackage{mathrsfs} 
\usepackage{upgreek}

\usepackage{hyperref,cleveref,autonum}

\usepackage{esint} 
\usepackage{graphicx,caption} 
\usepackage{float} 

\usepackage{pgf,tikz,pgfplots}
\usetikzlibrary{arrows}
\usetikzlibrary{intersections}
\usetikzlibrary{fadings}
\usepackage{wrapfig}

\usepackage[margin=0.9in]{geometry}
\newtheorem{theorem}{Theorem}[section]
\newtheorem*{theorem*}{Theorem}
\newtheorem{proposition}[theorem]{Proposition}
\newtheorem{lemma}[theorem]{Lemma}
\newtheorem{corollary}[theorem]{Corollary}
\newtheorem{conjecture}[theorem]{Conjecture}

\theoremstyle{definition}

\theoremstyle{remark}
\newtheorem{remark}[theorem]{Remark}
\newtheorem*{remark*}{Remark}

\newcommand{\con}[1]{\mathbb{#1}}
\newcommand{\C}{\con{C}} 
\newcommand{\R}{\con{R}} 
\newcommand{\Z}{\con{Z}} 
\newcommand{\N}{\con{N}} 
\newcommand{\Sph}{\con{S}} 

\newcommand{\ccal}{\mathscr{C}}

\newcommand{\ical}{\mathcal{I}}
\newcommand{\lcal}{\mathcal{L}}

\newcommand{\rcal}{\mathscr{R}}

\newcommand{\norm}[1]{\left \| {#1} \right \| }

\newcommand{\Dirac}{\mathcal{H}}

\newcommand\beqc[1]{\left\{\begin{array}{#1}}
	\newcommand\eeqc{\end{array} \right.}
\def\PDEsystem{rcll}
\def\bmatrix{\begin{pmatrix}}
	\def\ematrix{\end{pmatrix}}

\DeclareMathOperator{\sign}{sign}

\numberwithin{equation}{section}
\makeatletter
\def\l@subsection{\@tocline{2}{0pt}{2.5pc}{5pc}{}}
\makeatother

\DeclareSymbolFont{myletters}{OML}{ztmcm}{m}{it}
\DeclareMathSymbol{\uplambda}{\mathord}{myletters}{"15}

\DeclareRobustCommand{\SkipTocEntry}[5]{}

\title[Eigenvalue curves for generalized MIT bag models]
{Eigenvalue curves for generalized MIT bag models}

\author[N. Arrizabalaga]{Naiara Arrizabalaga}
\address{N. Arrizabalaga \newline
Departamento  de  Matem\'aticas,  Universidad  del  Pais  Vasco/Euskal Herriko Unibertsitatea,  Aptdo.   644,  48080 Bilbao, Spain}
\email{naiara.arrizabalaga@ehu.eus}

\author[A. Mas]{Albert Mas}
\address{A. Mas\textsuperscript{1,2}
	\newline
	\textsuperscript{1} Departament de Matem\`atiques,
	Universitat Polit\`ecnica de Catalunya.
	Campus Diagonal Bes\`os, Edifici A (EEBE), 
	Av. Eduard Maristany 16, 08019
	Barcelona, Spain.
	\newline \textsuperscript{2} Centre de Recerca Matem\`atica. Edifici C, Campus Bellaterra, 08193 Bellaterra, Spain.}
\email{albert.mas.blesa@upc.edu}

\author[T. Sanz-Perela]{Tom\'as Sanz-Perela}
\address{T. Sanz-Perela\textsuperscript{1,2}
	\newline
	\textsuperscript{1}  Basque Center for Applied Mathematics, Alameda Mazarredo 14, 48009 Bilbao, Spain
\newline \textsuperscript{2} Departamento de Matemáticas, Universidad Autónoma de Madrid, Ciudad Universitaria de Cantoblanco, 28049 Madrid, Spain}
\email{tomas.sanz@uam.es}

\author[L. Vega]{Luis Vega} 
\address{L. Vega\textsuperscript{1,2}
	\newline
	\textsuperscript{1} Basque Center for Applied Mathematics, Alameda Mazarredo 14, 48009 
	\newline \textsuperscript{2} Departamento  de  Matemáticas,  Universidad  del  Pais  Vasco/Euskal Herriko Unibertsitatea,  Aptdo.   644,  48080 Bilbao, Spain
	}
\email{lvega@bcamath.org}

\date{\today}
\subjclass[2010]{Primary: 35Q40; Secondary 35P05, 81Q10.}
\keywords{Dirac operator, spectral theory, MIT bag model, shape optimization.}

\thanks{All authors are supported by the  ERC-2014-ADG project HADE Id. 669689 (European Research Council).
	N.\!~A.\! is supported by the MINECO grant PGC2018-094522-B-I00 (Spain) and IT1247-19 (Gobierno Vasco).
	A.\!~M.\! is supported by grants MTM2017-84214-C2-1-P and RED2018-102650-T funded by MCIN/AEI/10.13039/501100011033 and by ``ERDF A way of making Europe'', by MINECO grant MTM2017-83499-P (Spain), and by the Spanish State Research Agency, through the Severo Ochoa and Mar\'ia de Maeztu Program for Centers and Units of Excellence in R\&D (CEX2020-001084-M).
	T.\!~S.-P.\! is supported by grants MTM2017-84214-C2-1-P and RED2018-102650-T funded by MCIN/AEI/10.13039/501100011033 and by ``ERDF A way of making Europe'', AGAUR research group 2017-SGR-1392 (Catalunya), and EPSRC grant EP/S03157X/1.
L.\!~V.\! is supported by the Basque Government through the BERC 2018-2021 program and by the Spanish State Research Agency through BCAM Severo Ochoa excellence accreditation SEV-2017-0718.
 }

\begin{document}
	
\begin{abstract}
We study spectral properties of Dirac operators on bounded domains  $\Omega \subset \mathbb{R}^3$ with boundary conditions of electrostatic and Lorentz scalar type and which depend on a parameter $\tau\in\R$; the case $\tau = 0$ corresponds to the MIT bag model.
We show that the eigenvalues are parametrized as increasing functions of  $\tau$, and we exploit this monotonicity to study the limits as $\tau \to \pm \infty$.
We prove that if $\Omega$ is not a ball then the first positive eigenvalue is greater than the one of a ball with the same volume for all $\tau$ large enough.
Moreover, we show that the first positive eigenvalue converges to the mass of the particle as $\tau \downarrow -\infty$, and we also analyze its first order asymptotics.
\end{abstract}

\maketitle
\tableofcontents


\section{Introduction}

Dirac operators acting on domains 
$\Omega\subset\R^d$ are used in relativistic quantum
mechanics to describe particles that are confined in a box. The so-called MIT bag model is a very remarkable example in dimension $d=3$. It was introduced in the
1970s as a simplified model to study confinement of quarks in hadrons \cite{bogoliubov1987,MIT061974,PhysRevD.12.2060,johnson}. The mathematical study of this and related three-dimensional models has gained attention in the recent years \cite{ALMR,ALR2017,ALR1,BehrndtHolzmannMas,Benhellal2019InfiniteMass,HolzmannZigZag,OurmieresBonafosVega,Rabinovich2020}. In dimension $d=2$, Dirac operators with special boundary conditions similar to the ones in the MIT bag
model are used in the description of graphene \cite{AB08,BM87,CGPNG09,gucclu2014graphene, MC_F_2004,PSKYHNG08}. They have been also investigated in the past few years from the mathematical point of view \cite{barbaroux2019resolvent,barbaroux2020dirac,BFSB17_1,cassano2019self,FreitasSiegl2014spectra,LT_OB_2018,PizzichilloVanDenBosch2019self,StockmeyerVugalter2018infinite}.
The present work focuses on the spectral study of a family 
$\{\Dirac_\tau\}_{\tau\in\R}$ of Dirac operators acting on bounded domains 
$\Omega\subset\R^3$ with
electrostatic plus Lorentz scalar type boundary conditions which depend on a parameter $\tau\in\R$; the particular case $\tau=0$ corresponds to the MIT bag model. 

\medskip

Throughout this article we assume that $\Omega\subset \R^3$ is a bounded domain with $C^2$ boundary. The unit normal vector field at $\partial\Omega$ which points outwards of $\Omega$ is denoted by $\nu$, and the surface measure on $\partial\Omega$, by $\upsigma$. Given $\tau\in\R$, let $\Dirac_\tau$ be the Dirac operator on $\Omega$ defined by
\begin{equation}
	\begin{split}
		\mathrm{Dom}(\Dirac_\tau) &:= \big\{ \varphi \in H^1(\Omega)\otimes\C^4: \, \varphi = i (\sinh\tau- \cosh\tau \, \beta)( \alpha \cdot\nu ) \varphi  \,\text{ on } \partial \Omega \big\},\\
		\Dirac_\tau\varphi &:= (-i \alpha \cdot \nabla + m \beta)\varphi \quad\text{for all 
			$\varphi\in\mathrm{Dom}(\Dirac_\tau),$}
	\end{split}
\end{equation}
where $-i \alpha \cdot \nabla + m \beta=:\Dirac$ denotes the differential expression which gives the action of the free Dirac operator on $\R^3$. More precisely,
$m\geq0$ denotes the mass of the particle, $\alpha := (\alpha_1, \alpha_2, \alpha_3)$, 
$$
\alpha_j:=\begin{pmatrix}0&\sigma_j\\\sigma_j&0\end{pmatrix} \mbox{ for }j=1,2,3, \quad  \text{and}\quad
\beta:=\begin{pmatrix}I_2&0\\0&-I_2\end{pmatrix}	
$$
are the $\C^{4\times4}$-valued Dirac matrices,
$I_d$ denotes the identity matrix in $\C^{d\times d}$ (it will also be denoted by $1$ when no confusion arises), and
$$
\sigma_1 := 
\begin{pmatrix}
	0&1\\1&0
\end{pmatrix},
\quad
\sigma_2 := 
\begin{pmatrix}
	0&-i\\i&0
\end{pmatrix},
\quad
\sigma_3 := 
\begin{pmatrix}
	1&0\\0&-1
\end{pmatrix}
$$
are the Pauli matrices. As customary, we use the notation $\alpha \cdot X:=\alpha_1 X_1+\alpha_2 X_2+\alpha_3 X_3$ for $X = (X_1,X_2,X_3)$, and analogously for $\sigma \cdot X$ with
$\sigma := (\sigma_1,\sigma_2,\sigma_3)$.

\medskip

The family $\{\Dirac_\tau\}_{\tau\in\R}$ naturally arises in the context of confining $\delta$-shell interactions. 
In the last decade, Dirac operators coupled with  $\delta$-shell potentials have been investigated from a mathematical perspective: their self-adjointness and spectral properties \cite{AMV1,AMV2,BehrndtEtAl2018spectral,BehrndtEtAl2019,BehrndtHolzmann2019dirac,BehrndtEtAl2020limiting,behrndt2020two,DittrichExnerSeba,Mas2017dirac,Rabinovich2020fredholm},  
the case of rough domains \cite{Benhellal2021rough,Benhellal2021spectral}, and their approximations and other asymptotic regimes \cite{CLMT2021,HolzmannEtAl2018dirac,MasPizzichillo-Klein,MasPizzichillo-Sphere,MoroianuEtAl2020dirac}; we refer to the survey \cite{OurmieresBonafosPizzichillo-Survey} for further details on the state of the art of shell interactions for Dirac operators. 
Several of these works addressed singular perturbations of the form
\begin{equation}\label{delta.shell.conf} 
	-i \alpha \cdot \nabla + m \beta+  (\lambda_e I_4+ \lambda_s \beta)\delta_{\partial\Omega},
\end{equation}
which correspond to the free Dirac operator on $\R^3$ coupled with electrostatic and Lorentz scalar $\delta$-shell potentials with strengths $\lambda_e$ and $\lambda_s$, respectively. Here, the $\delta$-shell distribution 
acts as $\delta_{\partial\Omega}(\varphi)= {\textstyle\frac{1}{2}}(\varphi_++\varphi_-)$, where $\varphi_{\pm}$ denotes the boundary values of $\varphi:\R^3\to\C^4$ when one approaches $\partial\Omega$ from inside/outside $\Omega$.
It is well known that the operator associated to \eqref{delta.shell.conf} decouples as the orthogonal sum of two operators, one acting in $L^2(\Omega)\otimes\C^4$ and the other in $L^2(\R^3\setminus\overline\Omega)\otimes\C^4$, if and only if $\lambda_e^2-\lambda_s^2=-4$. This has important consequences in the time-dependent scenario: the Hamiltonian \eqref{delta.shell.conf} generates confinement if and only if 
$\lambda_e^2-\lambda_s^2=-4$, meaning that a particle which is initially located inside/outside $\Omega$ will remain inside/outside $\Omega$ for all time. Under the confining relation $\lambda_e^2-\lambda_s^2=-4$, the boundary condition for the operator acting in $L^2(\Omega)\otimes\C^4$ is
\begin{equation}\label{MITgen.delta.shell}
	\varphi = \frac{i}{2} (\lambda_e - \lambda_s \beta) (\alpha \cdot\nu)\varphi\quad \text{on $\partial \Omega$;}
\end{equation}
recall that the MIT bag boundary condition corresponds to \eqref{MITgen.delta.shell} with
$\lambda_e=0$ and $\lambda_s=2$. Hence, if we set 
\begin{equation}\label{MITgen.delta.shell.2}
	\tau\mapsto\lambda_e(\tau):=2\sinh \tau\quad\text{and}\quad
	\tau\mapsto\lambda_s(\tau):=2\cosh \tau
	\quad\text{for $\tau\in\R$},
\end{equation} 
we obtain a parametrization of the whole branch of the hyperbola 
$\lambda_e^2-\lambda_s^2=-4$ that contains the MIT bag boundary condition, which is attained through the parametrization at $\tau=0$. Observe that the boundary condition used in the definition of $\mathrm{Dom}(\Dirac_\tau)$ is simply the combination of \eqref{MITgen.delta.shell} and \eqref{MITgen.delta.shell.2}. That is, from the singular perturbations point of view, 
$\{\Dirac_\tau\}_{\tau\in\R}$ is the restriction to $\Omega$ of the branch of 
confining electrostatic plus Lorentz scalar $\delta$-shell interactions  (of constant strength in $\partial \Omega$) that contains the MIT bag model.

In two dimensions, a parametrization similar to 
$\tau\mapsto\Dirac_\tau$ was used in \cite{BFSB17_1} to describe graphene quantum dots. In there, the boundary conditions are given by $\delta$-shell potentials of Lorentz scalar plus magnetic type. In our three-dimensional framework, the situation analogous to \cite{BFSB17_1} would be to couple the free Dirac operator with potentials of the form 
$(\lambda_s \beta+\lambda_a i(\alpha\cdot\nu)\beta)\delta_{\partial\Omega}$ with $\lambda_s^2+\lambda_a^2=4$, instead of the ones given by \eqref{delta.shell.conf}. Of course, both shell interactions agree for $\lambda_s=2$ (hence 
$\lambda_e=\lambda_a=0$), and they lead to the MIT bag operator. We refer to \cite[Sections 2.3 and 9]{CLMT2021} for more details. 

\medskip

The operator $\Dirac_\tau$ is self-adjoint. Its spectrum 
$\sigma(\Dirac_\tau)$ is purely discrete and is contained in
$\R\setminus[-m,m]$.
That is, for every $\tau\in\R$, $\sigma(\Dirac_\tau)$ is a sequence $\{\lambda_k^{\pm} (\tau)\}_{k\geq1}\subset\R$ such that
$$
\ldots \leq \lambda_2^-(\tau) \leq \lambda_1^-(\tau) < -m \leq m <  \lambda_1^+(\tau) \leq \lambda_2^+(\tau) \leq \ldots
$$
In addition, it holds that $\lambda$ is an eigenvalue of 
$\Dirac_\tau$ if and only if $-\lambda$ is an eigenvalue of $\Dirac_{-\tau}$ ---in particular $\sigma(\Dirac_0)$ is symmetric. All these properties are gathered in \Cref{Lemma:SpectrumGenMIT} below.
Our main goal in this work is to describe the {\em eigenvalue curves} of the family $\{\Dirac_\tau\}_{\tau\in\R}$, that is, the mappings 
$\tau\mapsto\lambda_k^{\pm}(\tau)$, $k\geq1$, 
as $\tau$ ranges over all $\R$.
We pay special attention to the study of the {\em first eigenvalues}, and by this we mean the ones whose absolute value is closest to $m$, i.e., 
$\lambda_1^-(\tau)$ and $\lambda_1^+(\tau)$. 
By the odd symmetry with respect to the parameter $\tau$ mentioned above, it is enough to study $\lambda_1^+(\tau)$, which will be called in the sequel the {\em first positive eigenvalue} of $\Dirac_\tau$.

In \Cref{Subsec:Sphere} we investigate how the eigenvalue curves look like when $\Omega$ is a ball, a case where explicit formulas involving Bessel functions are available. This is our starting point for the spectral analysis in the general case: the evidences observed on the ball provide us with clues of what could be expected to hold on any domain. Our main results for general domains are described in \Cref{Subsec:MainResults}.

\medskip

Before getting into more details, a few words on our motivation for the results presented in this work are in order. From a general perspective, there is a large body of literature on the spectral analysis of differential operators on domains with parameter-dependent  boundary conditions. The Robin Laplacian is a very remarkable example; the interested reader may look at \cite{BFK_2017} and the references therein. 
However, as far as we know, the type of perturbative analysis carried out in the present work has not been considered so far in the context of shell interactions for Dirac operators, except for \cite{AMV3}. In there, the monotonicity of the eigenvalues with respect to the parameter that defines the electrostatic $\delta$-shell interaction is used to optimize the threshold of admissible strengths that yield nontrivial point spectrum, and to characterize the optimal domains. Roughly speaking, a property of the parameter-dependent family of operators (the monotonicity) is successfully used to solve a shape optimization problem for a certain quantity of spectral nature (the threshold of strengths).

One may consider the Dirac operators on domains as the relativistic counterpart of Laplacians with boundary conditions, as for example of Robin type. In this way, the study carried out in the present work has its own interest from the point of view of perturbation theory. However, the main motivation that originated this article was to address the shape optimization problem for the spectral gap of the MIT bag operator, which consists of minimizing the first squared eigenvalue of $\Dirac_0$ among all domains with prescribed volume. The analogous question in the two-dimensional framework ---the optimization of the spectral gap for Dirac operators with infinite mass boundary conditions--- is considered a hot open problem in spectral geometry~\cite{ProblemListShapeOptimization}. More generally, the quest of geometrical upper and lower bounds for the spectral gap is a trending topic of research \cite{AntunesEtAl,BFSB17_2, BrietKrejcirik2021spectral,LotoreichikO-B2019sharp}. This quest is also addressed in the differential geometry literature for Dirac operators on spin manifolds, where sharp inequalities for spectral gaps in terms of geometric quantities are shown \cite{AgricolaFriedrich1999upper,AmmannBar2002dirac,Bar1992lower, Bar1998extrinsic,ChamiGinouxHabib2019Bessel,Hijazi1986conformal,HijaziMontielRoldan2002eigenvalue,HijaziMontielZhang2001eigenvalues,KramerSemmelmannWeingart1998first}. 
Despite the amount of works available on this topic, for the case of bounded domains in euclidean spaces the problem of minimizing the spectral gap under a volume constraint (and with no further restrictions on the geometry of the boundary) remains open.

Following the line of our comments about \cite{AMV3}, in order to address the shape optimization problem for the MIT bag operator it could be useful to take benefit from the inclusion of $\Dirac_0$ in the family $\{\Dirac_\tau\}_{\tau\in\R}$, and to exploit the connection of $\Dirac_\tau$ with the Dirichlet Laplacian as $\tau\to\pm\infty$; see \Cref{Th:EigLimits} below. In this regard, \Cref{Th:BallOptimalLargeTau} shows the optimality of the ball in the asymptotic regime $\tau\uparrow+\infty$. In addition, \Cref{Th:Rayleigh.Intro} draws a path to address the optimality of the ball in the asymptotic regime $\tau\downarrow-\infty$. As we said, we expect that this information for $|\tau|$ large enough will be useful to deal with the optimality of the ball in the general case of $\tau\in\R$ and, in particular, for $\tau=0$.

\medskip

We close this introductory section clarifying some conventions that we are going to use in the sequel. Besides an amount of standard notation, as well as further shorthand that will be introduced in due time, we use the following notation throughout the paper:

\medskip

\begin{center}
	\begin{tabular}{l c l}
		\hline\hline
		\rule{0pt}{3ex}
		\!\!$E^d$ && tensor product $E\otimes\C^d$ of a vector space $E$ with $\C^d$\\
		$A^\intercal$ && transpose of a matrix $A$\\
		$\varphi= (u, v)^\intercal$ && decomposition of $\varphi\in\C^4$ in upper and lower components, that is,\\
		&& \qquad  $u=(\varphi_1,\varphi_2)^\intercal$ and $v=(\varphi_3,\varphi_4)^\intercal$ for
		$\varphi=(\varphi_1,\varphi_2,\varphi_3,\varphi_4)^\intercal$\\
		$\langle\cdot,\cdot\rangle_{L^2(\Omega)^d}$ &&
		scalar product in $L^2(\Omega)^d$ given by
		$\langle u,v\rangle_{L^2(\Omega)^d}:=\int_{\Omega} u\cdot\overline v\,dx$\\
		$\langle\cdot,\cdot\rangle_{L^2(\partial\Omega)^d}$ &&
		scalar product in $L^2(\partial\Omega)^d$ given by
		$\langle u,v\rangle_{L^2(\partial\Omega)^d}:=\int_{\partial\Omega} u\cdot\overline v\,d\upsigma$\\
		$H^1(\Omega)$ && Sobolev space of functions in $L^2(\Omega)$ with first weak partial\\
		&& \qquad derivatives in $L^2(\Omega)$\\
		$\sigma$ && $3$-tuple $(\sigma_1, \sigma_2, \sigma_3)$ of Pauli matrices\\
		$\sigma(T)$ && spectrum of the operator $T$\\
		$\{T_1,T_2\}$ && anticommutator $T_1T_2+T_2T_1$ of operators $T_1$ and $T_2$\\
		[1ex]\hline\hline
	\end{tabular}
\end{center}

\bigskip

\subsection{Eigenvalue curves for a ball} \label{Subsec:Sphere}

We present here a brief summary of the spectral study of $\Dirac_\tau$ in the case $\Omega = B_R\subset\R^3$, the ball of radius $R>0$ centered at the origin.
In this radially symmetric domain we can use separation of variables ($r \in [0,R)$ and $\theta \in \Sph^2$) and, thanks to the spherical harmonic spinors, obtain explicit equations for the eigenvalues and eigenfunctions of $\Dirac_\tau$.
The analysis done for the ball provides some intuition on which kind of situations one can expect when studying the operator in a general domain $\Omega \subset \R^3$, for which no explicit formulas are available.
A more detailed analysis including the proofs of the facts stated in this section can be found in the Appendix~\ref{Sec:ExplicitComputationsSphere}.

To deal with the problem in the ball, it is suitable to use the  decomposition
$$
L^2(\R^3)^4 = \bigoplus_{j=1/2}^{+\infty} \, \bigoplus_{\mu_j=-j}^j L_{j, \mu_j}^+ \oplus  L_{j, \mu_j}^-, 
$$
where  $L_{j, \mu_j}^\pm $ for $j = 1/2, \, 3/2,\, \ldots\, $, and $\mu_j = -j, \, -j + 1, \, \ldots\, ,\, j-1,\, j$, are invariant spaces under the action of $\Dirac$, and are defined in terms of the spherical harmonic spinors; see Appendix~\ref{Sec:ExplicitComputationsSphere} for the explicit expressions.
Thanks to the above decomposition, the eigenvalue problem for $\Dirac_\tau$ can be reduced to a system of two Bessel-type ODE.
After imposing the boundary conditions, for each invariant space $L_{j, \mu_j}^\pm $ we obtain the eigenvalue equation 
\begin{equation}
	\label{Eq:IntroEigEquationSphere}
	0 = e^\tau J_{\ell + 1/2} ( \sqrt{\lambda^2-m^2} R ) \pm \dfrac{\sqrt{\lambda^2-m^2}}{\lambda+m} J_{\ell' + 1/2} ( \sqrt{\lambda^2-m^2} R ),
\end{equation}
where $J_k$ is the $k$-th Bessel function of the first kind,  $\ell = j \pm 1/2$, and $\ell' = j \mp 1/2$.
This equation already appears in \cite[formula~(6.3)]{DittrichExnerSeba}.
Note that the index $\mu_j$ does not appear in \eqref{Eq:IntroEigEquationSphere}, meaning that for a given half-integer $j$ there are $2j + 1$ linearly independent eigenfunctions associated to the same eigenvalue.
Therefore, all the eigenvalues have even multiplicity.

If we solve numerically \eqref{Eq:IntroEigEquationSphere} for some choice of indexes, we obtain the plot shown in \Cref{Fig:ALL_eigenvalues}.
\begin{figure}[h]
	\centering
	\vskip-25pt
	\includegraphics[width=0.7\textwidth]{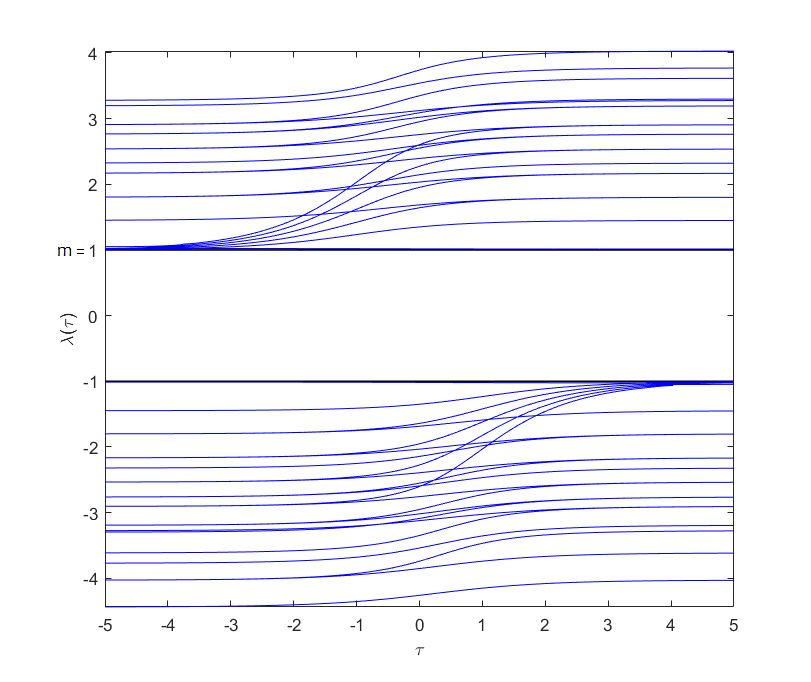}
	\vskip-20pt
	\caption{Some eigenvalue curves $\tau\mapsto\lambda(\tau)$ of $\Dirac_\tau$ on $B_R$ for $R=3$ and $m=1$.}
	\label{Fig:ALL_eigenvalues}
\end{figure}
\newline
As we see, the eigenvalues can be parametrized in terms of $\tau$, obtaining a family of increasing functions.
To show rigorously that this parametrization can be done, it suffices to write the eigenvalue equation \eqref{Eq:IntroEigEquationSphere}  as
\begin{equation}
	\label{Eq:IntroEigEquationSphereParam}
	e^\tau =  \mp \dfrac{\sqrt{\lambda^2-m^2}}{\lambda+m} \dfrac{ J_{\ell' + 1/2} ( \sqrt{\lambda^2-m^2} R ) } {J_{\ell + 1/2} ( \sqrt{\lambda^2-m^2} R ) } =: h(\lambda)
\end{equation}
and invert the function $h$ in suitable intervals $I$.
This provides, for each of these intervals, a parametrization of an eigenvalue given by
$\tau \mapsto \lambda (\tau) = h^{-1} (e^\tau) \in I$.
By the distribution of the zeroes and singularities of $h$, it can be seen that $h$ is only invertible in intervals $I$ in which the above quotient of Bessel functions does not change sign. These maximal intervals have the form  $( -  z_{\ell'}, - z_\ell)$, $(- z_{\ell'}, -m)$, $(m,  z_\ell)$, or $( z_{\ell'},  z_\ell)$, where $ z_k$ denotes a positive zero of the function $J_{k + 1/2} ( \sqrt{(\cdot)^2-m^2} R )$.
In each of these intervals, $h: I \to (0,+\infty)$ is strictly increasing and surjective, and thus it can be inverted.
From this it follows that $\lambda(\tau)$ is also strictly increasing.
\Cref{Prop:ParamEgienvaluesSphere} gathers all these considerations.

In addition, the limit of $|\lambda(\tau)|$ as $\tau \to \pm \infty$ must be at the boundary of the interval $I$, and thus be either $m$ or a positive zero of $J_{k + 1/2} ( \sqrt{(\cdot)^2-m^2} R )$ for some $k=0,1,2,\ldots$; note that each of these zeroes corresponds to the square root of a Dirichlet eigenvalue of $-\Delta + m^2$ in $B_R$.
Since the same Bessel function may appear in the denominator of the expression of $h$ in \eqref{Eq:IntroEigEquationSphereParam} for different choices of the index $j$, more than one eigenvalue curve may converge to the same value as $\tau \uparrow +\infty$, and analogously as $\tau \downarrow-\infty$.
This is illustrated in \Cref{Fig:ALL_eigenvalues} and, more explicitly, in \Cref{Fig:j32BOTH_eigenvalues,Fig:Crossings_eigenvalues}.

\begin{remark}
	\label{Rem:SymmetrySphere}
	Note that \eqref{Eq:IntroEigEquationSphereParam} can be rewritten as	
	\begin{equation}
		\label{Eq:IntroEigEquationSphereParamBis}
		e^\tau =  \mp \sign (\lambda+m ) \sqrt{\dfrac{\lambda-m}{\lambda+m}} \dfrac{ J_{\ell' + 1/2} ( \sqrt{\lambda^2-m^2} R ) } {J_{\ell + 1/2} ( \sqrt{\lambda^2-m^2} R ) } =: h(\lambda).
	\end{equation}
	From this, it follows that $\lambda(\tau)$ is an eigenvalue curve corresponding to the subspace $L^\pm_{j, \mu_j}$ if and only if $-\lambda(-\tau)$ is an eigenvalue curve associated to  $L^\mp_{j, \mu_j}$. 
	This fact is illustrated in \Cref{Fig:j12_eigenvalues,Fig:j12BOTH_eigenvalues}.
	As we will see, the odd symmetry of the eigenvalue curves with respect to $\tau$ holds as well for a general domain $\Omega$; see \Cref{Lemma:SpectrumGenMIT}~$(iii)$.
	Therefore, it will be enough to study the positive spectrum of $\Dirac_\tau$.
\end{remark}

\hskip -15pt
\begin{minipage}[t]{0.5\textwidth}
	\centering
	\includegraphics[width=0.95\textwidth]{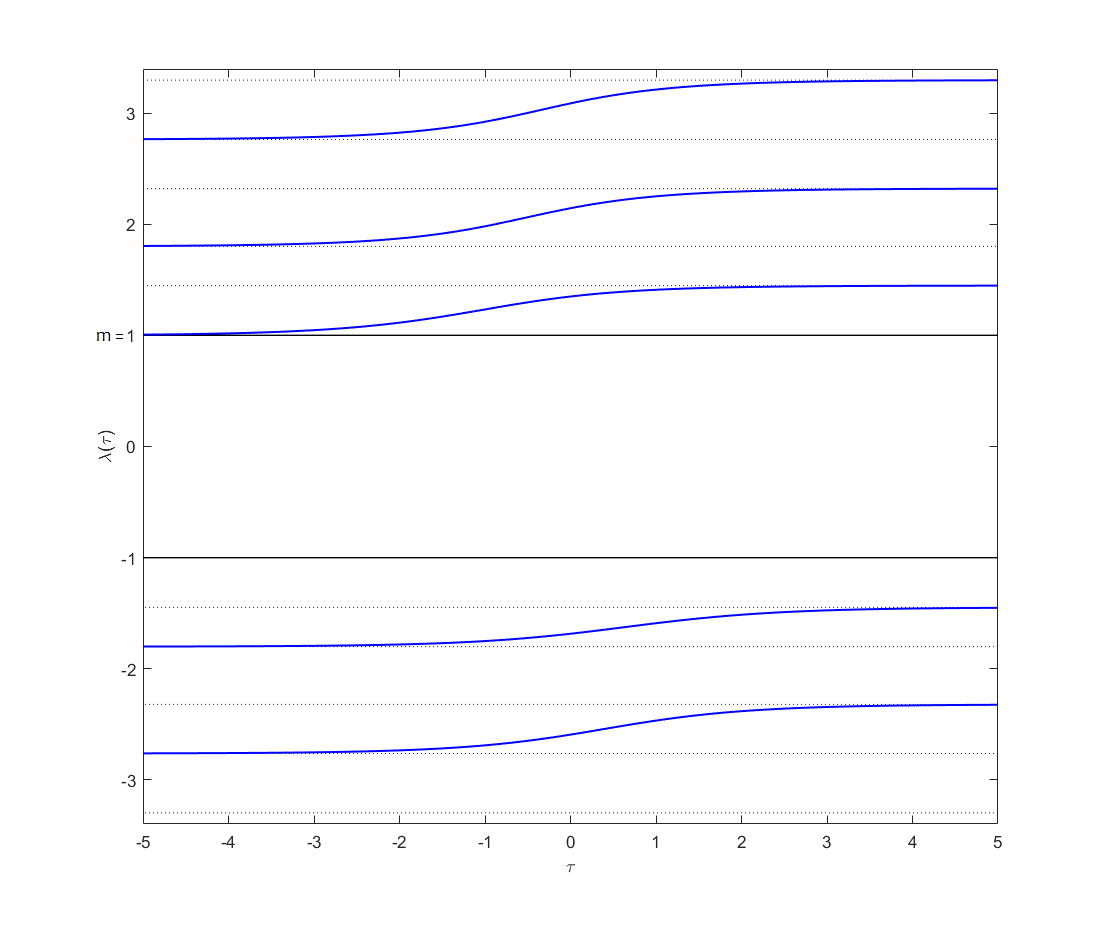}
	\vskip-20pt
	\captionof{figure}{Eigenvalue curves for $L^-_{1/2, \mu_{1/2}}$.}
	\label{Fig:j12_eigenvalues}
\end{minipage}
\begin{minipage}[t]{0.5\textwidth}
	\centering
	\includegraphics[width=0.95\textwidth]{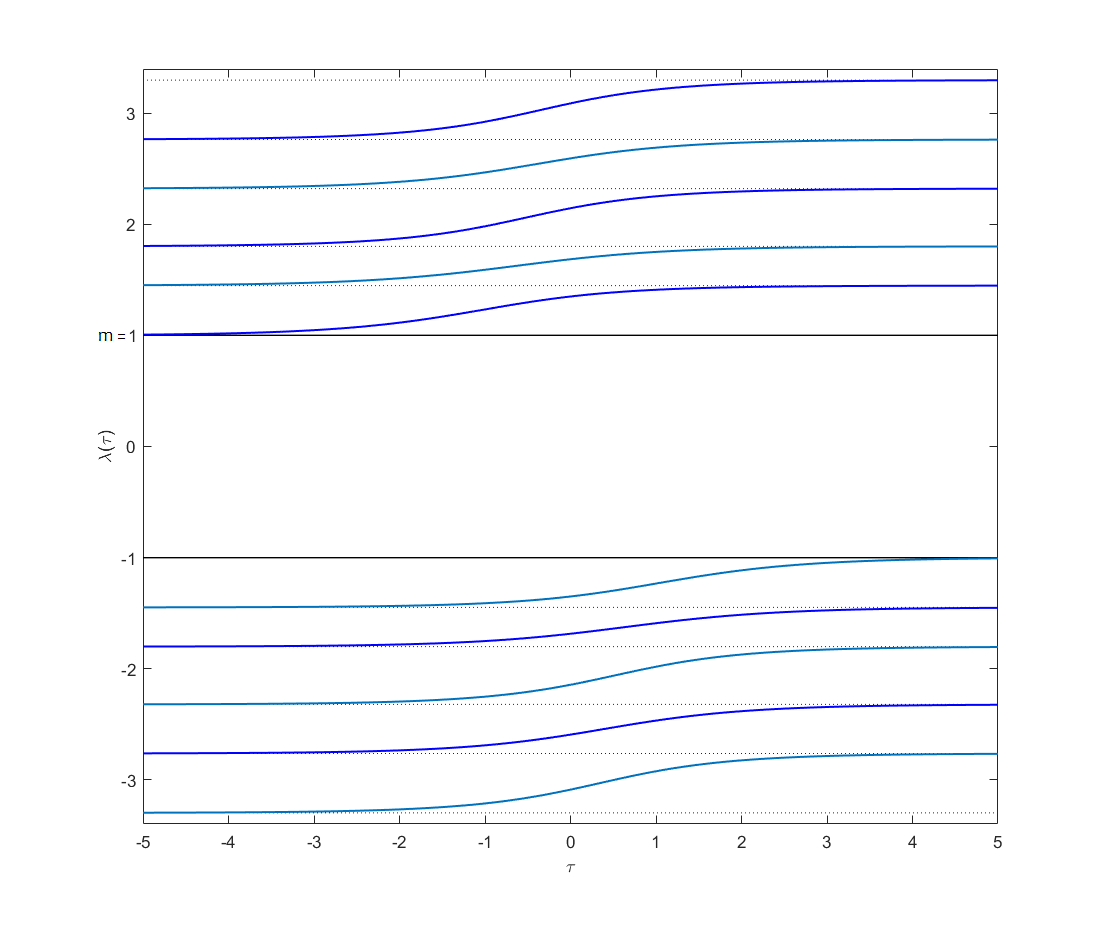}
	\vskip-20pt
	\captionof{figure}{Eigenvalue curves for $L^-_{1/2, \mu_{1/2}}$ and $L^+_{1/2, \mu_{1/2}}$.}
	\label{Fig:j12BOTH_eigenvalues}
\end{minipage}

\hskip -15pt
\begin{minipage}[t]{0.5\textwidth}
	\centering
	\includegraphics[width=0.95\textwidth]{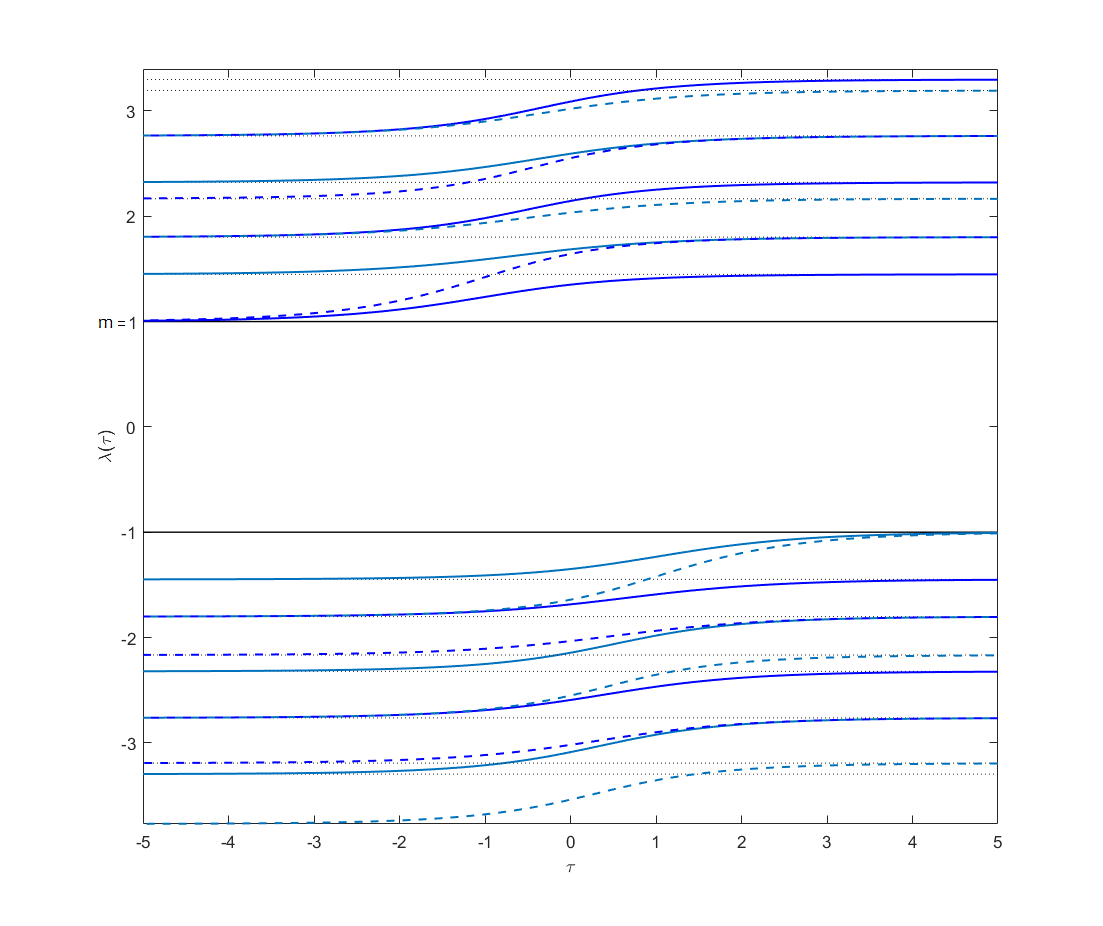}
	\vskip-20pt
	\captionof{figure}{Eigenvalue curves for $j=1/2,3/2$.}
	\label{Fig:j32BOTH_eigenvalues}
\end{minipage}
\begin{minipage}[t]{0.5\textwidth}
	\centering
	\includegraphics[width=0.95\textwidth]{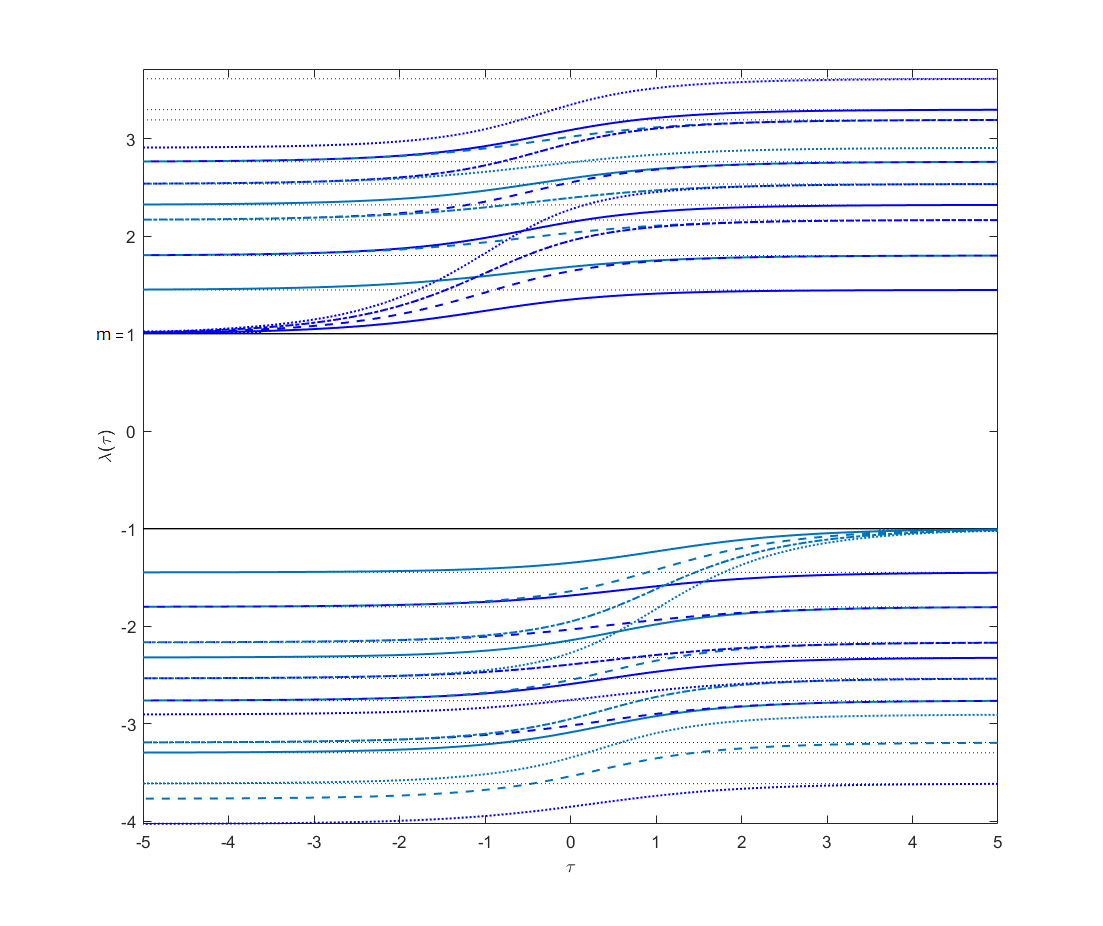}
	\vskip-20pt
	\captionof{figure}{Eigenvalue curves for $j=1/2,3/2,5/2,7/2$.}
	\label{Fig:Crossings_eigenvalues}
\end{minipage}

\bigskip

To summarize, all the eigenvalues of $\Dirac_{\tau}$ in $B_R$ can be represented as a set of monotone increasing curves parametrized by $\tau \in \R$, which may cross among them; see \Cref{Fig:Crossings_eigenvalues}. 
For a given curve $\tau \mapsto \lambda(\tau) \in \sigma(\Dirac_{\tau}) \cap (m,+\infty)$, we know which are the possible limits of $\lambda(\tau)$ as $\tau \to \pm \infty$: with the only exception of some eigenvalues that, as $\tau \downarrow -\infty$, converge to $m$, the limiting values of $\lambda(\tau)$ are of the form $\sqrt{\Lambda + m^2}$, where $\Lambda$ is an eigenvalue of the Dirichlet Laplacian in $B_R$.
In the next sections we will show that all these properties hold for every bounded domain $\Omega \subset \R^3$ with $C^2$ boundary; see \Cref{Th:ParamEigenvalues,Th:EigLimits}.
In addition, we will use the explicit information available for the ball to study a shape optimization problem for the first positive eigenvalue of $\Dirac_{\tau}$ when $|\tau|$ is large enough.

\subsection{Main results}\label{Subsec:MainResults}
Let $\Omega\subset \R^3$ be a bounded domain with $C^2$ boundary. In this section we state our main results on the spectral study of the operator $\Dirac_\tau$ in terms of the parameter $\tau\in\R$.
The study is carried out taking into account the different phenomena observed in \Cref{Fig:ALL_eigenvalues}, where the case of a ball is illustrated.

Recall that
\begin{equation}
	\label{Eq:DefGenMIT}
	\begin{split}
		\mathrm{Dom}(\Dirac_\tau) &:= \big\{ \varphi \in H^1(\Omega)^4: \, \varphi = i (\sinh\tau- \cosh\tau \, \beta)( \alpha \cdot\nu ) \varphi  \,\text{ on } \partial \Omega \big\},\\
		\Dirac_\tau\varphi &:= \Dirac\varphi 
		\quad\text{for all 
			$\varphi\in\mathrm{Dom}(\Dirac_\tau)$,}
	\end{split}
\end{equation}
and $\Dirac:=-i \alpha \cdot \nabla + m \beta$.
The fact that $\Dirac_\tau$ is self-adjoint follows by \cite[Proposition 5.15]{BehrndtHolzmannMas}.  
Let us begin by recalling several known properties of the spectrum of $\Dirac_\tau$. We gather them in the following lemma, whose proof is given in \Cref{Sec:Formulas} for the sake of completeness.

\begin{lemma}
	\label{Lemma:SpectrumGenMIT}
	For every $\tau\in\R$, the following holds:
	\begin{itemize}
		\item[$(i)$] The spectrum $\sigma(\Dirac_\tau)$ is a sequence of real eigenvalues that only accumulate at 
		$\pm\infty$.
		\item[$(ii)$] 
		$\sigma(\Dirac_\tau) \subset (-\infty,-m)\cup(m,+\infty)$.
		\item[$(iii)$] $\lambda\in\sigma(\Dirac_\tau)$ if and only if 
		$-\lambda\in\sigma(\Dirac_{-\tau})$. In particular, 
		$\sigma(\Dirac_0)$ is symmetric.
		\item[$(iv)$] Every $\lambda\in\sigma(\Dirac_\tau)$ has finite and even multiplicity.		
	\end{itemize}
\end{lemma}

As we see from \Cref{Lemma:SpectrumGenMIT}~$(i)$, for every $\tau\in\R$ the spectrum $\sigma(\Dirac_\tau)$ only contains real eigenvalues with no accumulation points. It is then natural to investigate how the eigenvalues depend on the parameter $\tau$. Our first result in this direction is that they can be parametrized by means of strictly increasing real analytic functions of $\tau$ whose graphs may cross at locally finitely many exceptional values of $\tau$. Note that these crossings indeed occur in the case of the ball, as shown in 
\Cref{Fig:ALL_eigenvalues,Fig:Crossings_eigenvalues}.
The following proposition gives the precise statement about the existence of the mentioned parametrization; see \Cref{Sec:Parametrization+Monotonicity} for a proof.

\begin{theorem}
	\label{Th:ParamEigenvalues}
	Given $\tau_0\in \R$, let $\{\lambda_k (\tau_0) \}_{k \in \Z \setminus \{0\}}$ be an enumeration of the eigenvalues of $\Dirac_{\tau_0}$ (each repeated according to its finite algebraic multiplicity).
	Then, each $\lambda_k (\tau_0)$ can be extended to a real analytic function $\tau\mapsto\lambda_k(\tau)\in\R$ in such a way that
	\begin{equation}
		\sigma(\Dirac_{\tau}) = {\textstyle \bigcup_{k \in \Z \setminus \{0\}}} \{\lambda_k(\tau)\}\quad\text{for all }\tau\in\R.
	\end{equation}
	All the functions $\tau\mapsto\lambda_k (\tau)$ are strictly increasing in $\R$ and, up to repetition (due to the even multiplicity), for each eigenvalue curve $\lambda_k(\tau)$ there are only finitely many eigenvalue curves meeting at locally finitely many crossing points with $\lambda_k(\tau)$.
	
	Furthermore, for each $k\in \Z\setminus \{0\}$, there exists an analytic function  $\tau \mapsto \varphi_k(\tau)\in L^2(\Omega)^4$ such that, for every $\tau\in \R$, the family $\{\varphi_k(\tau)\}_{k \in \Z \setminus \{0\}}$ is a basis of $L^2(\Omega)^4$ and  each $\varphi_k(\tau)$ is an eigenfunction associated to the eigenvalue $\lambda_k(\tau)$ ---that is, $\varphi_k(\tau)\in \mathrm{Dom}(\Dirac_{\tau})\setminus \{0\}$ and $\Dirac_{\tau}\varphi_k(\tau)=\lambda_k (\tau)\varphi_k(\tau)$.

\end{theorem}

In the statement of this theorem regarding crossing points we mean that, modulo eigenvalue curves which are exactly the same for all $\tau \in \R$, in every compact set each eigenvalue curve can cross other curves only a finite number of times, and in each crossing it will coincide only with a finite number of eigenvalue curves.
It is also worth pointing out that since the eigenvalue curves are continuous, $\sigma(\Dirac_0)$ is symmetric, and $\sigma(\Dirac_\tau) \subset (-\infty,-m)\cup(m,+\infty)$ for all $\tau$, then it holds that $\sigma(\Dirac_\tau) \cap(m,+\infty)\neq\emptyset$ and $\sigma(\Dirac_\tau) \cap(-\infty,-m)\neq\emptyset$ for all $\tau\in\R$.

The next step in our analysis is to address the asymptotic behavior of the eigenvalue curves as 
$\tau\to\pm\infty$. Thanks to the odd symmetry of the eigenvalues with respect to the parameter $\tau$ shown in \Cref{Lemma:SpectrumGenMIT}~$(iii)$, it is enough to consider only positive eigenvalues, i.e., the elements of $\sigma(\Dirac_\tau)\cap(m,+\infty)$. 
In the following result we describe their asymptotic behavior in terms of $\sigma(-\Delta_D)$, the spectrum of the Dirichlet Laplacian $-\Delta_D$.
The proof is given in \Cref{Sec:AsymptoricBehavior}.

\begin{theorem}
	\label{Th:EigLimits}
	Let $\tau\mapsto\lambda(\tau)\in\sigma(\Dirac_\tau)\cap(m,+\infty)$ be a continuous function defined on an interval $I\subset\R$. The following holds:
	\begin{enumerate}[label=$(\roman*)$]
		\item If $I=(-\infty,\tau_0)$ for some $\tau_0\in\R$, then 
		$\lambda(-\infty):=\lim_{\tau\downarrow-\infty}\lambda(\tau)$ exists and belongs to $[m,+\infty)$. In addition,
		$\lambda(-\infty)=m$ if 
		$\lambda(\tau)^2-m^2\leq\min\sigma(-\Delta_D)$ for some 
		$\tau\in I$, and 
		$\lambda(-\infty)^2-m^2\in \sigma(-\Delta_D)$ otherwise.
		\item If $I=(\tau_0,+\infty)$ for some $\tau_0\in\R$, then 
		$\lambda(+\infty):=\lim_{\tau\uparrow+\infty}\lambda(\tau)$ exists as an element of the set $(m,+\infty]$. In addition, if $\lambda(+\infty)<+\infty$ then 
		$\lambda(+\infty)^2-m^2\in \sigma(-\Delta_D)$.
	\end{enumerate}
\end{theorem}

The function $\lambda_1^+:\R\to(m,+\infty)$ defined by 
\begin{equation}\label{defi:ParamFirstEig}
	\tau \mapsto  \lambda_1^+(\tau):=\min(\sigma(\Dirac_\tau)\cap(m,+\infty))
\end{equation}
assigns to every $\tau\in\R$ the first (smallest) positive eigenvalue of $\Dirac_\tau$, as \Cref{Lemma:SpectrumGenMIT}~$(ii)$ shows. Combining \Cref{Th:ParamEigenvalues,Th:EigLimits}, we deduce several properties of the function $\lambda_1^+$ which are gathered in the following result; see \Cref{Subsec:EigLimits} for a proof. Analogous conclusions hold for the mapping
$\tau \mapsto  \lambda_1^-(\tau):=\max(\sigma(\Dirac_\tau)\cap(-\infty,-m))$,
which assigns to every $\tau\in\R$ the first (largest) negative eigenvalue of $\Dirac_\tau$. This is because
$\lambda_1^-(\tau)=-\lambda_1^+(-\tau)$ by \Cref{Lemma:SpectrumGenMIT}~$(iii)$.

\begin{theorem}
	\label{Th:FirstEigTom}
	The function $\lambda_1^+$ defined by \eqref{defi:ParamFirstEig} is continuous and strictly increasing in $\R$, and satisfies
	\begin{equation}\label{lim:pm.infty.lambda1+.state}	
		\lim_{\tau\downarrow-\infty}\lambda_1^+(\tau)=m
		\quad\text{and}\quad
		\lim_{\tau\uparrow+\infty}\lambda_1^+(\tau)^2-m^2\in \sigma(-\Delta_D)\cup\{+\infty\}.
	\end{equation}
	In addition, $\lambda_1^+$ is real analytic in 
	$\R\setminus E$, where $E\subset\R$ is some set such that $E\cap [-R,R]$ is finite for all $R>0$.
\end{theorem}

In \Cref{Remark:FirstEigFinite} we make a comment on  the possibility of having $\lim_{\tau\uparrow+\infty}\lambda_1^+(\tau) = +\infty$.
See also \Cref{Remark:SetE} for a comment on the set $E$. 

\Cref{Th:FirstEigTom} has some consequences on a shape optimization problem for the first positive eigenvalue of $\Dirac_\tau$ when 
$\tau$ is large enough. In order to highlight the dependence of $\Dirac_\tau$ on the domain $\Omega\subset\R^3$ ---which we assume through the paper that has a $C^2$ boundary---, let us now denote the first positive eigenvalue of $\Dirac_\tau$ by $\lambda_\Omega(\tau)$  (that is, we set 
$\lambda_\Omega:=\lambda_1^+$). Then, using \eqref{lim:pm.infty.lambda1+.state} and Faber-Krahn inequality, we get the following result, whose proof is given in \Cref{Subsec:EigLimits}.

\begin{corollary}
	\label{Th:BallOptimalLargeTau}
	Let $\Omega\subset \R^3$ be a bounded domain with $C^2$ boundary, and let $B\subset\R^3$ be a ball such that 
	$|\Omega| = |B|$. If $\Omega$ is not a ball,
	then there exists $\tau_0\in \R$ depending on 
	$\Omega$ such that 
	$\lambda_B(\tau) < \lambda_\Omega(\tau)$ for all $\tau \geq \tau_0$.
\end{corollary}

This means that, if $\Omega$ is not a ball, for large values of $\tau$ the first positive eigenvalue of 
$\Dirac_\tau$ on $\Omega$ is strictly larger than the one on a ball with the same volume as $\Omega$. Then, a natural question is whether the analogous result for $\tau$ tending to $-\infty$ holds or not. Recall that the first positive eigenvalue of $\Dirac_\tau$ tends to $m$ as $\tau\downarrow-\infty$ independently of the shape of $\Omega$, by \eqref{lim:pm.infty.lambda1+.state}. Hence, in order to find the optimal shape of $\Omega$ for $\tau$ negative and far from the origin,
one is forced to study the asymptotic expansion of $\lambda_\Omega(\tau)-m$ as $\tau\downarrow-\infty$. Our next result goes in this direction, but requires some preliminaries to state it properly.

Let $L_\Omega:\R\to(0,+\infty)$ be defined by
\begin{equation}\label{def.L.Omega}
	\tau\mapsto L_\Omega(\tau):= (\lambda_\Omega(\tau)-m)e^{-\tau},
\end{equation} 
thus
$\lambda_\Omega(\tau)=m+e^{\tau}L_\Omega(\tau)$. In 
\Cref{Lemma:LDecreasing} we will prove that $L_\Omega$ is strictly decreasing in 
$\R$. Therefore, the limit 
$$L_\Omega^\star:=\lim_{\tau\downarrow-\infty}L_\Omega(\tau)$$
exists as an element of the set $(0,+\infty]$. Moreover, in \Cref{Prop:ParamFirstEigSphere} we will also show that $L_B^\star<+\infty$ for every ball $B\subset\R^3$. This in particular suggests that $L_\Omega^\star$ is the natural quantity to look at when addressing the asymptotic expansion of $\lambda_\Omega(\tau)-m$ as $\tau\downarrow-\infty$.

Assume now for a while that $\Omega\subset\R^3$ is a domain such that $L_\Omega^\star=+\infty$, and let $B\subset\R^3$ be a ball, not necessarily with the same volume as $\Omega$. Since $L_B^\star<+\infty$, there exists $\tau_0\in\R$ such that $L_\Omega(\tau)>L_B^\star\geq L_B(\tau)$ for all 
$\tau\leq\tau_0$. Therefore,
\begin{equation}\label{ball.best.-infty}
	\lambda_\Omega(\tau)=m+e^{\tau}L_\Omega(\tau)
	>m+e^{\tau}L_B(\tau)=\lambda_B(\tau)
\end{equation}
for all $\tau\leq\tau_0$. That is, if $L_\Omega^\star=+\infty$,  then the first positive eigenvalue of 
$\Dirac_\tau$ on $\Omega$ is strictly greater than the one on any ball $B$ whenever $-\tau$ is large enough depending on $\Omega$ and $B$. As a consequence, in order to look for the optimal shape of $\Omega$ (under a volume constraint) for $\tau$ negative and far from the origin, we can assume without loss of generality that $L_\Omega^\star<+\infty$.

Our last result in this work provides a lower bound for $L_\Omega^\star$ in terms of the optimization,  among functions in a boundary Hardy space, of a Rayleigh quotient which involves the single layer potential for the Laplacian. We expect that this lower bound, which is attained if $\Omega$ is a ball, will  be useful to solve the above-mentioned optimization problem for $\tau$ negative and far from the origin. 
On the one hand, the single layer potential appears in trace form as the operator $K:L^2(\partial\Omega)^2\to L^2(\partial\Omega)^2$ defined by
$$K u(x):=\frac{1}{4\pi} \int_{\partial \Omega}\frac{u(y)}{|x-y|} \, d\upsigma(y)\quad\text{for $\upsigma$-a.e.\! $x\in\partial\Omega$}.$$
It is well known that $K$
is a bounded, self-adjoint, compact, positive, and injective operator in $L^2(\partial\Omega)^2$. In particular, $K$ diagonalizes in an $L^2(\partial\Omega)^2$-orthonormal basis of eigenfunctions and all its eigenvalues are strictly positive real numbers. On the other hand, the boundary Hardy space referred above is the subspace $P(L^2(\partial\Omega)^2)$ of $L^2(\partial\Omega)^2$, where  
$P:L^2(\partial\Omega)^2\to L^2(\partial\Omega)^2$ 
is defined by
$P:=\frac{1}{2}+ iW(\sigma\cdot\nu),$
and
\begin{equation}
	\begin{split}
		W u(x) := \lim_{\epsilon \downarrow 0} \frac{i}{4\pi }
		\int_{ \{ y \in \partial \Omega :\, |x-y|> \epsilon\}}    
		\Big({\sigma} \cdot \frac{x-y}{|x-y|^{3}}\Big) u(y)\, d \upsigma(y)
		\quad\text{for $\upsigma$-a.e.\! $x\in\partial\Omega$}.
	\end{split}
\end{equation}
It is well known that $W$ is a bounded operator in $L^2(\partial\Omega)^2$, and that $P$ is a projection. The subspace $P(L^2(\partial\Omega)^2)$ arises as the trace space on 
$\partial\Omega$ of null-solutions of 
$\sigma\cdot\nabla$ in $\Omega$; see \Cref{Subsec:Proj.Hardy} for more details (by a matter of notation, in Sections \ref{Subsec:Proj.Hardy} and \ref{Subsec:Rayleigh} the operators $K$ and $W$ are denoted by $K_m$ and $W_m$, respectively, and $P$ by $P_+$).

Before stating our result regarding the lower bound for $L_\Omega^\star$, let us briefly explain the heuristics behind it. Recall that $L_\Omega^\star$ arises when considering the behavior of  $\lambda_\Omega(\tau)$ as $\tau\downarrow-\infty$. If one looks at the associated eigenfunction for 
$\lambda_\Omega(\tau)$, in the limit $\tau\downarrow-\infty$ one is led to an eigenvalue-type problem of the form
\begin{equation}\label{form.rayleigh.l.omega.setL.heur}
	P^*u=L_\Omega^\star( \sigma\cdot\nu)K( \sigma\cdot\nu)u
	\quad\text{for some }u\in P(L^2(\partial\Omega)^2)\setminus\{0\},
\end{equation}
where $P^*$ denotes the adjoint of $P$.
Hence, in order to estimate $L_\Omega^\star$ it is natural to investigate the set
\begin{equation}\label{form.rayleigh.l.omega.setL}
	\mathcal L_\Omega:=\Big\{L\in\C:\,
	\text{there exists }u\in P(L^2(\partial\Omega)^2)\setminus\{0\}
	\text{ with }
	P^*u=
	L( \sigma\cdot\nu)K( \sigma\cdot\nu)u\Big\}.
\end{equation}
By definition, we have $L_\Omega^\star\in\mathcal L_\Omega$. 
Despite that we initially think of $\mathcal L_\Omega$ as a subset of $\C$, it turns out that $\mathcal L_\Omega$ only contains positive real numbers. Actually, 
given $L\in\mathcal L_\Omega$ and an associated function $u$ as in 
\eqref{form.rayleigh.l.omega.setL}, it is not hard to show (see the proof of \Cref{Th:Rayleigh.Intro.l2}) that 
\begin{equation}\label{heurist.1}
	\|u\|_{L^2(\partial\Omega)^2}^2
	=L\langle(\sigma\cdot\nu)K(\sigma\cdot\nu)u,u\rangle_{L^2(\partial\Omega)^2}.
\end{equation}
This suggests considering the functional
\begin{equation}\label{form.rayleigh.R}
	\rcal(u):=
	\frac{\langle(\sigma\cdot\nu)K(\sigma\cdot\nu)u,u\rangle_{L^2(\partial\Omega)^2}}{\|u\|_{L^2(\partial\Omega)^2}^2}\quad
	\text{for } u\in L^2(\partial\Omega)^2\setminus\{0\},
\end{equation}
so that for $L$ and $u$ as in \eqref{heurist.1} we get 
$L=1/\rcal(u)$. Since 
$K$ is bounded in $L^2(\partial\Omega)^2$ and is strictly positive, we have $0<\rcal(u)\leq\|K\|_{L^2(\partial\Omega)^2\to L^2(\partial\Omega)^2}$
for all $u\in L^2(\partial\Omega)^2\setminus\{0\}$, and this yields $$\mathcal L_\Omega\subset\big[1/\|K\|_{L^2(\partial\Omega)^2\to L^2(\partial\Omega)^2},+\infty\big)\subset\R$$ by \eqref{heurist.1}. Therefore, it is reasonable to use $\inf\mathcal L_\Omega$  to bound $L_\Omega^\star$ from below. Finally, recalling the relation $L=1/\rcal(u)$ given by \eqref{heurist.1}, and 
defining
\begin{equation}\label{form.rayleigh.r.omega}
	\rcal_\Omega:=\sup_{u\in P(L^2(\partial\Omega)^2)\setminus\{0\}}\rcal(u),
\end{equation}
we end up with $$L_\Omega^\star\geq\inf\mathcal L_\Omega\geq 1/\rcal_\Omega.$$
This is the lower bound for $L_\Omega^\star$ in terms of the optimization of a Rayleigh quotient that we were referring to as our last main result in this work.
At this point the reader might think that substantial information may have been lost during all these steps, and that in the end we have $L_\Omega^\star>1/\rcal_\Omega$ for all $\Omega$. However, we will show that the equality holds if 
$\Omega$ is a ball, hence the lower bound that we found is sharp. Furthermore, we will prove that 
$1/\rcal_\Omega=\min \mathcal L_\Omega$ for every $\Omega$, which means that $1/\rcal_\Omega$ is actually the smallest admissible value in the eigenvalue-type problem \eqref{form.rayleigh.l.omega.setL.heur}. It would be very interesting to know whether the equality $1/\rcal_\Omega=L_\Omega^\star$ holds in general or not.

We gather all these considerations in the following theorem,  which is our last main result of this work. It will be proven in \Cref{Subsec:Rayleigh}.

\begin{theorem}\label{Th:Rayleigh.Intro}
	Let $\Omega\subset \R^3$ be a bounded domain with $C^2$ boundary. The following holds:
	\begin{itemize}
		\item[$(i)$] If $L_\Omega^\star<+\infty$ then $L_\Omega^\star\in\mathcal L_\Omega$.
		\item[$(ii)$] $\mathcal L_\Omega\subset\R$, and 
		$1/\rcal_\Omega\leq L$ for all $L\in\mathcal L_\Omega$.
		\item[$(iii)$] The supremum in \eqref{form.rayleigh.r.omega} is attained. Moreover, if  
		$u\in P(L^2(\partial\Omega)^2)\setminus\{0\}$ is such that 
		$\rcal(u)=\rcal_\Omega$, then
		$P^*u=
		\frac{1}{\rcal_\Omega}( \sigma\cdot\nu)K( \sigma\cdot\nu)u$. In particular, 
		$1/\rcal_\Omega\in\mathcal L_\Omega$.
	\end{itemize}
	
	As a consequence of $(i)$, $(ii)$, and $(iii)$, it holds that
	$1/\rcal_\Omega=\min\mathcal L_\Omega\leq L_\Omega^\star$ for every bounded domain $\Omega\subset \R^3$ with $C^2$ boundary. Furthermore, $1/\rcal_B=L_B^\star$ for every ball 
	$B\subset \R^3$.
\end{theorem}

In view of \Cref{Th:Rayleigh.Intro}, the question that we would like to answer at this point is whether the ball is the unique maximizer of $\rcal_\Omega$ among all bounded and smooth domains $\Omega$ of the same volume.  
The fact that Hardy spaces come into play in \eqref{form.rayleigh.r.omega} makes this question difficult and challenging, and it requires further study to be answered. 
By the same argument as in \eqref{ball.best.-infty}, but using also \Cref{Th:Rayleigh.Intro}, an affirmative answer would yield that if $\Omega$ is not a ball then for $-\tau$ large enough the first positive eigenvalue of $\Dirac_\tau$ on $\Omega$ is strictly greater than the one on a ball with the same volume as $\Omega$.
That is, we would get the analogue of \Cref{Th:BallOptimalLargeTau} in the regime $\tau \downarrow -\infty$.

Having in mind the results presented above (and with a bit of positive thinking), we would like to finish this introduction by posing the following
\begin{conjecture}\label{conjecture_FK_all_tau}
	Let $\Omega\subset \R^3$ be a bounded domain with $C^2$ boundary, and let $B\subset\R^3$ be a ball such that 
	$|\Omega| = |B|$. If $\Omega$ is not a ball,
	then $\lambda_B(\tau) < \lambda_\Omega(\tau)$ for all $\tau \in\R$.
\end{conjecture}
A positive answer to this conjecture for $\tau=0$ would give the solution to the shape optimization problem for the spectral gap of the MIT bag operator on smooth domains.

\begin{remark}
	\Cref{conjecture_FK_all_tau} also makes sense in the two-dimensional framework, that is, for a bounded domain 
	$\Omega\subset \R^2$. The corresponding Dirac operator would be
	$-i(\sigma_1\partial_1+\sigma_2 \partial_2) +m\sigma_3$ 
	acting on functions 
	$\varphi \in H^1(\Omega)^2$ such that 
	\begin{equation}
		\varphi =i (\sinh\tau- \cosh\tau \, \sigma_3)( \sigma_1\nu_1+\sigma_2 \nu_2 ) \varphi\quad\text{on $\partial \Omega$}.
	\end{equation}
	When $\tau=0$, this is the so-called {\em Dirac operator with infinite mass boundary conditions}.
	
	The techniques used in this work are not specific of the three-dimensional framework. They mainly depend on the algebraic structure of the operator and its boundary conditions. Therefore, they would also apply in the two-dimensional framework, taking into account the obvious modifications which arise when replacing $\alpha$ by $(\sigma_1,\sigma_2)$ and 
	$\beta$ by $\sigma_3$. In this regard, we expect that the main results presented in \Cref{Subsec:MainResults}  also hold true for the two-dimensional Dirac operator. 
\end{remark}

\addtocontents{toc}{\SkipTocEntry}
\subsection*{Organization of the paper} 

In \Cref{Sec:PropSpectrum} we present some preliminary results and introduce the boundary integral operators associated to $\Dirac$, proving some regularity estimates.
With these tools at hand, in \Cref{Sec:Parametrization+Monotonicity} we establish \Cref{Th:ParamEigenvalues}, proving that the eigenvalues of $\Dirac_\tau$ can be parametrized as increasing functions of $\tau\in \R$.
Then, in \Cref{Sec:AsymptoricBehavior} we study the asymptotic behavior of the eigenvalue curves.
First, in \Cref{Subsec:EigLimits} we prove \Cref{Th:EigLimits} regarding the limits of the eigenvalue curves, \Cref{Th:FirstEigTom}, on the first positive eigenvalue, and \Cref{Th:BallOptimalLargeTau}, the shape optimization result.
Next, in \Cref{Subsec:Proj.Hardy} we introduce skew projections onto Hardy spaces, which are used in \Cref{Subsec:Rayleigh} to establish \Cref{Th:Rayleigh.Intro}.

This article contains two appendixes.
In \Cref{Sec:Formulas} we establish the main properties of the spectrum of $\Dirac_\tau$, gathered in \Cref{Lemma:SpectrumGenMIT}.
In there the reader can also find a formula relating  $\Dirac_\tau$ with the mean curvature of $\partial \Omega$.
Finally, \Cref{Sec:ExplicitComputationsSphere} is devoted to the spectral analysis when the underlying domain is a ball, and we provide an extensive and explicit description of the eigenvalue curves.

\section{Regularity estimates} \label{Sec:PropSpectrum}

In this section, after presenting some preliminary results, we will introduce  boundary integral operators associated to $\Dirac$. 
We will use them to rewrite the eigenvalue equation for $\Dirac_\tau$ as an integral equation on $\partial\Omega$; see \Cref{Lemma:BoundaryPb} below.
Then, we will establish regularity estimates for solutions to this boundary eigenvalue equation (stated in \Cref{Lemma:BoundaryPb.2}) which will be crucial in the following sections when using compactness arguments.

\subsection{Preliminaries} \label{Subsec:Preliminaries}

The contents of this section are known results from the literature and simple computations that will be used in the sequel. Let us begin by recalling a result from~\cite{OurmieresBonafosVega} which explores the relation between a function in $\Omega$ an its trace on 
$\partial \Omega$. Given $u \in L^2(\Omega)^2$ such that 
$\sigma \cdot \nabla u \in L^2(\Omega)^2$, in general one cannot assure that $u \in H^1(\Omega)^2$. 
Nevertheless, as the following lemma shows, if the trace of $u$ belongs to $H^{1/2}(\partial \Omega)^2$, then $u \in H^1(\Omega)^2$.
Here, $H^s({\partial \Omega})^d$ with $s \in (0, 1)$ and $d>0$ integer denotes the fractional Sobolev space of functions $f\in L^2(\partial\Omega)^d$ such that 
\begin{equation} \label{def_Sobolev_Slobodeckii}
	\|f\|_{H^s(\partial\Omega)^d}:=\bigg(
	\int_{{\partial \Omega}}|f|^2\, d\upsigma
	+\int_{{\partial \Omega}}\int_{{\partial \Omega}}
	\frac{|f(x)-f(y)|^2}{|x-y|^{2+2s}}\,d\upsigma(y)\,d\upsigma(x)\bigg)^{1/2}<+\infty.
\end{equation}
We will also use the symbol $H^{-s}(\partial \Omega)^d$ to denote the continuous dual of $H^s({\partial \Omega})^d$.
Recall that 
$$\|\cdot\|_{H^{-s}(\partial\Omega)^d}\leq 
\|\cdot\|_{L^2(\partial\Omega)^d}\leq
\|\cdot\|_{H^{s}(\partial\Omega)^d}\quad \text{for all } s\in(0,1).$$

\begin{lemma}
	\label{Lemma:SigmaGradEstimates}
	Let $u \in L^2(\Omega)^2$ be such that $ \sigma \cdot \nabla u \in L^2(\Omega)^2$. 
	Then,	
	
	\begin{itemize}
		\item[$(i)$] the trace of $u$ belongs to $H^{-1/2}(\partial \Omega)^2$, and satisfies 
		\begin{equation}
			\label{Eq:SigmaGradTraceEstimate}
			\|u \|_{H^{-1/2}(\partial \Omega)^2} \leq C\left (   \| u \|_{L^2(\Omega)^2}    + \| \sigma \cdot \nabla u \|_{L^2(\Omega)^2} \right)
		\end{equation}
		for some $C>0$ depending only on $\Omega$.
		
		\item[$(ii)$] Assume in adittion that $u\in H^{1/2}(\partial \Omega)^2$.
		Then, $u \in H^1(\Omega)^2$ and 
		\begin{equation}
			\label{Eq:SigmaGradH1Estimate}
			\| u \|_{H^1(\Omega)^2} \leq C \left (  \|  u \|_{L^2(\Omega)^2} + \| \sigma \cdot\nabla u \|_{L^2(\Omega)^2}  + \| u \|_{H^{1/2}(\partial \Omega)^2} \right )
		\end{equation}
		for some $C>0$ depending only on 
		$\Omega$.
	\end{itemize}
\end{lemma}

\begin{proof}
	The proof of these results follows similar arguments as the ones in \cite[Section~2]{OurmieresBonafosVega}, and therefore some details may be omitted. 
	Throughout the proof, we shall denote 
	$$
	H({\sigma}, \Omega) := \big\{u\in L^2(\Omega)^2:\, \sigma \cdot \nabla u \in L^2(\Omega)^2\big\},
	$$
	which is a Hilbert space equipped with the scalar product
	$$
	\langle u, v \rangle_{H({\sigma}, \Omega)} := \langle u, v \rangle_{L^2(\Omega)^2} + \langle {\sigma} \cdot \nabla u, {\sigma} \cdot \nabla v \rangle_{L^2(\Omega)^2}
	$$
	and the associated norm $\norm{u}_{H({\sigma}, \Omega)} := \big(\| u \|_{L^2(\Omega)^2}^2    + \| \sigma \cdot \nabla u \|_{L^2(\Omega)^2}^2\big)^{1/2}$.

	Statement $(i)$ simply means that the trace operator is bounded from $H({\sigma}, \Omega)$ into $H^{-1/2}(\partial \Omega)^2$, and it is proved exactly as \cite[Proposition~2.1]{OurmieresBonafosVega}. 
	Hence, we will only address the proof of $(ii)$.
	Given $u\in  H({\sigma}, \Omega)$ such that $u\in H^{1/2}(\partial \Omega)^2$, set
	$$
	\tilde{u} := u - E(u),
	$$
	where $E:  H^{1/2}(\partial \Omega)^2\to H^1(\Omega)^2$ is a bounded extension operator, i.e.,
	\begin{equation}
		\label{Eq:ExtensionOperator}
		\norm{ E(u)}_{H^1(\Omega)^2} \leq C \norm{u}_{H^{1/2}(\partial \Omega)^2}
	\end{equation}
	for some $C>0$ depending only on $\Omega$.
	Note that the new function $\tilde{u}$ has zero trace on $\partial \Omega$ and, therefore, denoting by $u_0$ its extension by $0$ in $\R^3\setminus\Omega$, we have that
	\begin{equation}
		\label{Eq:SigmaDeriveExtZero}
		{\sigma} \cdot \nabla u_0 = ({\sigma} \cdot \nabla \tilde{u}) \chi_\Omega\quad\text{in $\R^3$},
	\end{equation}
	where $\chi_\Omega$ denotes the characteristic function of the set $\Omega$; for a rigorous proof of \eqref{Eq:SigmaDeriveExtZero} see for instance the proof of \cite[Proposition~2.16]{OurmieresBonafosVega}.
	
	We claim that
	\begin{equation}
		\label{Eq:EqNorms}
		\norm{\tilde{u}}_{H^1(\Omega)^2} = \norm{\tilde{u}}_{H({\sigma},\Omega)}.
	\end{equation}
	Note first that $\norm{\tilde{u}}_{H^1(\Omega)^2} = \norm{u_0}_{H^1(\R^3)^2}$. Using the fact that $H^1(\R^3)^2 = H({\sigma}, \R^3)$, which can be shown using integration by parts, we obtain 	$\norm{\tilde{u}}_{H^1(\Omega)^2} = \norm{u_0}_{H({\sigma},\R^3)}$. 
	Thus, the claim follows from \eqref{Eq:SigmaDeriveExtZero}.
	
	Finally, noting that 
	$$
	\norm{\tilde{u}}_{H({\sigma},\Omega)} 
	\leq  \norm{u}_{H({\sigma},\Omega)} +  \norm{ E(u)}_{H({\sigma},\Omega)} 
	\leq \norm{u}_{H({\sigma},\Omega)} +  \norm{ E(u)}_{H^1(\Omega)^2}
	$$
	and that
	$$
	\norm{\tilde{u}}_{H^1(\Omega)^2}  \geq \norm{u}_{H^1(\Omega)^2} - \norm{ E(u)}_{H^1(\Omega)^2},
	$$
	from \eqref{Eq:EqNorms} we get
	$$
	\norm{u}_{H^1(\Omega)^2} \leq \norm{u}_{H({\sigma},\Omega)} +  2\norm{ E(u)}_{H^1(\Omega)^2},
	$$
	and using \eqref{Eq:ExtensionOperator} we conclude the proof of $(ii)$.
\end{proof}

Using \Cref{Lemma:SigmaGradEstimates} we can easily get estimates for the trace of the eigenfunctions of 
$\Dirac_\tau$, once they are written component-wise. Given 
$\tau,\lambda\in\R$, by definition we have that 
$\varphi\in\mathrm{Dom}(\Dirac_\tau)$ satisfies 
$\Dirac_\tau\varphi=\lambda\varphi$ if and only if $\varphi\in H^1(\Omega)^4$ and
\begin{equation}
	\label{Eq:EigenvaluePb.1}
	\beqc{\PDEsystem}
	\Dirac \varphi  &=&  \lambda \varphi & \text{in } \Omega, \\
	\varphi &=&  i (\sinh \tau - \cosh \tau \, \beta) (\alpha \cdot\nu ) \varphi & \text{on } \partial \Omega.
	\eeqc
\end{equation}
If we write
$\varphi=(u,v)^\intercal$
with $u, v \in H^1(\Omega)^2$, 
and we recall that $\Dirac = -i \alpha \cdot \nabla + m \beta$,
then \eqref{Eq:EigenvaluePb.1} is equivalent to
\begin{equation}
	\label{Eq:EigenvaluePbUV.1}
	\beqc{\PDEsystem}
	-i {\sigma} \cdot \nabla v & = & (\lambda-m) u  & \text{in } \Omega, \\
	-i {\sigma} \cdot \nabla u & = & (\lambda+m) v  & \text{in } \Omega, \\
	v & = & i e^\tau ( {\sigma}\cdot\nu)u & \text{on } \partial \Omega. 
	\eeqc
\end{equation}
In particular, since $(-i {\sigma} \cdot \nabla)^2 = - \Delta$, from the first two equations in \eqref{Eq:EigenvaluePbUV.1} we deduce that
\begin{equation}
	\label{Eq:HelmholtzUV}
	- \Delta  u  =  (\lambda^2 -m^2) u\quad\text{and}\quad
	- \Delta  v  =  (\lambda^2 -m^2) v \quad  \text{in } \Omega.
\end{equation}
In addition, combining \eqref{Eq:EigenvaluePbUV.1} with \Cref{Lemma:SigmaGradEstimates}~$(i)$ we obtain the following result.

\begin{lemma}
	\label{Lemma:H1/2Estimate}
	Let $u,v\in H^1(\Omega)^2$ be such that 
	$\varphi=(u,v)^\intercal\in H^1(\Omega)^4$ solves \eqref{Eq:EigenvaluePb.1}. Then,
	\begin{equation}
		\label{Eq:H-1/2Estimate}
		\begin{split}
			\|u \|_{H^{-1/2}(\partial \Omega)^2} &\leq  C\left(   \| u \|_{L^2(\Omega)^2}    + |\lambda+m|  \| v \|_{L^2( \Omega)^2}\right), \\	
			\|v \|_{H^{-1/2}(\partial \Omega)^2} &\leq  C\left( |\lambda-m|  \| u \|_{L^2(\Omega)^2}    +   \| v \|_{L^2( \Omega)^2}\right) 
		\end{split}
	\end{equation}
	for some $C>0$ depending only on $\Omega$.
\end{lemma}

\begin{proof}
	
	Using \Cref{Lemma:SigmaGradEstimates}~$(i)$ we get
	$
	\|u \|_{H^{-1/2}(\partial \Omega)^2} \leq C(   \| u \|_{L^2(\Omega)^2}    + \| {\sigma} \cdot \nabla u \|_{L^2(\Omega)^2} ).
	$
	The equality $-i {\sigma} \cdot \nabla u = (\lambda+m) v$ from \eqref{Eq:EigenvaluePbUV.1} yields the first estimate. The second one is analogous.
\end{proof}

We finish this section by giving the relation between the 
$L^2$-norm of $u$ on $\partial\Omega$ and the $L^2$-norms of $u$ and $v$ on $\Omega$. An analogous formula for the $L^2$-norm of $v$ on $\partial\Omega$ can be obtained using the last equality in \eqref{Eq:EigenvaluePbUV.1}.

\begin{lemma}
	\label{Lemma:RelationNormsIntBdry}
	Let $u,v\in H^1(\Omega)^2$ such that 
	$\varphi=(u,v)^\intercal\in H^1(\Omega)^4$ solves \eqref{Eq:EigenvaluePb.1}. Then,
	\begin{equation}
		\label{Eq:RelationNormsIntBdry}
		e^\tau \| u \|^2_{L^2(\partial \Omega)^2}= 
		e^{-\tau} \| v \|^2_{L^2(\partial \Omega)^2}= (\lambda-m) \| u \|^2_{L^2( \Omega)^2} - (\lambda+m)\| v \|^2_{L^2( \Omega)^2}\,.
	\end{equation}
\end{lemma}

\begin{proof}
	Note that $e^\tau \| u \|^2_{L^2(\partial \Omega)^2}= 
	e^{-\tau} \| v \|^2_{L^2(\partial \Omega)^2}$ because $v=ie^\tau({\sigma}\cdot\nu)u$ by \eqref{Eq:EigenvaluePbUV.1}. 
	Now,
	if we multiply the first equation in \eqref{Eq:EigenvaluePbUV.1} by $\overline{u}$, we integrate by parts in $\Omega$, and we use the other two equations in \eqref{Eq:EigenvaluePbUV.1}, we get
	\begin{equation} \begin{split}
			(\lambda-m) \int_{\Omega} |u|^2  
			&= \int_{\Omega} (-i {\sigma} \cdot \nabla  v) \cdot \overline{ u } 
			= \int_{\Omega} v \cdot\overline{(-i {\sigma} \cdot \nabla u)} - i \int_{\partial \Omega}  ({\sigma}\cdot\nu) v \cdot\overline{ u } \,d \upsigma \\
			&= (\lambda + m) \int_{\Omega} |v|^2
			+ e^\tau \int_{\partial \Omega} |u|^2 \,d \upsigma.
	\end{split} \end{equation}
\end{proof}

\subsection{Boundary integral operators}\label{Subsec:BIOperators}

In this section we introduce boundary integral operators associated to $\Dirac$. 
We will use them to rewrite the eigenvalue problem as an integral equation on $\partial\Omega$ (see \Cref{Lemma:BoundaryPb} below), and to prove regularity estimates for the eigenfunctions of $\Dirac_\tau$ that will be used in the following sections.
The spectral analysis using boundary integral operators is an approach already used in \cite{AMV1,AMV2,AMV3}. 
It represents the Birman-Schwinger principle on our setting.

A fundamental solution of $\Dirac-\lambda$ for $\lambda\in\R$ is given by
\begin{equation}
	\label{Eq:DefPhi}
	\phi_\lambda(x) := \dfrac{e^{-i \sqrt{\lambda^2 - m^2} |x|}}{4\pi |x|} \left[ \lambda + m \beta + \left (1 + i \sqrt{\lambda^2 - m^2} |x| \right ) i \alpha \cdot \dfrac{x}{|x|^2} \right]
	\quad\text{for all $x\in\R^3\setminus\{0\}$}
\end{equation}
if $\lambda \in (-\infty,-m)\cup(m,+\infty)$, and by
\begin{equation}
	\label{Eq:DefPhi.selfadj}
	\phi_\lambda(x):=  \dfrac{e^{- \sqrt{ m^2 - \lambda^2} |x|}}{4\pi |x|} \left[ \lambda + m \beta + \left (1 +  \sqrt{ m^2 - \lambda^2} |x| \right ) i \alpha \cdot \dfrac{x}{|x|^2} \right]\quad\text{for all $x\in\R^3\setminus\{0\}$}
\end{equation}
if $\lambda\in[-m,m]$.
This means that $(\Dirac-\lambda)\phi_\lambda=\delta_0 I_4$ in the sense of distributions, where $\delta_0$ denotes the Dirac delta measure centered at the origin of coordinates. Therefore, given $g\in L^2(\partial \Omega)^4$, if we set
\begin{equation}
	\label{Eq:DefPhiMaj}
	\Phi_\lambda g(x) := \int_{\partial \Omega} \phi_\lambda(x-y) g (y)\, d \upsigma(y) \quad \text{ for all $x \in \Omega$},
\end{equation}
then $(\Dirac-\lambda)\Phi_\lambda g=0$ in $\Omega$. Reciprocally, null solutions of $\Dirac-\lambda$ in 
$\Omega$ can be expressed by means of $\phi_\lambda$, as the following result shows. 

\begin{lemma}[Reproducing formula]
	\label{Lemma:ReproductionFormula}
	Let $\lambda \in \R$ and $\varphi \in H^1(\Omega)^4$ be such that
	$(\Dirac-\lambda)\varphi = 0$ in $\Omega$. 
	Then, $\varphi = \Phi_\lambda (i (\alpha \cdot\nu)\varphi)$ in $\Omega$. That is,
	\begin{equation}
		\label{Eq:ReproductionFormula}
		\varphi (x) = \int_{\partial \Omega} \phi_\lambda(x-y) i (\alpha \cdot\nu(y)) \, \varphi (y) \,d \upsigma(y) \quad \text{ for all } x \in \Omega.
	\end{equation}
\end{lemma}

\begin{proof}
	Note that if $\Dirac\varphi=\lambda\varphi$ in 
	$\Omega$ then 
	$-\Delta\varphi=(\lambda^2-m^2)\varphi$ in 
	$\Omega$. By standard regularity theory, $\varphi$ is infinitely differentiable in $\Omega$. In particular, it makes sense to write ``for all $x\in\Omega$'' in \eqref{Eq:ReproductionFormula}.
	
	The proof of the lemma follows the same lines as in \cite[Lemma~3.3]{AMV1}: given $x\in\Omega$, one combines that $\phi_\lambda$ is a fundamental solution of $\Dirac-\lambda$ with an integration by parts in $\Omega \setminus \{y\in\R^3:\,|x-y|<\epsilon\}$, and then one lets $\epsilon \downarrow 0$ to obtain \eqref{Eq:ReproductionFormula}. We omit the details.
\end{proof}

In the following lemma, the boundary trace of the function
$\Phi_\lambda g$ given in \eqref{Eq:DefPhiMaj} is described in terms of a singular integral operator, and a trace identity for null solutions of $\Dirac-\lambda$ is deduced thanks to \Cref{Lemma:ReproductionFormula}. This operator is defined, for every $g \in L^2(\partial \Omega)^4$, by 
\begin{equation}
	\label{Eq:DefC}
	\ccal_\lambda g(x) := \lim_{\epsilon \downarrow 0} 
	\int_{\{y\in\partial \Omega:\, |x-y|> \epsilon\}} \phi_\lambda(x-y) g (y) \,d \upsigma(y) \quad \text{ for $\upsigma$-a.e.\! $x \in \partial \Omega$}.
\end{equation}
The same arguments as in \cite[Lemma~3.3]{AMV1} show that $\ccal_\lambda$ is bounded in $L^2(\partial\Omega)^4$. In addition, since $\ccal_\lambda$ is defined by means of the kernel \eqref{Eq:DefPhi.selfadj} for $\lambda\in[-m,m]$, $\ccal_\lambda$ is self-adjoint in $L^2(\partial\Omega)^4$ for 
$\lambda\in[-m,m]$.

\begin{lemma}[Plemelj-Sokhotski jump formula]
	\label{Lemma:PropertiesOfC}
	Let $\lambda \in \R$.
	Then, for every $g\in  L^2(\partial \Omega)^4$,
	\begin{equation}
		\label{Eq:Plemelj-Sokhotski}
		\lim_{z \in \Omega, \, z \to x} \Phi_\lambda g(z) = - \dfrac{i}{2} (\alpha \cdot\nu(x) ) g(x) + \ccal_\lambda g(x) \quad \text{ for $\upsigma$-a.e.\! $x \in \partial \Omega$},
	\end{equation}
	where the limit on the left-hand side of \eqref{Eq:Plemelj-Sokhotski} is taken nontangentially.
	As a consequence, if $\varphi\in H^1(\Omega)^4$ is such that $\Dirac\varphi=\lambda \varphi$ in $\Omega$, then 
	$\varphi = 2 i \ccal_\lambda (\alpha \cdot\nu) \varphi$  
	$\upsigma$-almost everywhere on $\partial \Omega$.
\end{lemma}

\begin{proof}
	The identity \eqref{Eq:Plemelj-Sokhotski} is proved exactly as in \cite[Lemma~3.3]{AMV1}.
	Now, let $\varphi\in H^1(\Omega)^4$ be such that $\Dirac \varphi = \lambda \varphi$ in $\Omega$. 
	From Lemma~\ref{Lemma:ReproductionFormula} it follows that 
	$\varphi = \Phi_\lambda (i (\alpha \cdot\nu)  \varphi)$ in 
	$\Omega$.
	Taking the nontangential limit towards $\partial \Omega$, and using \eqref{Eq:Plemelj-Sokhotski} and that $(\alpha\cdot\nu)^2=1$, we have 
	\begin{equation}
		\varphi = -\dfrac{i}{2} (\alpha \cdot\nu)(i(\alpha \cdot\nu) \varphi) + i\ccal_\lambda (\alpha \cdot\nu) \varphi
		= \dfrac{1}{2}  \varphi + i\ccal_\lambda (\alpha \cdot\nu) \varphi\quad 
		\text{$\upsigma$-a.e.\! on $\partial \Omega$}.
	\end{equation}
\end{proof}

\begin{remark}
	Note that the left-hand side of \eqref{Eq:Plemelj-Sokhotski} concerns a pointwise nontangential limit. If  
	$\Phi_\lambda g\in H^1(\Omega)$, this limit coincides 
	$\upsigma$-a.e.\! on $\partial \Omega$ with the standard trace of functions in Sobolev spaces defined through a trace operator; see \cite[Lemma 3.1]{BGM.trace} for example.
\end{remark}

In \Cref{Lemma:BoundaryPb} below we will use  \Cref{Lemma:PropertiesOfC} to rewrite the eigenvalue problem 
$\Dirac_\tau\varphi=\lambda\varphi$ as an integral equation for $\varphi$ on $\partial\Omega$. To do so, it is convenient to express the operator $\ccal_\lambda$ in terms of its action on the components of 
$\varphi\in L^2(\partial\Omega)^4$, namely,
\begin{equation}
	\label{Eq:DecompCa}
	\ccal_\lambda\varphi = \begin{pmatrix}
		(\lambda+m) K_\lambda & W_\lambda \\
		W_\lambda & (\lambda -m) K_\lambda
	\end{pmatrix}\begin{pmatrix}u\\v\end{pmatrix},
	\quad\text{where $\varphi=\begin{pmatrix}u\\v\end{pmatrix}$ and $u,v\in L^2(\partial\Omega)^2$.}
\end{equation}
The operators $K_\lambda$ and $W_\lambda$ are defined, for every $u\in L^2(\partial \Omega)^2$ and $\upsigma$-a.e.\! $x \in \partial \Omega$, by
\begin{equation}\label{Eq:DefK}
	\begin{split}
		K_\lambda u(x) &:=  \dfrac{1}{4\pi}\int_{\partial \Omega} k_\lambda(x-y) u (y) \,d \upsigma(y),\\
		W_\lambda u(x) &:=  \lim_{\epsilon \downarrow 0} \dfrac{1}{4\pi}\int_{\{y\in\partial \Omega:\, |x-y|> \epsilon\}}  w_\lambda(x-y) u (y) \,d \upsigma(y),
	\end{split}
\end{equation}
where the kernels $k_\lambda$ and $w_\lambda$ are given, for $x\in\R^3\setminus\{0\}$, by
\begin{equation}\label{kernels.K.W}
	\begin{split}
		k_\lambda(x)
		:=\dfrac{e^{-i \sqrt{\lambda^2 - m^2} |x|}}{|x|},\quad
		w_\lambda (x) := e^{-i \sqrt{\lambda^2 - m^2} |x|}\left (1 + i \sqrt{\lambda^2 - m^2} |x| \right ) i {\sigma} \cdot \dfrac{x}
		{|x|^3}
	\end{split}
\end{equation}
if $\lambda \in (-\infty,-m)\cup(m,+\infty)$, and by
\eqref{kernels.K.W} replacing $i \sqrt{\lambda^2 - m^2}$ by
$\sqrt{m^2-\lambda^2}$ if $\lambda\in[-m,m]$.
The fact that $\ccal_\lambda$ is bounded in $L^2(\partial\Omega)^4$ yields the boundedness of $K_\lambda$ and $W_\lambda$ in $L^2(\partial\Omega)^2$. In addition, since $\ccal_\lambda$ is self-adjoint in $L^2(\partial\Omega)^4$ for 
$\lambda\in[-m,m]$, we deduce that $K_\lambda$ and $W_\lambda$ are self-adjoint in $L^2(\partial\Omega)^2$ for $\lambda\in[-m,m]$.

The combination of \Cref{Lemma:ReproductionFormula,Lemma:PropertiesOfC} leads to the important identity $(\ccal_\lambda(\alpha\cdot\nu))^2=-1/4$ as operators in $L^2(\partial\Omega)^4$. This, apart from showing that 
$(\ccal_\lambda)^{-1}=-4(\alpha\cdot\nu)\ccal_\lambda(\alpha\cdot\nu)$, has some consequences on $K_\lambda$ and $W_\lambda$ that will be used in the proof of \Cref{Rstar.EL}. We gather them  in the following result.

\begin{lemma}\label{l.prop.K.W}
	For every $\lambda\in\R$, the following holds:
	\begin{itemize}
		\item[$(i)$] $(\ccal_\lambda(\alpha\cdot\nu))^2=-1/4$ as operators in $L^2(\partial\Omega)^4$.
		\item[$(ii)$] $(W_\lambda(\sigma\cdot\nu))^2
		+(\lambda^2-m^2)(K_\lambda(\sigma\cdot\nu))^2=-1/4$ as operators in $L^2(\partial\Omega)^2$.
		\item[$(iii)$] $\{K_\lambda(\sigma\cdot\nu),
		W_\lambda(\sigma\cdot\nu)\}=0$ as operators in $L^2(\partial\Omega)^2$.
	\end{itemize}
\end{lemma}

\begin{proof}
	Let us show $(i)$. Given $g\in C^\infty(\partial\Omega)^4$, set 
	$\varphi=\Phi_\lambda g$. Then, $\varphi\in H^1(\Omega)$ and $(H-\lambda)\varphi=0$ in $\Omega$. Thus, by \Cref{Lemma:ReproductionFormula} and \eqref{Eq:Plemelj-Sokhotski} we get that
	$$\Phi_\lambda g=\varphi = \Phi_\lambda (i (\alpha \cdot\nu)\varphi)
	=\Phi_\lambda \Big(i (\alpha \cdot\nu)\big({\textstyle - \frac{i}{2}} (\alpha \cdot\nu ) g + \ccal_\lambda g\big)\Big)
	=\Phi_\lambda \Big( \big({\textstyle \frac{1}{2}}   + i(\alpha \cdot\nu )\ccal_\lambda \big)g\Big)$$
	in $\Omega$. Taking traces on $\partial\Omega$ and using once again \eqref{Eq:Plemelj-Sokhotski}, we deduce that
	\begin{equation}
		\begin{split}
			\big({\textstyle - \frac{i}{2}} (\alpha \cdot\nu )  + \ccal_\lambda\big) g
			&=\big({\textstyle - \frac{i}{2}} (\alpha \cdot\nu )  + \ccal_\lambda\big)\big({\textstyle \frac{1}{2}}   + i(\alpha \cdot\nu )\ccal_\lambda \big)g\\
			&=\big({\textstyle - \frac{i}{4}}(\alpha \cdot\nu ) 
			+\ccal_\lambda+i\ccal_\lambda(\alpha \cdot\nu )\ccal_\lambda\big)g.
		\end{split}
	\end{equation}
	From here, and by the density of $C^\infty(\partial\Omega)^4$
	in $L^2(\partial\Omega)^4$, we get
	${\textstyle - \frac{1}{4}} (\alpha \cdot\nu )  
	= \ccal_\lambda(\alpha \cdot\nu )\ccal_\lambda$, which proves $(i)$. Then, $(ii)$ and $(iii)$ follow by $(i)$ writing $\ccal_\lambda$ as in \eqref{Eq:DecompCa}, in terms of $K_\lambda$ and $W_\lambda$.
\end{proof}

Let us now establish a regularity result that will be a key tool in the compactness arguments carried out in \Cref{Subsec:EigLimits}.
It concerns regularity estimates for the operators $ K_\lambda $ and $ \{W_\lambda, \sigma \cdot\nu\}$.

\begin{lemma}
	\label{Prop:RegularityKW}
	For every $\lambda \in \R$ there exists $C_\lambda>0$ depending only on $\Omega$ and $\lambda^2 - m^2$ such that
	\begin{equation}
		\label{Eq:RegularityK}
		\norm{ K_\lambda u }_{H^{1/2}(\partial \Omega)^2} \leq C_\lambda \norm{u}_{H^{-1/2}(\partial \Omega)^2}
	\end{equation}
	and 
	\begin{equation}
		\label{Eq:RegularityWComm}
		\norm{ \{W_\lambda, \sigma \cdot\nu\}u}_{H^{1/2}(\partial \Omega)^2} \leq C_\lambda \norm{u}_{H^{-1/2}(\partial \Omega)^2}
	\end{equation}
	for all $u\in H^{-1/2}(\partial \Omega)^2$. 
	The constant $C_\lambda$ can be chosen in such a way that for every $C_\circ>0$ there exists $C'>0$ depending only on $\Omega$ and $C_\circ$ such that 		$C_\lambda<C'$ for all $\lambda\in\R$ with $|\lambda^2-m^2|<C_\circ$.
\end{lemma}

\begin{proof}	
	The result essentially follows from the fact that both operators $K_\lambda$ and $ \{W_\lambda, \sigma \cdot\nu\}$ are integral operators given by kernels whose singularity is comparable to $1/|x|$, and which are regular enough far from the singularity.
	The proof is similar to that of \cite[Proposition~2.8]{OurmieresBonafosVega}, and thus we will omit some details. In particular, we will only address the proof of the proposition for 
	$|\lambda|\geq m$; the proof for $|\lambda|<m$ follows by similar arguments.
	Through the proof we will make use of the Sobolev space $H^1 (\partial \Omega)^2$ as well as its (continuous) dual $H^{-1} (\partial \Omega)^2$, which are defined in the standard way thanks to the fact that $\Omega$ has a $C^2$ boundary.

	Given $b\in \R$, consider the bounded operators in $L^2(\partial \Omega)^2$ defined, for $u \in L^2(\partial \Omega)^2$, by
	$$
	\mathcal{T}_b u (x):= \int_{\partial \Omega} T_b(x-y) u(y) \,d \upsigma (y) \quad \text{with}\quad   T_b (x) := \dfrac{e^{-i b |x|}}{|x|},
	$$
	and 
	$$
	\mathcal{S}_b u (x):= \lim_{\epsilon \downarrow 0} 
	\int_{\{y\in\partial \Omega:\, |x-y|> \epsilon\}}  S_b(x-y) u(y) \,d \upsigma (y) \quad \text{with}\quad   S_b (x) := e^{-i b |x|} (1 + i b |x|)i {\sigma} \cdot \dfrac{x}{|x|^3}.
	$$
	It is easy to check that
	$
	\overline{ T_b (-x)} = (T_{-b} (x))^\intercal$  and $\overline{ S_b (-x)} = (S_{-b} (x))^\intercal$. Thus,
	\begin{equation}
		\label{Eq:Duality1}
		(\mathcal{T}_b)^* = \mathcal{T}_{-b}  \quad \text{and} \quad (\mathcal{S}_b)^* = \mathcal{S}_{-b}.
	\end{equation}
	Let us also define
	$
	\mathcal{A}_b := \{\mathcal{S}_b,  \sigma \cdot\nu\}.
	$
	By \eqref{Eq:Duality1}, we have
	\begin{equation}
		\label{Eq:Duality2}
		(\mathcal{A}_b)^* = \mathcal{A}_{-b}.
	\end{equation}
	In addition, $\mathcal{A}_b$ can be written as the integral operator
	$$
	\mathcal{A}_b u (x) = \int_{\partial \Omega} A_b(x,y) u(y) \,d \upsigma (y)
	$$
	with kernel
	$$
	A_b (x,y) := e^{-i b |x-y|} (1 + i b |x-y|)i\Big[({\sigma} \cdot\nu(x) ) \Big({\sigma} \cdot \frac{x-y}{|x-y|^3}\Big) 
	+ \Big( {\sigma} \cdot \frac{x-y}{|x-y|^3} \Big)({\sigma} \cdot\nu(y) )\Big].
	$$
	As we will see from \eqref{Eq:GrowthKernels} below, the singularity of the kernel $A_b$ does not require to write 
	the integral operator $\mathcal{A}_b$ as an integral in the principal value sense.

	Note that the operators $\mathcal{T}_b$ and $\mathcal{A}_b$ coincide, respectively, with $K_\lambda$ and $\{W_\lambda, {\sigma} \cdot\nu \}$ (up to a multiplicative constant) for $ b= \sqrt{\lambda^2 - m^2}$.
	We will establish the regularity estimates for $\mathcal{T}_b$ and $\mathcal{A}_b$, since the notation is slightly more convenient thanks to \eqref{Eq:Duality1} and \eqref{Eq:Duality2}.
	
	First, let us prove that  $\mathcal{T}_b$ and $\mathcal{A}_b$ are bounded operators from $L^2(\partial \Omega)^2$ into $H^1 (\partial \Omega)^2$ for every $b\in \R$, with a norm uniformly controlled for all $b$ such that $|b|^2<C_\circ$.
	To prove it for $\mathcal{T}_b$, note that its kernel $T_b$ can be written as
	\begin{equation}
		T_b (x-y) = \dfrac{1}{|x-y|} + h_b(x-y)
	\end{equation}
	for some smooth function $h_b$.
	This entails that the kernel $T_b$ is pseudo-homogeneous of class $-1$ in the sense of \cite[\S 4.3.3]{Nedelec}, and thus \cite[Theorem~4.3.2]{Nedelec} yields that  $\mathcal{T}_b$ is bounded from $L^2(\partial \Omega)^2$ into $H^1 (\partial \Omega)^2$.
	Note that the dependence on $b$ in the norm of the operator comes from the kernel $h_b$, which has been obtained after considering the Taylor expansion of $e^{ib |x-y|}$.
	Therefore, it follows readily that it can be uniformly controlled for all $b$ such that $|b|^2<C_\circ$.
	The result for $\mathcal{A}_b$ is obtained using the same argument, provided that we show that	
	\begin{equation}
		\label{Eq:GrowthKernels}
		\widetilde{A}_b(x,y) := ({\sigma} \cdot\nu(x) ) \Big({\sigma} \cdot \frac{x-y}{|x-y|^3}\Big) 
		+ \Big( {\sigma} \cdot \frac{x-y}{|x-y|^3} \Big)({\sigma} \cdot\nu(y) ) = O(|x-y|^{-1}) \quad \text{as $|x-y|\to 0$.}
	\end{equation}
	To see this, note first that, for every $p,q\in\R^3$, 
	\begin{equation}
		\begin{split}
			(\sigma\cdot p)(\sigma\cdot q)
			&=p\cdot q+i\sigma\cdot(p\times q)
			=-q\cdot p-i\sigma\cdot(q\times p)+2p\cdot q
			=-(\sigma\cdot q)(\sigma\cdot p)+2p\cdot q.
		\end{split}
	\end{equation}
	Thus
	$$
	({\sigma} \cdot\nu(x)) ({\sigma} \cdot (x-y)) = - ({\sigma} \cdot (x-y) ) ({\sigma} \cdot\nu(x)) + 2 \nu(x) \cdot (x-y),
	$$
	which yields
	$$
	\widetilde{A}_b(x,y) =  \Big({\sigma} \cdot \dfrac{x-y} {|x-y|^3}\Big) \big({\sigma} \cdot (\nu(y) - \nu(x)) \big) + 2 \nu(x) \cdot \dfrac{x-y}{|x-y|^3}.
	$$
	Then, \eqref{Eq:GrowthKernels} follows by using that 
	$|\nu(x) - \nu(y)| = O(|x-y|)$ and $ |\nu(x) \cdot (x-y)| = O(|x-y|^2)$ because $\partial\Omega$ is of class $C^2$; see~\cite[Example~4.5]{Nedelec} or \cite[Lemma~3.15]{Folland}.
	
	Finally, since $\mathcal{T}_b$ and $\mathcal{A}_b$ are linear bounded operators from $L^2(\partial \Omega)^2$ into $H^1 (\partial \Omega)^2$ for every $b\in \R$, by \eqref{Eq:Duality1}, \eqref{Eq:Duality2}, and duality it follows that $\mathcal{T}_b$ and $\mathcal{A}_b$ are also bounded from $H^{-1} (\partial \Omega)^2$ into $L^2(\partial \Omega)^2$, with the same norm.
	From this, we obtain the desired estimates by using classical interpolation results; see for instance \cite[Propositions~2.1.62 and 2.4.3]{Sauter-Schwab}.
\end{proof}

We finish this section by showing a reformulation of the eigenvalue equation
\begin{equation}
	\Dirac_\tau\varphi=\lambda\varphi,\quad\varphi\in\mathrm{Dom}(\Dirac_\tau)
\end{equation}
in terms of boundary integral equations which involve the components of the trace of
$\varphi$ on $\partial\Omega$ and the operators $K_\lambda$ and $W_\lambda$ introduced in \eqref{Eq:DefK}. Despite that $\varphi$ belongs to $H^1(\Omega)^4$ by assumption, and thus the trace of its components belong to $H^{1/2}(\partial\Omega)^2$, we will see that it is equivalent to pose the corresponding boundary integral equations in the larger space $L^2(\partial\Omega)^2$. This is because they indeed force the solution to belong to $H^{1/2}(\partial\Omega)^2$ thanks to \Cref{Prop:RegularityKW}, as we will see.

\begin{proposition}
	\label{Lemma:BoundaryPb}
	Let $\tau,\lambda\in\R$, let $\varphi:\Omega\to\C^4$, and let $\Phi_\lambda$ be as in \eqref{Eq:DefPhiMaj}. Then, the following are equivalent:
	\begin{itemize}
		\item[$(i)$] $\varphi\in\mathrm{Dom}(\Dirac_\tau)$ and 
		$\Dirac_\tau\varphi=\lambda\varphi$.
		\item[$(ii)$] $\varphi=\Phi_\lambda(i (\alpha \cdot\nu)g)$ in $\Omega$, where
		\begin{equation}\label{bdy.components.rel}
			g = \begin{pmatrix} u \\ i e^\tau ({\sigma}\cdot\nu) u \end{pmatrix} 
		\end{equation}
		for some $u\in L^2(\partial \Omega)^2$ such that 
		\begin{equation}
			\label{Eq:BoundaryPb}
			\begin{split}
				u & =  \big(2i W_\lambda  ({\sigma}\cdot\nu) - 2(\lambda+m) e^\tau K_\lambda\big) u\quad\text{and}\\
				u & =  \big( 2 i ({\sigma}\cdot\nu)  W_\lambda  + 2(\lambda-m)e^{-\tau} ({\sigma}\cdot\nu )  K_\lambda ({\sigma}\cdot\nu) \big) u 
			\end{split}
		\end{equation}
		in $L^2(\partial\Omega)^2$.
	\end{itemize}
	In addition, if $\varphi$ is as in $(ii)$, then	$\varphi=g$ in $L^2(\partial\Omega)^4$.
\end{proposition}

To prove \Cref{Lemma:BoundaryPb} we will use the following regularity result, which  will be also used in the compactness arguments carried out in \Cref{Sec:Parametrization+Monotonicity,Subsec:EigLimits}.
It will be specially useful when studying the asymptotic behavior of the eigenvalue curves as $\tau\to \pm \infty$, since a precise control over the dependence on $\tau$ in the constants involved in the estimates is needed.

\begin{lemma}
	\label{Lemma:BoundaryPb.2}	
	Let $u\in L^2(\partial \Omega)^2$ be a solution to \eqref{Eq:BoundaryPb}, and set $v=i e^\tau ({\sigma}\cdot\nu) u$. Then,
	\begin{eqnarray}
		u\!\!\! &=& \!\!\! \big(i\{W_\lambda , {\sigma}\cdot\nu\}   - (\lambda+m) e^\tau K_\lambda  + (\lambda-m)e^{-\tau} ({\sigma}\cdot\nu)  K_\lambda ({\sigma}\cdot\nu )\big) u
		\quad\text{and}\label{Eq:BoundaryPbComm}\\
		v\!\!\! &=& \!\!\!\big(i \{W_\lambda , {\sigma}\cdot\nu\}   - (\lambda+m) e^\tau ({\sigma}\cdot\nu)  K_\lambda ({\sigma}\cdot\nu )  + (\lambda-m)e^{-\tau} K_\lambda \big)v\label{Eq:BoundaryPbCommv}
	\end{eqnarray}
	$\upsigma$-a.e.\! on $\partial\Omega$. As a consequence, $u,v\in H^{1/2}(\partial\Omega)^2$, and 
	\begin{equation}\label{Eq:BoundaryPbComm.UV}
		\begin{split}
			\|u\|_{H^{1/2}(\partial\Omega)^2}&\leq 
			C(1+ |\lambda+m| e^\tau+|\lambda-m|e^{-\tau})C_\lambda
			\norm{u}_{H^{-1/2}(\partial \Omega)^2},\\
			\|v\|_{H^{1/2}(\partial\Omega)^2}&\leq 
			C(1+ |\lambda+m| e^\tau+|\lambda-m|e^{-\tau})C_\lambda
			\norm{v}_{H^{-1/2}(\partial \Omega)^2},
		\end{split}
	\end{equation}
	where $C>0$ only depends on $\Omega$, and $C_\lambda$ is as in \Cref{Prop:RegularityKW}.
\end{lemma}

\begin{proof}
	By adding the two equations from \eqref{Eq:BoundaryPb} we directly get \eqref{Eq:BoundaryPbComm} in the $L^2(\partial \Omega)^2$ sense, and therefore also $\upsigma$-a.e.\! on $\partial\Omega$. Then, plugging $v = i e^\tau ({\sigma}\cdot\nu) u$ in \eqref{Eq:BoundaryPbComm} we obtain \eqref{Eq:BoundaryPbCommv}.
	
	We now prove that \eqref{Eq:BoundaryPbComm} yields $u\in H^{1/2}(\partial\Omega)^2$ and the estimate for $\|u\|_{H^{1/2}(\partial\Omega)^2}$; the proof for $v$ is analogous. Using that $\nu$ is of class $C^1$ on $\partial\Omega$ (which gives that ${\sigma}\cdot\nu$ is a bounded operator in $H^{\pm1/2}(\partial \Omega)^2$; see \cite[Lemma A.2]{BehrndtHolzmannMas} for the case $H^{1/2}(\partial \Omega)^2$) and \Cref{Prop:RegularityKW}, we have
	\begin{equation}
		\begin{split}
			\|u\|_{H^{1/2}(\partial\Omega)^2}&\leq 
			\|\{W_\lambda , {\sigma}\cdot\nu\}u\|_{H^{1/2}(\partial\Omega)^2}   + |\lambda+m| e^\tau 
			\|K_\lambda u\|_{H^{1/2}(\partial\Omega)^2}\\ 
			&\quad+C|\lambda-m|e^{-\tau} 
			\|  K_\lambda ({\sigma}\cdot\nu ) u\|_{H^{1/2}(\partial\Omega)^2}\\
			&\leq C_\lambda \norm{u}_{H^{-1/2}(\partial \Omega)^2}
			+ |\lambda+m| e^\tau C_\lambda\norm{u}_{H^{-1/2}(\partial \Omega)^2}\\
			&\quad+C|\lambda-m|e^{-\tau} C_\lambda\norm{({\sigma}\cdot\nu ) u}_{H^{-1/2}(\partial \Omega)^2}.
		\end{split}
	\end{equation}
\end{proof}

Using the previous result we can now establish \Cref{Lemma:BoundaryPb}.

\begin{proof}[Proof of \Cref{Lemma:BoundaryPb}]
	We first show that $(i)$ implies $(ii)$. Let 
	$\varphi\in\mathrm{Dom}(\Dirac_\tau)$ such that
	$\Dirac_\tau \varphi=\lambda\varphi$. Then $\varphi\in H^1(\Omega)^4\subset H^{1/2}(\partial\Omega)^4\subset L^2(\partial \Omega)^4$, $\varphi = \Phi_\lambda (i (\alpha \cdot\nu)\varphi)$ in $\Omega$ by \Cref{Lemma:ReproductionFormula}, and
	$\varphi = i(\sinh \tau  - \cosh \tau \, \beta) (\alpha \cdot\nu ) \varphi$ on $\partial\Omega$.
	Writing this boundary equation in terms of the components of $\varphi$, we get
	\begin{equation}\label{comput.compon.K.W:eq2}  \begin{split}
			\begin{pmatrix} u \\ v \end{pmatrix}&:=\varphi 
			= i(\sinh \tau  - \cosh \tau \, \beta) (\alpha \cdot\nu ) \varphi
			\\
			&= i \begin{pmatrix} - e^{-\tau}  & 0 \\ 0 &  e^\tau  \end{pmatrix} \begin{pmatrix} 0 & {\sigma}\cdot\nu \\ {\sigma}\cdot\nu  & 0 \end{pmatrix}   \begin{pmatrix} u \\ v \end{pmatrix}  = i \begin{pmatrix} - e^{-\tau} ({\sigma}\cdot\nu ) v \\ e^\tau ({\sigma}\cdot\nu ) u \end{pmatrix}.
	\end{split} \end{equation}
	This shows \eqref{bdy.components.rel}. Hence, it only remains to prove that $u$ satisfies \eqref{Eq:BoundaryPb}.
	Thanks to \Cref{Lemma:PropertiesOfC}, we have $\varphi = 2 i \ccal_\lambda (\alpha \cdot\nu) \varphi$ on $\partial \Omega$.
	Writing this equation in terms of $u$ and $v$ using \eqref{Eq:DecompCa} we get
	\begin{equation}\label{comput.compon.K.W:eq1} \begin{split}
			\begin{pmatrix} u \\ v \end{pmatrix}  
			& =2 i \begin{pmatrix}
				(\lambda+m) K_\lambda & W_\lambda \\
				W_\lambda & (\lambda -m) K_\lambda
			\end{pmatrix}
			\begin{pmatrix} 0 & {\sigma}\cdot\nu \\ {\sigma}\cdot\nu  & 0 \end{pmatrix}   \begin{pmatrix} u \\ v \end{pmatrix} \\
			&= 2 i \begin{pmatrix}
				(\lambda+m) K_\lambda ({\sigma}\cdot\nu) v + W_\lambda ({\sigma}\cdot\nu )u \\
				W_\lambda ({\sigma}\cdot\nu) v +  (\lambda -m) K_\lambda ({\sigma}\cdot\nu) u
			\end{pmatrix} 
			= \begin{pmatrix}
				- 2 (\lambda+m) e^\tau  K_\lambda  u  + 2 i \, W_\lambda ({\sigma}\cdot\nu )u \\
				- 2 e^\tau W_\lambda u + 2 i  (\lambda -m) K_\lambda ({\sigma}\cdot\nu) u
			\end{pmatrix},
	\end{split} \end{equation}
	where in this last equality we used the identity 
	$v=i e^\tau ({\sigma}\cdot\nu) u $  proved in \eqref{comput.compon.K.W:eq2}.
	Therefore, the first equation in \eqref{Eq:BoundaryPb} is established.
	To prove the second one we simply have to multiply the equation
	$$
	i e^\tau ({\sigma}\cdot\nu) \,u =v= - 2 e^\tau W_\lambda u + 2 i  (\lambda -m) K_\lambda ({\sigma}\cdot\nu) u
	$$
	by $-i e^{-\tau}  ({\sigma}\cdot\nu)$ from the left.
	
	We now show that $(ii)$ implies $(i)$. Let $u\in L^2(\partial \Omega)^2$ be a solution to \eqref{Eq:BoundaryPb}, set 
	$v=i e^\tau ({\sigma}\cdot\nu) u$, $g$ as in \eqref{bdy.components.rel}, and define
	$\varphi=\Phi_\lambda(i (\alpha \cdot\nu)g)$ in $\Omega$. We want to prove that $\varphi\in\mathrm{Dom}(\Dirac_\tau)$ and that $\Dirac_\tau\varphi=\lambda\varphi$. Note that the eigenvalue equation $\Dirac_\tau\varphi=\lambda\varphi$ is automatically satisfied once we know that $\varphi\in\mathrm{Dom}(\Dirac_\tau)$, since we have set $\varphi=\Phi_\lambda(i (\alpha \cdot\nu)g)$ and $\phi_\lambda$ is a fundamental solution of $\Dirac-\lambda$. Therefore, it is enough to prove that $\varphi\in\mathrm{Dom}(\Dirac_\tau)$. 
	For this purpose, let us first show that $\varphi=g$ on $\partial\Omega$, which is the last statement of the proposition. 
	Since $u$ solves \eqref{Eq:BoundaryPb} and we have set $v=i e^\tau ({\sigma}\cdot\nu) u$, \eqref{comput.compon.K.W:eq1} yields $g=2i\ccal_\lambda(\alpha\cdot\nu)g$ in $L^2(\partial\Omega)^4$. Then, if we take traces on both sides of the identity $\varphi=\Phi_\lambda(i (\alpha \cdot\nu)g)$, \Cref{Lemma:PropertiesOfC} gives
	$$\varphi=- \dfrac{i}{2} (\alpha \cdot\nu) i (\alpha \cdot\nu)g + \ccal_\lambda i (\alpha \cdot\nu)g
	=\frac{1}{2}g+\ccal_\lambda i (\alpha \cdot\nu)g
	=\frac{1}{2}g+\frac{1}{2}g=g\quad\text{on $\partial\Omega$},$$
	as desired. At this point, that $\varphi$ solves the boundary equation $\varphi = i(\sinh \tau  - \cosh \tau \, \beta) (\alpha \cdot\nu ) \varphi$ on $\partial\Omega$ is straightforward arguing as in \eqref{comput.compon.K.W:eq2}. 
	
	It only remains to prove that $\varphi\in H^1(\Omega)^4$. We already know that
	$\varphi = g=(u,v)^\intercal$
	on $\partial\Omega$. Abusing notation, let us also denote by $u$ and $v$ the components of $\varphi$ as functions defined on $\Omega$. 
	It is known that $\Phi_\lambda$ is a bounded operator from $L^2(\partial\Omega)^4$ into $L^2(\Omega)^4$; one can argue as in the proof of \cite[Lemma 2.1]{AMV1} but using that $\Omega$ is bounded instead of the exponential decay assumed there. Hence, 
	$\varphi=\Phi_\lambda(i (\alpha \cdot\nu)g)\in L^2(\Omega)^4$, which means that $u,v\in L^2(\Omega)^2$. Then, using the equation 
	$\Dirac\varphi=\lambda\varphi$ once written component-wise as in \eqref{Eq:EigenvaluePbUV.1}, we deduce that 
	${\sigma} \cdot \nabla u, {\sigma} \cdot \nabla v\in L^2(\Omega)^2$. Note also that $u,v\in H^{1/2}(\partial\Omega)^2$ by \Cref{Lemma:BoundaryPb.2}. Therefore, thanks to \Cref{Lemma:SigmaGradEstimates}~$(ii)$, we conclude that
	$ u,v\in{H^1(\Omega)^2}$, that is, $ \varphi\in{H^1(\Omega)^4}$.
\end{proof}

\section{Parametrization and monotonicity} \label{Sec:Parametrization+Monotonicity}

In this section we study the dependence of the eigenvalues of $\Dirac_\tau$, and the associated eigenfunctions, on the parameter $\tau$, establishing \Cref{Th:ParamEigenvalues}. 
Note that even if the dependence on $\tau$ in $\Dirac_\tau$ is expressed through the smooth functions $\tau\mapsto\sinh\tau$ and $\tau \mapsto\cosh\tau$, the proof of \Cref{Th:ParamEigenvalues} is not straightforward and requires some ingredients.
First, we will prove that the eigenvalues can be parametrized locally; see \Cref{Prop:LocalParamEig} below.
Once this is done, we will show that the local parametrization can be extended to a global one. 
To accomplish this, we obtain an explicit formula for the derivative of the eigenvalues with respect to $\tau$; see \Cref{Prop:DerivativeEigenvalues}.
It will provide a growth estimate for the eigenvalue curves which will be crucial to eventually establish \Cref{Th:ParamEigenvalues}.

Let us first prove that the eigenvalues of the operator $\Dirac_\tau$  can be parametrized locally.
The dependence of $\Dirac_\tau$ on $\tau$ appears in the boundary conditions, and hence, in the domain of definition of the operator.
Therefore, we cannot apply directly the very extensive theory of \cite{Kato} devoted to the so-called holomorphic families of operators of type (A) ---in which the domain of the operators is independent of the parameter.
The usual strategy in these cases consists of passing to the weak formulation of the problem, constructing the associated bilinear form in which one of its terms account for the boundary condition ---and thus the domain of definition may be independent of the parameter.
However, this method does not work in our context: the theory developed in \cite{Kato} is devoted to sectorial operators and this is not our case (the spectrum of $\Dirac_{\tau}$ cannot be included into a sector of the complex plane, since it accumulates at $\pm \infty$).

Despite these difficulties, we can still use some of the theory developed in \cite{Kato}, since it can be shown that the resolvent of $\Dirac_\tau$ is analytic in $\tau\in\R$ (using crucially a Krein-type resolvent formula from \cite{BehrndtHolzmannMas}). 
Once this is proved, it readily follows that the eigenvalues can be locally parametrized, as stated in the following result.

\begin{lemma}
	\label{Prop:LocalParamEig}
	Let $\tau_0 \in \R$ and let $\lambda_\star$ be an eigenvalue of $\Dirac_{\tau_0}$ with multiplicity $\mu_\star$.
	Then, $\lambda_\star$ splits into one or several eigenvalue curves $\lambda_k(\tau)$ (with $k= 1,\ldots,k_\star$ for some $k_\star\leq \mu_\star$) which are real analytic close to $\tau = \tau_0$, and such that $\lambda_k(\tau) \in \sigma(\Dirac_\tau)$.
	Each of these eigenvalues has constant multiplicity $\mu_k$ and the sum of all the multiplicities is equal to $\mu_\star$.
	Moreover, for each $k$ there exist $\mu_k$ linearly independent  eigenfunctions $\varphi_k^j (\tau)$ associated to $\lambda_k(\tau)$ which depend analytically on $\tau$ and such that $\{ \varphi_k^j (\tau_0) : k = 1,\ldots,k_\star, \ j = 1,\ldots,\mu_k\}$ forms an orthogonal basis of the eigenspace associated to $\lambda_\star$.		
\end{lemma}

\begin{proof}
	Although in the rest of the article we will consider $\tau \in \R$, in this proof we may assume that $\tau \in \C$, since the results from perturbation theory that we will use are more naturally presented in this setting.
	Then, the result will follow restricting $\tau$ to $\R$.
	
	We claim that for $\zeta \in \C\setminus \R$, the resolvent operator $(\Dirac_\tau - \zeta)^{-1}$ is bounded and  holomorphic in $\tau$ in a neighborhood of $\tau_0$.
	To prove this claim, we will use a resolvent formula from \cite{BehrndtHolzmannMas} which we recall now.
	For $\zeta \in \C \setminus \R$, consider the $\gamma$-field $\gamma(\zeta)$ and the Weyl function $M(\zeta)$ associated to the quasi boundary triple introduced in \cite{BehrndtHolzmannMas} for $\Dirac$ acting on $H^1(\Omega)^4$. These operators are defined in terms of $\Phi_\lambda$ and $\ccal_\lambda$ from \Cref{Subsec:BIOperators} taking $\lambda=\zeta\in \C \setminus \R$.
	Their explicit expression  can be found in \cite[Proposition~4.2]{BehrndtHolzmannMas}, but for our purposes it suffices to use that $\gamma(\zeta) : \mathcal{G}_\Omega^{1/2} \to H^1(\Omega)^4$, $M(\zeta) : \mathcal{G}_\Omega^{1/2} \to \mathcal{G}_\Omega^{1/2} $, and $\gamma(\overline\zeta)^* :  L^2(\Omega)^4 \to \mathcal{G}_\Omega^{1/2}  $ are  bounded operators, where $\mathcal{G}_\Omega^{1/2} := \big( 1 + i \beta (\alpha \cdot\nu)\big) \left(H^{1/2}(\partial \Omega)^4 \right)$.
	From \cite[Theorem~5.9 and Proposition~5.15]{BehrndtHolzmannMas}, for every  $\zeta \in \rho(\Dirac_\tau)\cap\rho(\Dirac_0)$ the following resolvent formula holds:
	\begin{equation}
		\label{Eq:Resolvent}
		(\Dirac_\tau-\zeta)^{-1} = \big(\Dirac_0 - \zeta\big)^{-1} 
		+ \gamma(\zeta) \big( I_4 - \omega(\tau) M(\zeta) \big)^{-1} \omega(\tau) \gamma(\overline{\zeta})^*, \quad \text{where } \
		\omega (\tau) := - \dfrac{\sinh \tau}{1 + \cosh \tau}.
	\end{equation}
	Note that $\omega$ is holomorphic in $\tau$ wherever the denominator does not vanish.
	From this and using that $\big(\Dirac_0 - \zeta\big)^{-1} $, $\gamma(\zeta)$, and $\gamma(\overline\zeta)^*$ are bounded, the holomorphy of $(\Dirac_\tau-\zeta)^{-1}$  will follow if we prove that $( I_4 - \omega(\tau) M(\zeta) )^{-1}$ is a bounded operator holomorphic in $\tau$.
	From~\cite[Lemma~4.4]{BehrndtHolzmannMas} it follows that, for every $\omega_0\in \R$ such that $|\omega_0| \neq 1$ and for every $\zeta \in \C \setminus \R$, the operator $I_4 - \omega_0 M(\zeta)$ has a bounded and everywhere defined inverse in $\mathcal{G}_\Omega^{1/2}$.
	Therefore, using this last assertion with $\omega_0 := \omega(\tau_0) \in \R$ (note that $|\omega(\tau_0)|< 1$ for every $\tau_0 \in \R$) together with the identity
	\begin{equation}
		I_4 - \omega(\tau) M(\zeta) = \Big( I_4 - \big(\omega (\tau) - \omega(\tau_0)\big)M(\zeta) \big( I_4 - \omega(\tau_0) M(\zeta) \big)^{-1} \Big) \big( I_4 - \omega(\tau_0) M(\zeta) \big),
	\end{equation}
	and a Neumann series argument, it follows that $( I_4 - \omega(\tau) M(\zeta) )^{-1}$ is a bounded operator holomorphic in $\tau$ for $\tau \in \C$ close enough to $\tau_0 \in \R$; see the comments in \cite[\textrm{VII}-\S1.1]{Kato} for more details.
	Hence, the claim is proved.
	
	Once we know that the resolvent of $\Dirac_\tau$ is bounded and  holomorphic\footnote{From the proof of 
		\cite[Theorem~\textrm{VII}.1.3]{Kato}, it is enough to have  $(\Dirac_\tau - \zeta)^{-1}$ holomorphic in $\tau$ for a single $\zeta \in \rho(\Dirac_0)\cap \rho(\Dirac_\tau)$, thus we can take any $\zeta \in \C\setminus \R$.} in $\tau$, by classical results in perturbation theory \cite[Theorems~\textrm{VII}.1.3 and~\textrm{VII}.1.8]{Kato} it follows that every isolated eigenvalue $\lambda_\star$ of $\Dirac_{\tau_0}$ splits into one or several eigenvalues $\lambda_k(\tau)$ of $\Dirac_\tau$ which depend holomorphically on $\tau$ (for $\tau$ close to  $\tau_0$), and the same is true for the corresponding eigenprojections.
	As a consequence, if $\mu_k$ is the multiplicity of each eigenvalue $\lambda_k(\tau)$, there exist $\mu_k$ linearly independent associated eigenfunctions $\varphi_k^j(\tau)$, with $j=1,\ldots,\mu_k$, which depend holomorphically on $\tau$; see  \cite[\textrm{II}-\S4]{Kato} for more details on this.
	Note that $\mu_k$ is constant for $\tau$ close to  $\tau_0$, and the sum of all these multiplicities is equal to the multiplicity of $\lambda_\star$.
	Moreover, all the functions $\varphi_k^j(\tau_0)$ form an orthogonal basis of the eigenspace associated to $\lambda_\star$.
\end{proof}

\begin{remark}
	\label{Rem:CurvesMultiplicity}
	If $\lambda_\star$ is an eigenvalue of $\Dirac_{\tau_0}$ with multiplicity $\mu_\star$, then using \Cref{Prop:LocalParamEig} we can define  $\mu_\star$ eigenvalue curves (for $\tau$ close to $\tau_0$), each one associated to a single eigenfunction.
	Some of these eigenvalue curves may be equal, accounting for the multiplicity of the corresponding eigenvalue of $\Dirac_{\tau}$.
	This setting will be the appropriate one to establish \Cref{Th:ParamEigenvalues} below. 
\end{remark}

With a local parametrization of the eigenvalue curves at hand, our next result shows their monotonicity with respect to the parameter $\tau$. 
In addition, we also provide an explicit expression for the derivative of the eigenvalues with respect to $\tau$.
Note that the monotone behavior was expected if one looks at the curves plotted in \Cref{Fig:ALL_eigenvalues} for the case of a ball.

\begin{lemma}
	\label{Prop:DerivativeEigenvalues}
	Let $\tau\mapsto\lambda(\tau)\in\sigma(\Dirac_\tau)$ and 
	$\tau\mapsto\varphi_\tau\in \mathrm{Dom}(\Dirac_\tau)\setminus\{0\}$ be differentiable functions on an interval $I\subset\R$, with 
	$\varphi_\tau$ such that 
	$\Dirac_\tau\varphi_\tau=\lambda(\tau)\varphi_\tau$ for all $\tau\in I$. Then 
	\begin{equation}
		\label{Eq:Lambda'Formula}
		\lambda'(\tau) = e^\tau \dfrac{ \norm{ u_\tau}^2_{L^2(\partial \Omega)^2}}{\norm{ \varphi_\tau}^2_{L^2( \Omega)^4}} > 0 
		\quad\text{for all $\tau\in I$,}
	\end{equation}
	where $\varphi_\tau=(u_\tau, v_\tau)^\intercal$ with $u_\tau,v_\tau\in H^1(\Omega)^2$.
	In particular, $\lambda$ is strictly increasing on $I$. Moreover, 
	\begin{equation}
		\label{Eq:Lambda'Formula2}
		\lambda'(\tau) =  (\lambda(\tau) - m) \dfrac{\| u_\tau \|^2_{L^2( \Omega)^2}}{\| \varphi_\tau \|^2_{L^2( \Omega)^4}} - (\lambda(\tau) + m) \dfrac{\|v_\tau \|^2_{L^2(\Omega)^2}}{\| \varphi_\tau \|^2_{L^2( \Omega)^4}}  .
	\end{equation}
\end{lemma}

\begin{proof}
	For simplicity of notation, in the following we will write 
	$\varphi$ and $\varphi'$ instead of $\varphi_\tau$ and
	$\partial_\tau\varphi_\tau$, respectively,
	and analogously for $u_\tau$ and $v_\tau$.
	Differentiating the equation $\Dirac \varphi = \lambda\varphi$ with respect to $\tau$, we get 
	$\Dirac \varphi' = \lambda' \varphi + \lambda \varphi '$.
	Multiplying this equation by $\overline{ \varphi}$ and integrating by parts in $\Omega$ we obtain
	\begin{equation}
		\label{Eq:ProofMonotonicity}
		\begin{split}
			\lambda' \int_\Omega |\varphi|^2 + \lambda \int_\Omega \varphi'\cdot \overline{\varphi} 
			&= \int_\Omega \Dirac \varphi' \cdot\overline{\varphi} 
			= \int_\Omega  \varphi' \cdot\overline{\Dirac \varphi} - i \int_{\partial \Omega} (\alpha \cdot\nu) \varphi' \cdot\overline{\varphi} \,d \upsigma  \\
			&= \lambda \int_\Omega  \varphi'\cdot \overline{\varphi}
			+  \int_{\partial \Omega}  \varphi' \cdot \overline{i (\alpha \cdot\nu) \varphi} \,d \upsigma. 
		\end{split}
	\end{equation}
	Here we used that $\lambda$ is real-valued by \Cref{Lemma:SpectrumGenMIT}, and that $(\alpha \cdot\nu)$ is hermitian.	
	Thus, from \eqref{Eq:ProofMonotonicity} we deduce that
	\begin{equation}
		\label{Eq:Lambda'FormulaProof}
		\lambda' = \| \varphi \|^{-2}_{L^2( \Omega)^4}\int_{\partial \Omega}  \varphi' \cdot \overline{i (\alpha \cdot\nu) \varphi} \,d \upsigma.
	\end{equation}
	Now, using the boundary condition $(\alpha \cdot\nu) \varphi = i(\sinh \tau + \cosh \tau\, \beta) \varphi$ given by the fact that 
	$\varphi\in\mathrm{Dom}(\Dirac_\tau)$ (recall that $(\alpha \cdot\nu)^2 = 1$ and that $\alpha_j\beta = -\beta \alpha_j$ for $j=1,2,3$), we have
	\begin{equation} \label{Eq:Lambda'FormulaProof.1}
		\begin{split}
			\int_{\partial \Omega}  \varphi' \cdot \overline{i (\alpha \cdot\nu) \varphi} \,d \upsigma &= - \int_{\partial \Omega}  \varphi' \cdot(\sinh \tau + \cosh \tau\, \beta)\cdot \overline{\varphi} \,d \upsigma \\
			&= -e^\tau \int_{\partial \Omega}  u' \cdot\overline{u} \,d \upsigma
			+ e^{-\tau} \int_{\partial \Omega}  v' \cdot\overline{v} \,d \upsigma.
	\end{split} \end{equation}
	Let us rewrite the last term on the right-hand side in terms of $u$.
	Recall that the boundary condition defining $\mathrm{Dom}(\Dirac_\tau)$ is expressed component-wise as $v = i e^\tau ({\sigma} \cdot\nu) u$. If we differentiate this identity with respect to $\tau$, we get
	$$
	v' = i e^\tau ({\sigma} \cdot\nu) (u + u'),
	$$
	and thus
	$$
	e^{-\tau} v' \cdot \overline{v} = e^{\tau} ({\sigma} \cdot\nu) (u+u') \cdot\overline{({\sigma} \cdot\nu) u}
	=e^{\tau}(|u|^2+u'\cdot\overline{u}).
	$$
	Therefore,
	$$
	e^{-\tau} \int_{\partial \Omega}  v' \cdot\overline{v} 
	\,d \upsigma
	=   e^\tau \| u \|^2_{L^2(\partial \Omega)^2} + e^\tau \int_{\partial \Omega}  u' \cdot\overline{u} \,d \upsigma,
	$$
	and inserting this into \eqref{Eq:Lambda'FormulaProof.1} we obtain
	$$
	\int_{\partial \Omega}  \varphi' \cdot \overline{i (\alpha \cdot\nu) \varphi} \,d \upsigma  =  e^\tau \| u \|^2_{L^2(\partial \Omega)^2}.
	$$
	Plugging this last expression into \eqref{Eq:Lambda'FormulaProof} we obtain \eqref{Eq:Lambda'Formula}.
	Finally, \eqref{Eq:Lambda'Formula2} is obtained combining \eqref{Eq:Lambda'Formula} with \Cref{Lemma:RelationNormsIntBdry}.
\end{proof}

Once we have established the monotonicity of the eigenvalues, we can finally prove \Cref{Th:ParamEigenvalues}. 
To do it, it will be crucial to use the explicit formula \eqref{Eq:Lambda'Formula2}, which provides a growth estimate for the eigenvalue curves $\tau \mapsto\lambda(\tau)$, preventing them to escape to infinity at finite values of the parameter $\tau$.

\begin{proof}[Proof of \Cref{Th:ParamEigenvalues}]
	Let $\tau_0 \in \R$ and let $\lambda_\star\in \sigma (\Dirac_{\tau_0})$ be an eigenvalue with multiplicity $\mu_\star$.
	Note that thanks to \Cref{Lemma:SpectrumGenMIT}~$(iii)$ we can assume without loss of generality that $\lambda_\star >m$.
	Then, by \Cref{Prop:LocalParamEig} and taking into account \Cref{Rem:CurvesMultiplicity}, there exist $\mu_\star$ real analytic functions $\lambda_1(\tau), \ldots, \lambda_{\mu_\star}(\tau)$ defined in a neighborhood of $\tau_0$, such that $\lambda_j(\tau) \in \sigma(\Dirac_\tau)$ and $\lambda_j(\tau_0) = \lambda_\star$ for $j = 1,\ldots,\mu_\star$. 
	Each of these functions $\lambda_j(\tau)$, some of them possibly equal in a neighborhood of $\tau_0$, is associated to a different eigenfunction which also depends analytically on $\tau$.
	
	Let $\lambda(\tau)$ be one of these analytic eigenvalue curves and let $\varphi_\tau$ be an associated eigenfunction, which can be taken to have norm equal to $1$ in $L^2(\Omega)^4$. 
	The curve $\lambda(\tau)$ can be continued analytically in a maximal interval $I$ in which $\lambda(\tau)$ represents an eigenvalue of $\Dirac_\tau$ with associated eigenfunction $\varphi_\tau$ (this is true even when the graph of $\lambda(\tau)$ crosses the graph of another such eigenvalue curve, see the comments in \cite[\textrm{VII}-\S3.2]{Kato}).
	Our goal is to prove that $I =\R$.
	
	By contradiction, let us assume that $I \neq \R$, and let $\tau^\star$ be a finite end of the maximal interval $I$ ---that is, $\tau^\star \in \partial I$ and $|\tau^\star|<+\infty$.
	Then, we set $\lambda^\star := \lim_{\tau \to  \tau^\star}\lambda(\tau)$, which exists by the monotonicity of $\lambda(\tau)$ shown in \Cref{Prop:DerivativeEigenvalues}.
	We claim that $\lambda^\star $ is finite.
	Indeed, by the formula for $\lambda'(\tau)$ given by \eqref{Eq:Lambda'Formula2} and using that $\lambda(\tau) > m$, we get $(\lambda(\tau) - m)' \leq \lambda(\tau) - m$, and therefore $m < \lambda(\tau) \leq C( 1 + e^\tau)$ for $\tau \in I$, with some constant $C$ depending on $m$ and $\tau_0$.
	As a consequence, since $\tau^\star < +\infty$, $\lambda(\tau)$ cannot tend to $+\infty$ as $\tau \to \tau^\star$.
	
	We shall prove now that $\lambda^\star $ is an eigenvalue of $\Dirac_{\tau^\star}$.
	To do it, recall that the eigenfunctions $\varphi_\tau= (u_\tau, v_\tau)^\intercal$ are normalized in $L^2(\Omega)^4$.
	Our goal is to bound them uniformly in $H^1(\Omega)^4$.
	To accomplish this, note first that by Lemma~\ref{Lemma:SigmaGradEstimates} we have
	$$
	\| u_\tau \|_{H^1(\Omega)^2} \leq C \left (  \|  u_\tau \|_{L^2(\Omega)^2} + \| {\sigma} \cdot\nabla u_\tau \|_{L^2(\Omega)^2}  + \| u_\tau \|_{H^{1/2}(\partial \Omega)^2} \right ).
	$$
	Since $\| \varphi_\tau \|^2_{L^2(\Omega)^4} = 1 $, using the equation $-i {\sigma} \cdot\nabla u_\tau = (\lambda(\tau) + m) v_\tau$ and the bounds $m < \lambda(\tau)\leq \lambda^\star + 1$ for $\tau$ close enough to $\tau^\star$, it follows that
	$$
	\| u_\tau \|_{H^1(\Omega)^2} \leq C \left (1  + \| u_\tau \|_{H^{1/2}(\partial \Omega)^2} \right )
	$$
	for some constant $C$ depending only on $\Omega$, $m$, and $\lambda^\star$.
	It remains to estimate $\| u_\tau \|_{H^{1/2}(\partial \Omega)^2}$.
	To do it, we use \eqref{Eq:BoundaryPbComm.UV} to obtain
	\begin{equation} \begin{split}
			\| u_\tau \|_{H^{1/2}(\partial \Omega)^2} &\leq  C  \|  u_\tau  \|_{H^{-1/2}(\partial \Omega)^2},
	\end{split} \end{equation}
	for some constant $C$ depending only on $\Omega$, $m$, and $\lambda^\star$, provided that $\tau\in I$ is close to $\tau^\star$.	
	Finally, note that, under the normalization hypothesis, by \Cref{Lemma:H1/2Estimate} we have $\|  u_\tau  \|_{H^{-1/2}(\partial \Omega)^2} \leq C$ for some constant depending only on $\Omega$, $m$, and $\lambda^\star$.
	Therefore, for some other constant $C$ depending only on the same quantities, we get 
	$$
	\| u_\tau \|_{H^1(\Omega)^2} \leq C ,
	$$
	provided that $\tau\in I$ is close to $\tau^\star$.
	Using a similar argument for $v_\tau$, we obtain the uniform bound $\| \varphi_\tau \|_{H^1(\Omega)^4} \leq C$ for $\tau\in I$  close to $\tau^\star$. 
	Therefore, by the compact embedding of $H^1(\Omega)^4$ into $L^2(\Omega)^4$ and of $H^{1/2}(\partial\Omega)^4$ into $L^2(\partial\Omega)^4$, and by weak$^*$ compactness on $H^1(\Omega)^4$, we can find a sequence $\{\tau_k\}_{k\in \N}$ with $\lim_{k \uparrow +\infty} \tau_k = \tau^\star$ such that  $\varphi_{\tau_k}$ converges in $L^2(\Omega)^4$ and in $ L^2(\partial \Omega)^4$ to some function $\varphi^\star \ \in H^1(\Omega)^4$ as $k \uparrow +\infty$. 
	Since $\varphi_{\tau_k}$ are normalized in $L^2(\Omega)^4$, we have $\varphi^\star \not \equiv 0$.
	Moreover, by the convergence in  $L^2(\partial\Omega)^4$, it follows readily that $\varphi^\star \in \mathrm{Dom} (\Dirac_{\tau^\star})$.
	Writing the eigenvalue equation for $\varphi_{\tau_k}$ in weak form and taking the limit $k \uparrow+\infty$, we obtain that $\varphi^\star$ solves weakly an eigenvalue equation with associated eigenvalue $\lambda^\star$.
	Since $\varphi^\star \in H^1(\Omega)^4$ (by the weak$^*$ compactness), a standard density argument shows that $ \Dirac \varphi^\star = \lambda^\star \varphi^\star$ in $\Omega$.

	Finally, once we have proven that $\lambda^\star \in \sigma(\Dirac_{\tau^\star})$, we can apply again \Cref{Prop:LocalParamEig} to this eigenvalue.
	By doing this, we extend analytically the curve $\lambda(\tau)$ to a bigger interval, contradicting the maximality of the interval $I$.
	As a consequence, we have shown that any eigenvalue curve $\lambda(\tau)$ can be defined for all $\tau \in \R$ as a real analytic function.

	To conclude the proof, we should show that, for the given $\tau_0 \in \R$, the extensions of all the eigenvalues $\{ \lambda_k (\tau_0)\}_{k \in \Z \setminus \{0 \} }$ exhaust the spectrum of $\Dirac_\tau$ for every $\tau\in \R$. 
	Indeed, if there were some $\tau_\star$ and an eigenvalue $\lambda_\star \in\sigma(\Dirac_{\tau_\star})$ which did not lie on any of the eigenvalue curves found before, then $\lambda_\star$ could itself be extended to an analytic eigenvalue curve on $\R$ by the previous arguments, and in particular we would have an eigenvalue $\lambda_\star(\tau_0)$ not included in  $\{ \lambda_k (\tau_0)\}_{k \in \Z \setminus \{0\}}$, contradicting the assumption that $\{ \lambda_k (\tau_0)\}_{k \in \Z \setminus \{0\}}$ (counting multiplicities) is the totality of the spectrum of $\Dirac_{\tau_0}$.
\end{proof}

\section{Asymptotic behavior} \label{Sec:AsymptoricBehavior}

In this section we study the asymptotic behavior of the eigenvalue curves as $\tau \to \pm \infty$.
First, we will establish \Cref{Th:EigLimits}, regarding the possible limits of the eigenvalue curves, and we will provide a finer description of the curve corresponding to the first positive eigenvalue of $\Dirac_{\tau}$, as stated in \Cref{Th:FirstEigTom}.
Furthermore, we will prove the shape optimization result for large values of $\tau$ stated in \Cref{Th:BallOptimalLargeTau}.
After that, we will focus on the first order asymptotics as $\tau \downarrow - \infty$. 
To do so, we will first need to introduce skew projections onto Hardy spaces, and then we will address the proof of \Cref{Th:Rayleigh.Intro}.

\subsection{Limits of eigenvalue curves as $\tau \to \pm \infty$}
\label{Subsec:EigLimits}
In this section we prove \Cref{Th:EigLimits,Th:FirstEigTom} and \Cref{Th:BallOptimalLargeTau}. We begin by showing two compactness results that will be used in the proof of \Cref{Th:EigLimits}, and also \Cref{Th:Rayleigh.Intro}. Roughly  speaking,  \Cref{Th:EigLimits} is based on the following argument: 
every $\varphi_\tau = (u_\tau, v_\tau)^\intercal\in\mathrm{Dom}(\Dirac_\tau)$ satisfies the boundary condition $v_\tau = i e^\tau ( {\sigma}\cdot\nu)u_\tau$ on $\partial \Omega$. 
If $\varphi_\tau$ is an eigenfunction, as $\tau \downarrow -\infty$ the boundary condition forces the trace of $v_\tau$ on $\partial \Omega$ to vanish, 
and as $\tau \uparrow +\infty$ the boundary condition forces the trace of $u_\tau$ to vanish. Combining this with uniform estimates on the $H^1(\Omega)^2$-norm of the components of $\varphi_\tau$ and a compactness argument, we end up with an eigenfunction of the Dirichlet Laplacian. From here we can then find a set of candidates to be the limit of the eigenvalue associated to $\varphi_\tau$ as $\tau \to \pm\infty$. This, in particular, will lead to the proof of \Cref{Th:EigLimits}.

Let us first address the compactness result for $v_\tau$ as $\tau \downarrow -\infty$. 
The crucial point is to establish a uniform bound for the $H^1(\Omega)^2$-norm of $v_\tau$, \eqref{Eq:H1boundV} below.
The proof of this estimate will be a slight modification of some ideas that were used to prove \Cref{Th:ParamEigenvalues}, but since now we do not have control on $e^{-\tau}$ as $\tau\downarrow-\infty$, we also need to use some estimates for $u_\tau$ and the boundary condition.

\begin{proposition}
	\label{Prop:CompactnessMinusInfty}
	Let $\mathcal{T}=\{\tau_k\}_{k\in\N}\subset\R$ be a sequence such that  $\lim_{k\uparrow+\infty}\tau_k = -\infty$. For every $\tau\in\mathcal{T}$,  
	let $\lambda(\tau)\in\sigma(\Dirac_\tau)\cap(m,+\infty)$  and $\varphi_\tau= (u_\tau, v_\tau)^\intercal\in \mathrm{Dom}(\Dirac_\tau)$ such that 
	$\Dirac_\tau\varphi_\tau=\lambda(\tau)\varphi_\tau$ and
	$\| \varphi_\tau \|_{L^2(\Omega)^4} = 1 $. Assume that
	$\lambda(\tau) \leq C_\circ$ for some $C_\circ>0$ and all 
	$\tau\in\mathcal{T}$.
	Then, 
	\begin{equation}
		\label{Eq:H1boundV}
		\norm{ v_\tau }_{H^1(\Omega)^2} \leq C \quad \text{for all } \tau\in\mathcal{T},
	\end{equation}
	where $C>0$ depends only on $m$, $C_\circ$, $\max\mathcal T$, and $\Omega$.
	
	As a consequence, there exists a subsequence 
	$\{\tau_{k_j}\}_{j\in\N}\subset\mathcal{T}$ for which the limit $$\lambda_\star := \lim_{j \uparrow +\infty} \lambda(\tau_{k_j})\in[m,C_\circ]$$ exists and such that $v_{\tau_{k_j}}$ converges in $L^2(\Omega)^2$ to a function $v_\star\in H^2(\Omega)^2$ satisfying
	\begin{equation}
		\label{Eq:LimitVPb}
		\beqc{\PDEsystem}
		-\Delta v_\star & = & (\lambda_\star^2-m^2) v_\star  & \text{in }  \Omega, \\
		v_\star & = & 0  & \text{on } \partial \Omega.
		\eeqc
	\end{equation}
\end{proposition}

\begin{proof}
	Throughout the proof we will use the letter $C$ to denote different constants depending only on $m$, $C_\circ$, 
	$\max\mathcal T$, and $\Omega$.

	First, note that by 
	\Cref{Lemma:SigmaGradEstimates}~$(ii)$ we have
	$$
	\norm{ v_\tau }_{H^1(\Omega)^2} \leq C \left (  \norm{ v_\tau }_{L^2(\Omega)^2} + \norm{{\sigma} \cdot\nabla v_\tau }_{L^2(\Omega)^2}  + \norm{ v_\tau }_{H^{1/2}(\partial \Omega)^2} \right ).
	$$
	Since $ \norm{ u_\tau }_{L^2(\Omega)^2}^2+ \norm{ v_\tau }_{L^2(\Omega)^2}^2=\| \varphi_\tau \|^2_{L^2(\Omega)^4} = 1 $, using the equation $-i{\sigma} \cdot\nabla v_\tau = (\lambda(\tau) - m) u_\tau$ from \eqref{Eq:EigenvaluePbUV.1} and the upper bound $\lambda(\tau)\leq C_\circ$, it follows that
	$$
	\norm{ v_\tau }_{H^1(\Omega)^2} \leq C \left (1  + \norm{ v_\tau }_{H^{1/2}(\partial \Omega)^2} \right ) \quad \text{for all } \tau\in\mathcal{T}.
	$$
	Therefore, to prove \eqref{Eq:H1boundV} we only need to estimate $\norm{ v_\tau }_{H^{1/2}(\partial \Omega)^2}$ uniformly in $\tau\in\mathcal{T}$. 		
	
	From the boundary condition $v_\tau = i e^\tau ({\sigma} \cdot\nu) u_\tau$, and since $\nu$ is of class $C^1$ on 
	$\partial\Omega$ (which yields that ${\sigma}\cdot\nu$ is a bounded operator in $H^{1/2}(\partial \Omega)^2$), we see that
	\begin{equation}
		\label{Eq:H12normsuv}
		\norm{v_\tau }_{H^{1/2}(\partial \Omega)^2} \leq C e^\tau \norm{u_\tau}_{H^{1/2}(\partial \Omega)^2}.
	\end{equation}
	Looking at $u_\tau$, if we combine \Cref{Lemma:BoundaryPb} with \eqref{Eq:BoundaryPbComm.UV}
	we get
	$$
	\|u_\tau\|_{H^{1/2}(\partial\Omega)^2}\leq 
	C\left(1+ |\lambda(\tau)+m| e^\tau+|\lambda(\tau)-m|e^{-\tau}\right)C_{\lambda(\tau)}
	\norm{u_\tau}_{H^{-1/2}(\partial \Omega)^2}.$$
	Now, using that $\lambda(\tau)\leq C_\circ$, it follows that
	$$
	\| u_\tau \|_{H^{1/2}(\partial \Omega)^2} \leq C (1 + e^{-\tau}) \| u_\tau \|_{H^{-1/2}(\partial \Omega)^2} \quad \text{for all } \tau\in\mathcal{T}.
	$$
	Here we have used once again that $\nu$ is of class $C^1$ on 
	$\partial\Omega$.
	Applying \Cref{Lemma:H1/2Estimate} and using that $\| \varphi_\tau \|_{L^2(\Omega)^4} = 1 $, we obtain
	$\| u_\tau \|_{H^{1/2}(\partial \Omega)^2} \leq C (1 + e^{-\tau})$ for all $\tau\in\mathcal{T}$.
	Combining this with \eqref{Eq:H12normsuv} we deduce that $\norm{ v_\tau }_{H^{1/2}(\partial \Omega)^2}\leq C$ for all $\tau\in\mathcal{T}$. This concludes the proof of \eqref{Eq:H1boundV}.
	
	Let us now address the proof of the statement regarding the function $v_\star$. Firstly, since $\lambda(\tau)\in(m,C_\circ]$ for all $\tau\in\mathcal{T}$, there exists a subsequence 
	$\mathcal{T}'=\{\tau_{k_j}\}_{j\in\N}\subset\mathcal{T}$ such that the limit $\lambda_\star := \lim_{j \uparrow +\infty} \lambda(\tau_{k_j})$ exists and satisfies $m\leq\lambda_\star\leq C_\circ$.
	Secondly, the normalization $\| \varphi_\tau \|_{L^2(\Omega)^4} = 1 $ together with \Cref{Lemma:RelationNormsIntBdry} yield $e^{-\tau} \| v_\tau \|^2_{L^2(\partial \Omega)^2}\leq C$ for all $\tau\in\mathcal{T}$, which leads to 
	\begin{equation}\label{lim.v.-infty.l2.zero}
		\lim_{j\uparrow+\infty}\| v_{\tau_{k_j}}\|_{L^2(\partial \Omega)^2}=0.
	\end{equation}
	Now, combining these two ingredients with the uniform estimate \eqref{Eq:H1boundV}, we can show the existence of $v_\star$. More precisely, by the compact embedding of $H^1(\Omega)^2$ into $L^2(\Omega)^2$ and of $H^{1/2}(\partial\Omega)^2$ into $L^2(\partial\Omega)^2$, and by weak$^*$ compactness on $H^1(\Omega)^2$, we can find a subsequence of $\mathcal{T}'$, which we denote again by
	$\{\tau_{k_j}\}_{j\in\N}$, for which $v_{\tau_{k_j}}$ converges in $L^2(\Omega)^2$ to a function $v_\star \in H^1_0(\Omega)^2$ ($v_\star$ has zero trace thanks to \eqref{lim.v.-infty.l2.zero}) satisfying
	$$-\Delta v_\star = (\lambda_\star^2-m^2) v_\star$$
	in the weak sense in $\Omega$ ---recall that $-\Delta v_\tau = (\lambda(\tau)^2 - m^2) v_\tau$ in $\Omega$ by \eqref{Eq:HelmholtzUV}.  	
	Finally, standard elliptic estimates show that $v_\star\in H^2(\Omega)^2$; see \cite[Theorem 4 in \S 6.3.2]{Evans} for example.
\end{proof}

We now address the compactness result related to the upper component $u_\tau$ of the eigenfunction $\varphi_\tau$ as $\tau \uparrow +\infty$. That is, the analogue of \Cref{Prop:CompactnessMinusInfty} for $\tau \uparrow +\infty$.

\begin{proposition}
	\label{Prop:CompactnessPlusInfty}
	Let $\mathcal{T}=\{\tau_k\}_{k\in\N}\subset\R$ be a sequence such that  $\lim_{k\uparrow+\infty}\tau_k = +\infty$. For every $\tau\in\mathcal{T}$,  
	let $\lambda(\tau)\in\sigma(\Dirac_\tau)\cap(m,+\infty)$  and $\varphi_\tau= (u_\tau, v_\tau)^\intercal\in \mathrm{Dom}(\Dirac_\tau)$ such that 
	$\Dirac_\tau\varphi_\tau=\lambda(\tau)\varphi_\tau$ and
	$\| \varphi_\tau \|_{L^2(\Omega)^4} = 1 $. Assume that
	$\lambda(\tau) \leq C_\circ$ for some $C_\circ>0$ and all 
	$\tau\in\mathcal{T}$.
	Then, 
	\begin{equation}
		\label{Eq:H1boundu.tau}
		\norm{ u_\tau }_{H^1(\Omega)^2} \leq C \quad \text{for all } \tau\in\mathcal{T},
	\end{equation}
	where $C>0$ depends only on $m$, $C_\circ$, $\min\mathcal T$, and $\Omega$.
	
	As a consequence, there exists a subsequence 
	$\{\tau_{k_j}\}_{j\in\N}\subset\mathcal{T}$ for which the limit $$\lambda^\star := \lim_{j \uparrow +\infty} \lambda(\tau_{k_j})\in[m,C_\circ]$$ exists and such that $u_{\tau_{k_j}}$ converges in $L^2(\Omega)^2$ to a function $u^\star\in H^2(\Omega)^2$ satisfying
	\begin{equation}
		\label{Eq:LimitUPb}
		\beqc{\PDEsystem}
		-\Delta u^\star & = & ((\lambda^\star)^2-m^2) u^\star  & \text{in }  \Omega, \\
		u^\star & = & 0  & \text{on } \partial \Omega,
		\eeqc
	\end{equation}
\end{proposition}

\begin{proof}
	The proof follows the same lines as the one of \Cref{Prop:CompactnessMinusInfty}.
	By 
	\Cref{Lemma:SigmaGradEstimates}~$(ii)$,
	$$
	\norm{ u_\tau }_{H^1(\Omega)^2} \leq C \left (  \norm{ u_\tau }_{L^2(\Omega)^2} + \norm{{\sigma} \cdot\nabla u_\tau }_{L^2(\Omega)^2}  + \norm{ u_\tau }_{H^{1/2}(\partial \Omega)^2} \right ).
	$$
	Since $\| \varphi_\tau \|_{L^2(\Omega)^4} = 1 $, using the equation $-i{\sigma} \cdot\nabla u_\tau = (\lambda(\tau) + m) v_\tau$ from \eqref{Eq:EigenvaluePbUV.1} and the upper bound $\lambda(\tau)\leq C_\circ$, we have
	$
	\norm{ u_\tau }_{H^1(\Omega)^2} \leq C  (1  + \norm{ u_\tau }_{H^{1/2}(\partial \Omega)^2}  )
	$.	
	Now, $v_\tau = i e^\tau ({\sigma} \cdot\nu) u_\tau$ yields
	\begin{equation}
		\label{Eq:H12normsuv.bis}
		\norm{u_\tau }_{H^{1/2}(\partial \Omega)^2} \leq C e^{-\tau} \norm{v_\tau}_{H^{1/2}(\partial \Omega)^2}.
	\end{equation}
	Looking at $v_\tau$, if we combine \Cref{Lemma:BoundaryPb} with \eqref{Eq:BoundaryPbComm.UV}
	we get
	$$
	\|v_\tau\|_{H^{1/2}(\partial\Omega)^2}\leq 
	C\left(1+ |\lambda(\tau)+m| e^\tau+|\lambda(\tau)-m|e^{-\tau}\right)C_{\lambda(\tau)}
	\norm{v_\tau}_{H^{-1/2}(\partial \Omega)^2}.$$
	Using that $\lambda(\tau)\leq C_\circ$, it follows that
	$
	\| v_\tau \|_{H^{1/2}(\partial \Omega)^2} \leq C (1 + e^{\tau}) \| v_\tau \|_{H^{-1/2}(\partial \Omega)^2}$
	for all $\tau\in\mathcal{T}$.
	Applying  Lemma~\ref{Lemma:H1/2Estimate} and using that $\| \varphi_\tau \|_{L^2(\Omega)^4} = 1 $, we obtain
	$\| v_\tau \|_{H^{1/2}(\partial \Omega)^2} \leq C (1 + e^{\tau})$ for all $\tau\in\mathcal{T}$.
	Combining this with \eqref{Eq:H12normsuv.bis}, we get 
	$\norm{ u_\tau }_{H^{1/2}(\partial \Omega)^2}\leq C$ for all $\tau\in\mathcal{T}$, which proves \eqref{Eq:H1boundu.tau}.
	Once we have this uniform bound we proceed as in the proof of Proposition~\ref{Prop:CompactnessMinusInfty}, using now that $e^\tau\|u_\tau \|^2_{L^2(\partial \Omega)^2} \leq  C$ by 
	\Cref{Lemma:RelationNormsIntBdry} to get that $u^\star=0$ on 
	$\partial\Omega$. 
\end{proof}

The following lemma will be used in the proof of \Cref{Th:EigLimits}. It assures that $\|u_\tau\|_{L^2(\Omega)^2}$ does not tend to zero as $\tau\uparrow+\infty$. 
This fact will be used to show that the limit function $u^\star$ found in \Cref{Prop:CompactnessPlusInfty} is not identically zero and, therefore, that $(\lambda^\star)^2-m^2$ is an eigenvalue of the Dirichlet Laplacian on $\Omega$. 
As we will see in the proof of \Cref{Th:EigLimits}, to show an analogous nondegeneracy for $v^\star$ we will need a different argument based on formula \eqref{Eq:Lambda'Formula2}.

\begin{lemma}
	\label{Coro:u>vL^2}
	Let $\tau\in\R$, $\lambda>0$, and 
	$\varphi= (u, v)^\intercal\in \mathrm{Dom}(\Dirac_\tau)\setminus\{0\}$
	be such that 
	$\Dirac_\tau \varphi=\lambda\varphi$.
	Then,
	\begin{equation}
		\dfrac{\|u \|_{L^2(\Omega)^2}}{\| \varphi \|_{L^2( \Omega)^4}} > \dfrac{1}{\sqrt{2}} > \dfrac{\|v \|_{L^2(\Omega)^2}}{\| \varphi \|_{L^2( \Omega)^4}}.
	\end{equation}
\end{lemma}

\begin{proof}
	Thanks to \Cref{Th:ParamEigenvalues}, we can take a smooth parametrization of the eigenvalue 
	$\lambda=\lambda(\tau)$ and the associated eigenfunction 
	$\varphi = \varphi_\tau=(u_\tau, v_\tau)^\intercal$ in a neighborhood of $\tau$ (indeed in the whole real line). Set
	$$
	\gamma(\tau) := \dfrac{\|u_\tau \|^2_{L^2(\Omega)^2}}{\| \varphi_\tau \|^2_{L^2( \Omega)^4}}.
	$$
	Then $1 - \gamma(\tau) = \|v_\tau \|^2_{L^2(\Omega)^2} / \| \varphi_\tau \|^2_{L^2( \Omega)^4} $. Hence, it suffices to prove that $\gamma (\tau) > 1/2$.
	From \eqref{Eq:Lambda'Formula2} and the fact that $\lambda'(\tau) > 0$ it follows that
	$
	(\lambda(\tau) - m) \gamma(\tau) > (\lambda(\tau) + m)(1-\gamma(\tau))$,	which entails
	$
	\lambda(\tau) (2\gamma (\tau) - 1) >   m \geq 0.
	$
	Since $\lambda(\tau) > 0$, we conclude  that $\gamma (\tau) > 1/2$.
\end{proof}

\medskip 

With the previous results at hand, we can now establish \Cref{Th:EigLimits}.

\medskip

\begin{proof}[Proof of \Cref{Th:EigLimits}]
	In the following, for each $\tau\in I$ let $\varphi_\tau= (u_\tau, v_\tau)^\intercal\in \mathrm{Dom}(\Dirac_\tau)$ be such that $\Dirac_\tau\varphi_\tau=\lambda(\tau)\varphi_\tau$ and 
	$\| \varphi_\tau \|_{L^2( \Omega)^4} = 1$. 
	
	Let us first prove $(ii)$, which is shorter than proving $(i)$.
	Since $\tau\mapsto\lambda(\tau)\in\sigma(\Dirac_\tau)\cap(m,+\infty)$ is assumed to be continuous, the monotonicity of the eigenvalue curves proved in \Cref{Th:ParamEigenvalues} assures that $\lambda$ is strictly increasing on $I$, hence 
	the limit $\lambda(+\infty)$ exists and satisfies 
	$m<\lambda(+\infty)\leq+\infty.$ 
	
	Assume that $\lambda(+\infty)<+\infty$. Then $\lambda(\tau)\leq \lambda(+\infty)<+\infty$ for all $\tau\in (\tau_0,+\infty)$. By \Cref{Prop:CompactnessPlusInfty}, there exists a sequence $\{\tau_k\}_{k\in\N}\subset(\tau_0,+\infty)$ with $\lim_{k\uparrow+\infty}\tau_k = +\infty$ for which $u_{\tau_k}$ converges in $L^2(\Omega)^2$ to a function $u^\star\in H^2(\Omega)^2$ satisfying
	\begin{equation}
		\label{Eq:LimitUPbBis}
		\beqc{\PDEsystem}
		-\Delta u^\star & = & \big(\lambda(+\infty)^2-m^2\big) u^\star  & \text{in }  \Omega, \\
		u^\star & = & 0  & \text{on } \partial \Omega.
		\eeqc
	\end{equation}
	Now, since $\| \varphi_\tau \|_{L^2( \Omega)^4} = 1$ for all $\tau$, \Cref{Coro:u>vL^2} gives that $\|u_{\tau_k}\|_{L^2(\Omega)^2}>1/2$ for all $k$, thus $\|u^\star\|_{L^2(\Omega)^2}>1/2$ by the convergence in $L^2(\Omega)^2$. Therefore, $u^\star$ is an eigenfunction of the Dirichlet Laplacian on $\Omega$, which yields 
	$\lambda(+\infty)^2-m^2\in \sigma(-\Delta_D)$, as desired.
	
	Let us now address the proof of $(i)$. 
	We are assuming that $\tau\mapsto\lambda(\tau)$ is a continuous function on $(-\infty,\tau_0)$. Then, thanks to \Cref{Th:ParamEigenvalues}, $\lambda$ is indeed real analytic everywhere on $(-\infty,\tau_0)$ except possibly on countable many exceptional points where the graph of $\lambda$ may change from one real analytic eigenvalue curve to another one through a crossing point. \Cref{Th:ParamEigenvalues} actually shows that on every compact set of $(-\infty,\tau_0)$ there are only a finite number of these exceptional points. Moreover, by the monotonicity shown in \Cref{Th:ParamEigenvalues}, $\lambda'>0$ wherever $\lambda$ is differentiable (see  \Cref{Prop:DerivativeEigenvalues}), and thus the limit $\lambda(-\infty)$ exists and satisfies
	$m\leq\lambda(-\infty)<+\infty$. All these considerations  justify the identities
	$$
	\int_{-\infty}^{\tau_0-1} |\lambda'|=\int_{-\infty}^{\tau_0-1} \lambda' = \lambda(\tau_0-1) - \lambda(-\infty).
	$$
	Since $\lambda(\tau_0-1) - \lambda(-\infty)<+\infty$, we deduce that $\lambda'$ is absolutely integrable in $(-\infty, \tau_0-1)$. In  particular, there exists a sequence $\{\tau_k\}_{k\in\N}\subset
	(-\infty,\tau_0-1)$ such that 
	$\lim_{k\uparrow+\infty}\tau_k=-\infty$ and 
	\begin{equation}\label{lambda'-tend.zero.-infty}
		\lim_{k\uparrow+\infty}\lambda'(\tau_k)=0.
	\end{equation}
	Set 
	$\gamma(\tau):=\|u_\tau\|^2_{L^2(\Omega)^2}.$ Then, 
	$1-\gamma(\tau)=\|v_\tau\|^2_{L^2(\Omega)^2}$ by the fact that $\| \varphi_\tau \|_{L^2( \Omega)^4} = 1$, and
	\begin{equation}\label{lambda'-tend.zero.-infty.aux}
		\gamma(\tau)>\frac{1}{2}
	\end{equation}
	for all $\tau\in I$ by \Cref{Coro:u>vL^2}. Combining \eqref{lambda'-tend.zero.-infty} with \eqref{Eq:Lambda'Formula2} we deduce that 
	\begin{equation}
		\label{Eq:Eig->mProofClaim1}
		\lim_{k\uparrow+\infty} \Big(\big(\lambda(\tau_k) - m\big) \gamma(\tau_k) - \big(\lambda(\tau_k) + m\big)\big(1 - \gamma(\tau_k)\big)\Big)= \lim_{k\uparrow+\infty}\lambda'(\tau_k)=0.
	\end{equation}
	This limit will be the key point in the proof of $(ii)$.
	
	Recall that $m\leq\lambda(-\infty)<+\infty$. The next step  is to show that if $m<\lambda(-\infty)$ then 
	$\lambda(-\infty)^2-m^2\in \sigma(-\Delta_D)$. Note that  if $m<\lambda(-\infty)$ then the fact that $\lambda$ is increasing and \eqref{lambda'-tend.zero.-infty.aux} yield that 
	$(\lambda(\tau_k) - m) \gamma(\tau_k)
	\geq(\lambda(-\infty)-m)/2>0$ for all $k\in\N$. Then,
	using \eqref{Eq:Eig->mProofClaim1} we get that
	$(\lambda(\tau_k) + m)(1 - \gamma(\tau_k))>C$ for some $C>0$ and all $k$ big enough. In particular, since $\lambda(-\infty)<+\infty$, we deduce that
	\begin{equation}\label{lambda'-tend.zero.-infty.aux2}
		\|v_{\tau_k}\|^2_{L^2(\Omega)^2}
		=(1 - \gamma(\tau_k))>C
	\end{equation} 
	for some $C>0$ and all $k$ big enough. At this point, we simply have to use \Cref{Prop:CompactnessMinusInfty} on the sequence $\{\tau_k\}_{k\geq k_0}$ for some $k_0$ big enough to find a function $v_\star\in H^2(\Omega)^2\setminus\{0\}$ satisfying
	\begin{equation}\label{extremal.problem.-infty}
		\beqc{\PDEsystem}
		-\Delta v_\star & = & \big(\lambda(-\infty)^2-m^2\big) v_\star  & \text{in }  \Omega, \\
		v_\star & = & 0  & \text{on } \partial \Omega.
		\eeqc
	\end{equation}
	Note that $v_\star$ is not identically zero thanks to \eqref{lambda'-tend.zero.-infty.aux2} and the convergence in $L^2(\Omega)^2$ of the subsequence of 
	$\{v_{\tau_k}\}_{k\geq k_0}$ given by \Cref{Prop:CompactnessMinusInfty}. From here, we conclude that
	\begin{equation}\label{lambda.infty>m}
		\text{if $\lambda(-\infty)>m$ then $\lambda(-\infty)^2-m^2\in \sigma(-\Delta_D)$.}
	\end{equation} 
	
	Finally, assume that $\lambda(\tau)^2-m^2\leq\min\sigma(-\Delta_D)$ for some 	$\tau\in I$. Since $\lambda$ is strictly increasing on $I$, we deduce that $\lambda(-\infty)^2-m^2<\min\sigma(-\Delta_D)$, which leads to $\lambda(-\infty)^2-m^2\not\in \sigma(-\Delta_D)$. This, together with \eqref{lambda.infty>m} and the fact that $m\leq\lambda(-\infty)<+\infty$, entails 
	$\lambda(-\infty)=m$. On the contrary, if $\lambda(\tau)^2-m^2>\min\sigma(-\Delta_D)$ for all $\tau\in I$ then 
	$\lambda(-\infty)^2-m^2\geq\min\sigma(-\Delta_D)>0$,
	and thus $\lambda(-\infty)^2-m^2\in \sigma(-\Delta_D)$ by \eqref{lambda.infty>m}. 
\end{proof}

\medskip 

We now establish \Cref{Th:FirstEigTom}, concerning the first positive eigenvalue.

\medskip 

\begin{proof}[Proof of \Cref{Th:FirstEigTom}]
	Let us first show that $\tau\mapsto\lambda_1^+(\tau):=\min(\sigma(\Dirac_\tau)\cap(m,+\infty))$ is continuous and strictly increasing on $\R$. To do so, we will show that $\lambda_1^+$ is continuous and strictly increasing on $[0,+\infty)$, the proof for $(-\infty,0]$ is analogous.

	Let $\{\tau\mapsto\lambda_k (\tau) \}_{k\in\Z \setminus \{0\}}$ be the family of eigenvalues curves associated to the mapping
	$\tau\mapsto\Dirac_{\tau}$ given by \Cref{Th:ParamEigenvalues}. Note that this family contains pairs of curves whose graphs coincide. This is due to the fact that in the statement of \Cref{Th:ParamEigenvalues} the eigenvalues were repeated according to their algebraic multiplicity. In order to avoid this repetition, let us remove from the family $\{\lambda_k\}_{k\in\Z \setminus \{0\}}$ any curve
	$\lambda_k$ whose graph coincides with the graph of 
	$\lambda_j$ for some $j<k$. In this way we get a new family of curves, still denoted by $\{\lambda_k\}_{k\in\Z \setminus \{0\}}$, such that 
	$\sigma(\Dirac_\tau)=\bigcup_{k\in\Z \setminus \{0\}}\lambda_k(\tau)$ for all $\tau\in\R$, and such that the graphs of $\lambda_j$ and $\lambda_k$ differ whenever $j\neq k$. Moreover, thanks to \Cref{Th:ParamEigenvalues} the following holds: if $P\subset\R^2$ denotes the union of all the intersection points among the graphs of the curves in $\{\lambda_k\}_{k\in\Z \setminus \{0\}}$, then $P\cap K$ is finite for every compact set $K\subset\R^2$, and for each $p\in P$ there are only a finite number of curves whose graph intersects $p$. 
	
	From the previous considerations, and recalling that 
	$\lambda_1^+(0)=\min(\sigma(\Dirac_0)\cap(m,+\infty))$, we see that
	there exist only finitely many curves
	$\lambda_{k_1},\ldots,\lambda_{k_J}\in\{\lambda_k\}_{k\in\Z \setminus \{0\}}$ whose graphs intersect the point 
	$p_0=(0,\lambda_1^+(0))\in\R^2$. Furthermore, in a neighborhood of $p_0$ these curves only intersect at $p_0$. 
	Hence, there exists $k'\in\{k_1,\ldots,k_J\}$  and $\epsilon>0$ such that
	\begin{equation}\label{param.lambda1.taus}
		\lambda_{k'}(\tau)<\lambda_{k_j}(\tau)\quad\text{for all $j=1,\ldots,J$ with $k_j\neq k'$, and all $\tau\in(0,\epsilon)$.}
	\end{equation} 
	Moreover, the curves $\lambda_{k_j}(\tau)$ do not intersect any other eigenvalue curve for  $\tau \in [0,\epsilon)$ ---since in that interval all the other eigenvalue curves lie either below $-m$ or above the second positive eigenvalue of $\Dirac_0$ (by monotonicity).
	This shows that 
	$\lambda_1^+(\tau)=\lambda_{k'}(\tau)$ for all 
	$\tau\in[0,\epsilon)$.
	Then, since $\lambda_{k'}$ is real analytic and strictly increasing by \Cref{Th:ParamEigenvalues}, the same holds for $\lambda_1^+$ on $(0,\epsilon)$. 
	Now, since $\tau\mapsto\lambda_{k'}(\tau)$ is defined for all 
	$\tau\in\R$, we can increase 
	$\tau$ starting from $\tau=0$ in order to move us to the right along the graph of $\lambda_{k'}$. Regarding the family of curves $\{\lambda_k\}_{k\in\Z \setminus \{0\}}$, only two situations can happen. Either 
	\begin{itemize}
		\item[$(i)$] there exists $\tau_1>0$ such that the graph of 
		$\lambda_{k'}$ does not intersect any other graph for any 
		$\tau\in(0,\tau_1)$, but it intersects the graphs of (at most) finitely many curves at the point $p_1=(\tau_1,\lambda_{k'}(\tau_1))\in\R^2$, or
		\item[$(ii)$] the graph of $\lambda_{k'}$ does not intersect the graph of any other curve for any $\tau\in(0,+\infty)$.
	\end{itemize}
	
	We claim that if $(ii)$ holds then $\lambda_{k'}=\lambda_1^+$ on $[0,+\infty)$. Clearly, $\lambda_{k'}\geq\lambda_1^+$ on $[0,+\infty)$ by the definition of $\lambda_1^+$, the fact that 
	$\lambda_{k'}$ is continuous, and that $\lambda_{k'}(0)=\lambda_1^+(0)$. To prove the claim, assume by contradiction that 
	$\lambda_{k'}(\tau_0)>\lambda_1^+(\tau_0)$ for some 
	$\tau_0>0$. We know that there exists
	$\lambda\in\{\lambda_k\}_{k\in\Z \setminus \{0\}}$ such that 
	$\lambda(\tau_0)=\lambda_1^+(\tau_0)$. Since 
	$\lambda(0)\geq\lambda_1^+(0)=\lambda_{k'}(0)$, by $(ii)$ and 
	continuity we deduce that $\lambda(0)=\lambda_{k'}(0)$ and that $\lambda(\tau)<\lambda_{k'}(\tau)$ for all $\tau\in[0,\tau_0]$, but this contradicts \eqref{param.lambda1.taus}. Therefore, $\lambda_{k'}=\lambda_1^+$ on $[0,+\infty)$ if $(ii)$ holds.
	
	Assume now that $(i)$ holds. Arguing as in $(ii)$ we see that 
	$\lambda_{k'}=\lambda_1^+$ on $[0,\tau_1]$. Then, we can proceed exactly as we did for the point $p_0=(0,\lambda_1^+(0))\in\R^2$ but now for the point
	$p_1=(\tau_1,\lambda_1^+(\tau_1))\in\R^2$. We would see that either $\lambda_1^+$ coincides with some curve in  
	$\{\lambda_k\}_{k\in\Z \setminus \{0\}}$ on $[\tau_1,+\infty)$ or there exists, as in $(i)$, a new intersection point
	$p_2=(\tau_2,\lambda_{1}^+(\tau_2))\in\R^2$ with $\tau_2>\tau_1$ associated to the curve in $\{\lambda_k\}_{k\in\Z \setminus \{0\}}$ that coincides with $\lambda_1^+$ on $[\tau_1,\tau_2]$. Iterating this argument, in the worst case we would get an infinite sequence of intersection points 
	$p_j=(\tau_j,\lambda_{1}^+(\tau_j))\in\R^2$ for $j=1,2,\ldots$ such that $\tau_j<\tau_{j+1}$ for all $j$. However, recall from the beginning of the proof that, if $P\subset\R^2$ denotes the union of all the intersection points among the graphs of the curves in 
	$\{\lambda_k\}_{k\in\Z \setminus \{0\}}$, then $P\cap K$ is finite for every compact set $K\subset\R^2$. This yields that the set
	$\{\tau_j\}_j\cap [0,R]$ is finite for all $R>0$, which in particular means that $\lim_{j\uparrow+\infty}\tau_j=+\infty$. Therefore, in this worst case we still get a covering of the graph of $\lambda_1^+$ by the graphs of the curves in 
	$\{\lambda_k\}_{k\in\Z \setminus \{0\}}$ in a locally finite way.
	
	In conclusion, from how we described $\lambda_1^+$ in terms of the eigenvalue curves $\{\lambda_k\}_{k\in\Z \setminus \{0\}}$, and since these curves are real analytic and strictly increasing on $\R$, we deduce that $\lambda_1^+$ is continuous and strictly increasing on $\R$, and real analytic on 
	$\R\setminus E$, where $E\subset\R$ is some set such that $E\cap [-R,R]$ is finite for all $R>0$. 
	
	It only remains to prove \eqref{lim:pm.infty.lambda1+.state}. 
	That 
	$\lim_{\tau\uparrow+\infty}\lambda_1^+(\tau)^2-m^2\in \sigma(-\Delta_D)\cup\{+\infty\}$ follows directly from \Cref{Th:EigLimits}~$(ii)$, hence we only need to prove that $\lim_{\tau\downarrow-\infty}\lambda_1^+(\tau)=m$. To show this, by \Cref{Th:EigLimits}~$(i)$ it is enough to check that 
	$\lambda_1^+(\tau)^2-m^2\leq\min\sigma(-\Delta_D)$ for some $\tau\in\R$. The idea will be to bound from above the first eigenvalue of the positive operator $\Dirac_0^2$ by the one of $-\Delta_D$, and then to use that $\Dirac_0$ diagonalizes in a basis of eigenvectors to bring this estimate to $\lambda_1^+(0)$. The operator $\Dirac_0^2$ is defined by
	\begin{equation}\label{def.H2.lapl}
		\begin{split}
			\mathrm{Dom} (\Dirac_0^2) &:= \{\varphi \in \mathrm{Dom} (\Dirac_0): \, \Dirac_0 \varphi \in \mathrm{Dom} (\Dirac_0)\},\\
			\Dirac_0^2 \varphi &:= \Dirac^2 \varphi
			= (-\Delta + m^2)\varphi\quad\text{for all $\varphi\in\mathrm{Dom} (\Dirac_0^2)$}.
		\end{split}
	\end{equation}
	From (the proof of) \Cref{Lemma:SpectrumGenMIT}, we know that $\Dirac_0$ diagonalizes in an $L^2(\Omega)^4$-orthonormal basis of eigenfunctions, and that $\sigma(\Dirac_0)$ is symmetric. This yields that
	$$\langle \Dirac_0^2 \varphi, \varphi \rangle_{L^2(\Omega)^4}\geq \lambda_1^+(0)^2\| \varphi \|^2_{L^2(\Omega)^4}$$ for all $\varphi \in \mathrm{Dom} (\Dirac_0^2)$, and the equality holds if $\Dirac_0\varphi=\pm\lambda_1^+(0)\varphi$. Therefore,
	\begin{equation}
		\label{Eq:DefLambda1}
		\begin{split}
			\lambda_1^+(0)^2
			&= \min_{\varphi \in \mathrm{Dom} (\Dirac_0^2)\setminus\{0\}}  \dfrac{ \langle \Dirac_0^2 \varphi, \varphi \rangle_{L^2(\Omega)^4}}{\| \varphi \|^2_{L^2(\Omega)^4}}
			=\min_{\varphi \in \mathrm{Dom} (\Dirac_0^2)\setminus\{0\}}  \dfrac{ \langle(-\Delta + m^2) \varphi, \varphi \rangle_{L^2(\Omega)^4}}{\| \varphi \|^2_{L^2(\Omega)^4}}\\
			&\leq
			\inf_{\varphi \in C^\infty_c(\Omega)^4\setminus\{0\}}  \dfrac{ \langle(-\Delta + m^2) \varphi, \varphi \rangle_{L^2(\Omega)^4}}{\| \varphi \|^2_{L^2(\Omega)^4}}
			\leq
			m^2+ \inf_{\phi \in C^\infty_c (\Omega)\setminus\{0\}} \dfrac{ \| \nabla \phi \|^2_{L^2(\Omega)}}{\| \phi \|^2_{L^2(\Omega)}}\\
			&=m^2+\min\sigma(-\Delta_D) ,
		\end{split}
	\end{equation}	
	where we used that $C^\infty_c(\Omega)^4\subset\mathrm{Dom} (\Dirac_0^2)$, and the Rayleigh-Ritz principle in the last equality above. With this estimate at hand, \Cref{Th:EigLimits}~$(i)$ shows that 
	$\lim_{\tau\downarrow-\infty}\lambda_1^+(\tau)=m$.
\end{proof}

\medskip

\begin{remark} 
	\label{Remark:FirstEigFinite}
	Despite that from our arguments we cannot assure that 
	$\lim_{\tau\uparrow+\infty}\lambda_1^+(\tau)<+\infty$ in \eqref{lim:pm.infty.lambda1+.state}, we believe that indeed 
	\begin{equation}\label{conjecture1}
		\lim_{\tau\uparrow+\infty}\lambda_1^+(\tau)=\sqrt{ \min\sigma(-\Delta_D) + m^2}
	\end{equation} 
	for every bounded domain 
	$\Omega\subset\R^3$, as \Cref{Prop:ParamFirstEigSphere} shows in the case of a ball. The reason for this belief is the following one: in view of the boundary equation 
	$v  =  i e^\tau ( {\sigma}\cdot\nu)u$ in \eqref{Eq:EigenvaluePbUV.1}, we expect that, as $\tau\downarrow-\infty$, 
	$\Dirac_\tau$ tends to the Dirac operator $A_m$ with zigzag type boundary conditions studied in \cite{HolzmannZigZag}, whose negative eigenvalue with the smallest modulus is 
	$-\sqrt{ \min\sigma(-\Delta_D)+ m^2}$. If, for example, the convergence of $\Dirac_\tau$ to $A_m$ as $\tau\downarrow-\infty$ is in the strong resolvent sense, an application of \Cref{Lemma:SpectrumGenMIT}~$(iii)$ would yield that
	$\sqrt{\min\sigma(-\Delta_D) + m^2}=\lim_{\tau\uparrow+\infty}
	\lambda(\tau)$ for some  
	$\lambda(\tau)\in\sigma(\Dirac_\tau)\cap(m,+\infty)$. Then, the fact that we must have $\lambda(\tau)=\lambda_1^+(\tau)$ for all $\tau$ big enough, which would lead to \eqref{conjecture1}, should follow from the monotonicity of the eigenvalue curves and the fact that $-\sqrt{\min\sigma(-\Delta_D) + m^2}$ is the negative eigenvalue of $A_m$ with the smallest modulus. 
	
	In order to use this argument to get \eqref{conjecture1}, the convergence of $\Dirac_\tau$ in a resolvent sense  as $\tau\to\pm\infty$ must be studied. This question requires further work, and it will not be addressed in this article, since \Cref{Th:FirstEigTom} suffices to establish the shape optimization result stated in \Cref{Th:BallOptimalLargeTau}. 
\end{remark}

\smallskip

\begin{remark}
	\label{Remark:SetE}
	As seen from the proof of \Cref{Th:FirstEigTom}, the existence of the set $E$ in which $\lambda^+_1$ is not analytic depends on the possible crossing which an eigenvalue curve $\tau \mapsto \lambda(\tau)$, locally representing $\lambda^+_1$, may have with other eigenvalue curves.
	This, in turn, depends on the fact of  $\lambda^+_1$ having constant multiplicity for all $\tau \in \R$.	
	Although in the case of   $\Omega$ being  a ball we know from \Cref{Prop:ParamFirstEigSphere} that $\lambda^+_1$ has always multiplicity $2$  (and thus $E= \emptyset$ in this case), these questions remain open for a general domain $\Omega$.
\end{remark}
 
 \medskip 
 
We conclude the section by proving our shape optimization result for large values of the parameter $\tau$.

\medskip 

\begin{proof}[Proof of \Cref{Th:BallOptimalLargeTau}]
	Denote by $\lambda_\Omega^D:=\min\sigma(-\Delta_D)$ the first eigenvalue of the Dirichlet Laplacian in $\Omega$. From the Faber-Krahn inequality  we know that if 
	$\Omega$ is not a ball then 
	\begin{equation}\label{FK.ineq}
		\lambda_\Omega^D>\lambda_B^D,
	\end{equation}
	since $\partial\Omega$ is regular enough; see \cite[Remark 3.2.2]{Henrot}. Now, on the one hand,  \Cref{Th:FirstEigTom} yields
	\begin{equation}\label{FK.ineq.eq1}
		\lim_{\tau \uparrow + \infty} \lambda_\Omega(\tau)\geq{\sqrt{\lambda_\Omega^D+m^2}}
	\end{equation}
	and, on the other hand, \Cref{Prop:ParamFirstEigSphere} shows that
	\begin{equation}\label{FK.ineq.eq2}
		\lim_{\tau \uparrow + \infty} \lambda_B(\tau) = \sqrt{\lambda_B^D+ m^2}.
	\end{equation}
	Combining \eqref{FK.ineq}, \eqref{FK.ineq.eq1}, and \eqref{FK.ineq.eq2} we deduce that
	$\lim_{\tau \uparrow + \infty} \lambda_\Omega(\tau)
	>\lim_{\tau \uparrow + \infty} \lambda_B(\tau)$, from which the corollary follows.	
\end{proof}

\subsection{Skew projections onto Hardy spaces}\label{Subsec:Proj.Hardy}

In this section we introduce skew projections of $L^2(\partial\Omega)^2$ onto Hardy spaces. They are used in  \Cref{Subsec:Rayleigh} to prove \Cref{Th:Rayleigh.Intro}, a result which addresses the asymptotic expansion of $\lambda_\Omega(\tau)-m$ as $\tau\downarrow-\infty$.
For a more general perspective on this topic from the point of view of Clifford algebras and the Cauchy-Clifford operator, the reader may look at \cite{HMMPT}.

Let $P_\pm:L^2(\partial\Omega)^2\to L^2(\partial\Omega)^2$ be defined by
$$P_\pm:=\frac{1}{2}\pm iW_m(\sigma\cdot\nu),$$
where $W_m$ is defined in \eqref{Eq:DefK}, and let $(P_\pm)^*:L^2(\partial\Omega)^2\to L^2(\partial\Omega)^2$ be the adjoint operators with respect to 
$\langle\cdot,\cdot\rangle_{L^2(\partial\Omega)^2}$, namely,
$$(P_\pm)^*:=\frac{1}{2}\mp i (\sigma\cdot\nu)W_m$$
(recall that both $W_m$ and $\sigma\cdot\nu$ are bounded self-adjoint operators in $L^2(\partial\Omega)^2$).

\begin{lemma}\label{skew_projections}
	$P_+$ and $P_-$ are complementary projections of $L^2(\partial\Omega)^2$. More precisely,
	\begin{itemize}
		\item[$(i)$] $P_++P_-=1$,
		\item[$(ii)$] $P_\pm P_\mp=0$,
		\item[$(iii)$] $P_\pm P_\pm=P_\pm$, 
	\end{itemize}
	as bounded operators in $L^2(\partial\Omega)^2$. The same holds replacing $P_\pm$ by $(P_\pm)^*$.
\end{lemma}

\begin{proof}
	Statement $(i)$ is obvious. To show $(ii)$, recall that 
	$(W_m(\sigma\cdot\nu))^2=-1/4$ by \Cref{l.prop.K.W}~$(ii)$, and thus
	\begin{equation}
		P_\pm P_\mp=({\textstyle \frac{1}{2}}\pm iW_m(\sigma\cdot\nu))
		({\textstyle \frac{1}{2}}\mp iW_m(\sigma\cdot\nu))
		={\textstyle \frac{1}{4}}+(W_m(\sigma\cdot\nu))^2=0.
	\end{equation}
	Then, combining $(i)$ and $(ii)$ we see that
	$P_\pm P_\pm=P_\pm (P_\pm+P_\mp)=P_\pm$, which proves $(iii)$. Finally, taking adjoints in these identities, we get the same conclusions for $(P_\pm)^*$.
\end{proof}

The previous lemma shows that $P_\pm$ are skew projections of $L^2(\partial\Omega)^2$ parallel to (with kernel) $P_\mp(L^2(\partial\Omega)^2)$, and analogously for their adjoints. The subspaces $P_\pm(L^2(\partial\Omega)^2)$ of $L^2(\partial\Omega)^2$ are the so-called boundary Hardy (or Smirnov) spaces obtained by taking traces on 
$\partial\Omega$ of inner/outer null-solutions of 
$\sigma\cdot\nabla$:
setting 
$\Omega_+:=\Omega$ and 
$\Omega_-:=\R^3\setminus\overline\Omega$,
from the reproducing formula for 
$\C^2$-valued functions analogous to \eqref{Eq:ReproductionFormula} one sees that if $u_\pm\in H^1(\Omega_\pm)^2$ are such that 
$(\sigma\cdot\nabla)u_\pm=0$ in $\Omega_\pm$, then their traces on $\partial\Omega$ 
satisfy $P_\pm u_\pm=u_\pm$  as functions in $L^2(\partial\Omega)^2$, which is equivalent to say that
$u_\pm\in P_\pm(L^2(\partial\Omega)^2)=\operatorname{Ker}(P_\mp)$.

Despite being projections, $P_+$ and $P_-$ are not orthogonal projections in general. Indeed, the fact of being orthogonal characterizes the shape of $\Omega$, as the following result shows. The reader should also look at \cite{HMMPT} for a deeper treatment of the interplay between the geometry of Hardy spaces and the geometry of the underlying domain 
$\Omega$. In \cite{HMMPT} the authors consider the general framework of domains with locally finite perimeter. For the sake of simplicity, in the following lemma we only focus on bounded regular domains. 

\begin{lemma}\label{Lemma:EquivalencesPBall}
	Let $\Omega\subset\R^3$ be a bounded domain with $C^1$ boundary. The following are equivalent:
	\begin{itemize}
		\item[$(i)$] $P_+$ and $P_-$ are self-adjoint operators in $L^2(\partial\Omega)^2$.
		\item[$(ii)$] $P_+$ and $P_-$ are orthogonal projections of $L^2(\partial\Omega)^2$.
		\item[$(iii)$] $\{W_m,\sigma\cdot\nu\}=0$ as operators in $L^2(\partial\Omega)^2$. 
		\item[$(iv)$] $\Omega$ is a ball. 
	\end{itemize}
\end{lemma}
\begin{proof}
	It is obvious that $(i)$ is equivalent to $(iii)$. We first prove that $(ii)$ is equivalent to $(iii)$. Then we prove that, if $\Omega$ is bounded, $(iii)$ is equivalent to $(iv)$. 
	
	On the one hand, if $P_+$ and $P_-$ are orthogonal projections of $L^2(\partial\Omega)^2$ then
	\begin{equation}
		\begin{split}
			0=\langle P_\pm u,P_\mp v\rangle_{L^2(\partial\Omega)^2}
			=\langle (P_\mp)^*P_\pm u,v\rangle_{L^2(\partial\Omega)^2}
		\end{split}
	\end{equation}
	for all $u,v\in L^2(\partial\Omega)^2$.
	This is equivalent to say that 
	\begin{equation}
		\begin{split}
			0&=(P_\mp)^*P_\pm=({\textstyle \frac{1}{2}}\pm i(\sigma\cdot\nu)W_m)
			({\textstyle \frac{1}{2}}\pm iW_m(\sigma\cdot\nu))\\
			&={\textstyle \frac{1}{4}}-(\sigma\cdot\nu)W_mW_m(\sigma\cdot\nu)\pm{\textstyle \frac{i}{2}}\{W_m,\sigma\cdot\nu\}.
		\end{split}
	\end{equation}
	Subtracting both expressions, we conclude that $\{W_m,\sigma\cdot\nu\}=0$.  
	On the other hand, if $(iii)$ holds then 
	$(P_\pm)^*=P_\pm$ by $(i)$, thus 
	$(P_\mp)^*P_\pm=P_\mp P_\pm=0$ by \Cref{skew_projections}. This shows that $(ii)$ and $(iii)$ are equivalent.
	
	We now prove that $(iii)$ and $(iv)$ are equivalent statements. Recall that
	\begin{equation}
		\begin{split}
			W_m u(x) = \lim_{\epsilon \downarrow 0} \frac{i}{4\pi }\int_{\partial \Omega \cap \{|x-y|> \epsilon\}}    |x-y|^{-3}
			\big({\sigma} \cdot (x-y)\big) u(y)\, d \upsigma(y);
		\end{split}
	\end{equation}
	see \eqref{Eq:DefK}.
	Therefore, $\{W_m,\sigma\cdot\nu\}=0$ if and only if
	\begin{equation}\label{proj.equiv.ball:eq1}
		\big({\sigma} \cdot (x-y)\big)\big(\sigma\cdot\nu(y)\big)
		=-\big(\sigma\cdot\nu(x)\big)\big({\sigma} \cdot (x-y)\big)
		\quad\text{ for $\upsigma$-a.e.\! $x,y\in\partial\Omega$.}
	\end{equation}
	Since $\nu$ is continuous because 
	$\partial\Omega$ is of class $C^1$, we can replace ``for $\upsigma$-a.e.\! $x,y\in\partial\Omega$'' by ``for all $x,y\in\partial\Omega$'' in \eqref{proj.equiv.ball:eq1}. Recall now that $(\sigma\cdot a)(\sigma\cdot b)=a\cdot b+i\sigma\cdot(a\times b)$ for all $a,b\in\R^3$. Therefore, multiplying by ${\sigma} \cdot (x-y)$ from the left both hand sides of \eqref{proj.equiv.ball:eq1}, we get that 
	$\{W_m,\sigma\cdot\nu\}=0$ if and only if
	\begin{equation}\label{proj.equiv.ball:eq2}
		|x-y|^2\big(\sigma\cdot\nu(y)\big)
		=-\big({\sigma} \cdot (x-y)\big)\big(\sigma\cdot\nu(x)\big)\big({\sigma} \cdot (x-y)\big)
		\quad\text{for all $x,y\in\partial\Omega$.}
	\end{equation}
	Observe also that, for every $a,b\in\R^3$, 
	\begin{equation}
		\begin{split}
			(\sigma\cdot a)(\sigma\cdot b)
			&=a\cdot b+i\sigma\cdot(a\times b)
			=-b\cdot a-i\sigma\cdot(b\times a)+2a\cdot b
			=-(\sigma\cdot b)(\sigma\cdot a)+2a\cdot b,
		\end{split}
	\end{equation}
	which yields
	\begin{equation}
		\begin{split}
			(\sigma\cdot a)(\sigma\cdot b)(\sigma\cdot a)
			&=-(\sigma\cdot b)(\sigma\cdot a)(\sigma\cdot a)
			+2(a\cdot b) (\sigma\cdot a)\\
			&=-|a|^2(\sigma\cdot b)
			+2(a\cdot b) (\sigma\cdot a)
			=\sigma\cdot\big(\!-\!|a|^2b+2(a\cdot b)a\big).
		\end{split}
	\end{equation}
	Using this formula on the right-hand side of \eqref{proj.equiv.ball:eq2} taking 
	$a=x-y$ and $b=\nu(x)$ we get
	\begin{equation}\label{proj.equiv.ball:eq3}
		\sigma\cdot \big(|x-y|^2\nu(y)\big)
		=\sigma\cdot\big(|x-y|^2\nu(x)-2((x-y)\cdot\nu(x))(x-y)\big).
	\end{equation}
	Now, since the Pauli matrices together with the identity matrix form a basis for the real vector space of $2\times2$ Hermitian matrices, from \eqref{proj.equiv.ball:eq2} and \eqref{proj.equiv.ball:eq3}, we see that $\{W_m,\sigma\cdot\nu\}=0$ if and only if
	\begin{equation}\label{proj.equiv.ball:eq4}
		\nu(y)=\nu(x)-\frac{2(x-y)\cdot\nu(x)}{|x-y|^2}(x-y)
		\quad\text{for all $x,y\in\partial\Omega$ with $x\neq y$.}
	\end{equation}
	Finally, since $\Omega$ is bounded and with $C^1$ boundary, the Reflection Lemma (see, for example, \cite[Lemma 5.3 on page 45]{Chipot}) shows that \eqref{proj.equiv.ball:eq4} holds if and only if $\Omega$ is a ball.
\end{proof}

\begin{remark}
	\label{Rem:AnticommutatorBall}
	From the previous proof it follows that the equivalence between $(iii)$ and $(iv)$ in \Cref{Lemma:EquivalencesPBall} holds not only for $W_m$, but for $W_\lambda$ with $\lambda \in \R$ ---recall that $W_\lambda$ is defined in \eqref{Eq:DefK}.
\end{remark}

\medskip

\subsection{First order asymptotics as $\tau\downarrow-\infty$}
\label{Subsec:Rayleigh}

Recall that in order to highlight the dependence of $\Dirac_\tau$ on the domain $\Omega\subset\R^3$, we denote by $\lambda_\Omega(\tau)$ the first positive eigenvalue of $\Dirac_\tau$ (that is, we set $\lambda_\Omega:=\lambda_1^+$).
Recall also that $\lim_{\tau\downarrow-\infty}\lambda_\Omega(\tau)=m$ by \Cref{Th:FirstEigTom}. The purpose of this section is to address the asymptotic expansion of $\lambda_\Omega(\tau)-m$ as $\tau\downarrow-\infty$. To do it, 
in \eqref{def.L.Omega} we introduced the function
$L_\Omega:\R\to(0,+\infty)$ defined by
\begin{equation}
	\tau\mapsto L_\Omega(\tau):= (\lambda_\Omega(\tau)-m)e^{-\tau}.
\end{equation}
With this notation, $\lambda_\Omega(\tau)=m+e^{\tau}L_\Omega(\tau)$ for all $\tau\in\R$.

\begin{lemma}
	\label{Lemma:LDecreasing}
	The function $L_\Omega$ is strictly decreasing on $\R$. 
\end{lemma}
\begin{proof}
	Since $\lambda_\Omega$ is differentiable everywhere except possibly at countable many points by \Cref{Th:FirstEigTom},  $L_\Omega$ too. Therefore, the lemma follows if we show that $L_\Omega'<0$ at the points of differentiability. 
	
	For every $\tau\in\R$, let $\varphi_\tau= (u_\tau, v_\tau)^\intercal\in \mathrm{Dom}(\Dirac_\tau)\setminus\{0\}$ such that 
	$\Dirac_\tau\varphi_\tau=\lambda_\Omega(\tau)\varphi_\tau$. We claim that $\|v_\tau\|_{L^2( \Omega)^2}>0$ for all $\tau\in\R$. To see this, assume that $\|v_\tau\|_{L^2( \Omega)^2}=0$. Then, \eqref{Eq:EigenvaluePbUV.1} shows that $-\Delta u_\tau=(-i {\sigma} \cdot \nabla)^2 u_\tau  =
	(-i {\sigma} \cdot \nabla)(\lambda_\Omega(\tau)+m) v_\tau=  0$ in $\Omega$, and $u_\tau  =  -i e^{-\tau} ( {\sigma}\cdot\nu)v_\tau=0$ on $\partial\Omega$. Hence,  
	$\|u_\tau\|_{L^2( \Omega)^2}=\|v_\tau\|_{L^2( \Omega)^2}=0$. This implies that $\|\varphi_\tau\|_{L^2( \Omega)^4}=0$, which contradicts the fact that $\varphi_\tau\in \mathrm{Dom}(\Dirac_\tau)\setminus\{0\}$. 
	
	Now, set 
	$\gamma(\tau) := { \|u_\tau \|^2_{L^2( \Omega)^2} }/{\|\varphi_{\tau}\|^2_{L^{2}(\Omega)^4}}=1-{ \|v_\tau \|^2_{L^2( \Omega)^2} }/{\|\varphi_{\tau}\|^2_{L^{2}(\Omega)^4}}$. We have shown that $\gamma(\tau)<1$ for all $\tau\in\R$.
	Then, using \eqref{Eq:Lambda'Formula2} we have
	\begin{equation}
		\begin{split}
			L_\Omega'(\tau) 
			&= \lambda_\Omega'(\tau) e^{-\tau} 
			- L_\Omega(\tau)\\
			&= (\lambda_\Omega(\tau)-m) \gamma(\tau) e^{-\tau}
			-(\lambda_\Omega(\tau)+m)(1-\gamma(\tau))e^{-\tau}
			-L_\Omega(\tau) \\
			&= L_\Omega(\tau) \gamma(\tau)
			-(\lambda_\Omega(\tau)+m)(1-\gamma(\tau))e^{-\tau}
			-L_\Omega(\tau) \\
			&= (\gamma(\tau)-1)
			\big( L_\Omega(\tau)
			+(\lambda_\Omega(\tau)+m) e^{-\tau}\big)  < 0,
		\end{split}
	\end{equation}
	establishing the result.
\end{proof}

Thanks to the monotonicity of $L_\Omega$ proved in \Cref{Lemma:LDecreasing}, we get that the limit 
$$L_\Omega^\star:=\lim_{\tau\downarrow-\infty}L_\Omega(\tau)$$
exists as an element of $(0,+\infty]$. In particular, in the case that $L_\Omega^\star<+\infty$, we deduce that 
$\lambda_\Omega(\tau)$ behaves like $m+L_\Omega^\star e^\tau$ as $\tau\downarrow-\infty$. That is, $L_\Omega^\star$ quantifies the speed of convergence of $\lambda_\Omega(\tau)$ towards $m$ as $\tau\downarrow-\infty$. Our purpose now is to prove \Cref{Th:Rayleigh.Intro}, where we give a lower bound for $L_\Omega^\star$ (which is sharp if $\Omega$ is a ball) in terms of an optimization problem posed on the boundary Hardy space $P_+(L^2(\partial\Omega)^2)$. For the convenience of the reader, we  first recall the definitions of $\mathcal L_\Omega$, $\rcal_\Omega$, and $\rcal$ given in \eqref{form.rayleigh.l.omega.setL}, \eqref{form.rayleigh.r.omega}, and \eqref{form.rayleigh.R}, respectively.
Under the notation used in \Cref{Subsec:Proj.Hardy}, the set 
$\mathcal L_\Omega$ is defined by
\begin{equation}\label{form.rayleigh.l.omega.setL.true}
	\begin{split}
		\mathcal L_\Omega:=\Big\{L\in\C:\,
		\exists\, u\in L^2(\partial\Omega)^2\setminus\{0\}
		\text{ with }P_-u=0,\,
		(P_+)^*u=
		L( \sigma\cdot\nu)K_m( \sigma\cdot\nu)u\Big\},
	\end{split}
\end{equation}
the functional $\rcal$ is given by 
\begin{equation}
	\rcal(u):=
	\frac{\langle(\sigma\cdot\nu)K_m(\sigma\cdot\nu)u,u\rangle_{L^2(\partial\Omega)^2}}{\|u\|_{L^2(\partial\Omega)^2}^2}
	\quad\text{for }u\in L^2(\partial\Omega)^2\setminus\{0\},
\end{equation}
and 
$$\rcal_\Omega:=\sup_{u\in L^2(\partial\Omega)^2\setminus\{0\},\,
	P_-u=0}\rcal(u).$$
Below \eqref{form.rayleigh.R} we mentioned that $0<\rcal(u)<\|K_m\|_{L^2(\partial\Omega)^2\to L^2(\partial\Omega)^2}$
for all $u\in L^2(\partial\Omega)^2\setminus\{0\}$. This combined with the fact that $P_+(L^2(\partial\Omega)^2) = \operatorname{Ker} (P_-) \neq\{0\}$ ---for instance, the constants belong to $P_+(L^2(\partial\Omega)^2)$--- yields 
\begin{equation}\label{ROmega.pos.finite}
	0<\rcal_\Omega\leq\|K_m\|_{L^2(\partial\Omega)^2\to L^2(\partial\Omega)^2}.
\end{equation}

The proof of \Cref{Th:Rayleigh.Intro} will be divided into several steps. First, in \Cref{Th:Rayleigh.Intro.l1} we will show that $L_\Omega^\star\in\mathcal L_\Omega$ whenever $L_\Omega^\star<+\infty$. Then, in \Cref{Th:Rayleigh.Intro.l2} we will prove that $\mathcal L_\Omega\subset\R$ and that
$1/\rcal_\Omega\leq L$ for all $L\in\mathcal L_\Omega$. The proof of the fact that $\rcal_\Omega$ is attained is given in \Cref{Rstar.attained}, and in \Cref{Rstar.EL} we will see that
the maximizers for $\rcal_\Omega$ are in the kernel of 
$(P_+)^*-
\frac{1}{\rcal_\Omega}( \sigma\cdot\nu)K_m( \sigma\cdot\nu)$. Finally, the proof of the fact that the equality in $1/\rcal_\Omega\leq L_\Omega^\star$ is attained if $\Omega$ is a ball is given in \Cref{Prop:ParamFirstEigSphere}.

\begin{lemma}\label{Th:Rayleigh.Intro.l1}
	If $L_\Omega^\star<+\infty$ then $L_\Omega^\star\in\mathcal L_\Omega$.
\end{lemma}

\begin{proof}
	From \Cref{Lemma:BoundaryPb} we see that the eigenvalue equation 
	$\Dirac_\tau\varphi_\tau=\lambda_\Omega(\tau)\varphi_\tau$ is equivalent the to system of equations
	\begin{equation}
		\label{Eq:BoundaryPb.aux1}
		\begin{split}
			u_\tau & =  \big(2i W_{\lambda_\Omega(\tau)}  ({\sigma}\cdot\nu) - 2({\lambda_\Omega(\tau)}+m) e^\tau K_{\lambda_\Omega(\tau)}\big) u_\tau,\\
			u_\tau & =  \big( 2 i ({\sigma}\cdot\nu) W_{\lambda_\Omega(\tau)}  + 2({\lambda_\Omega(\tau)}-m)e^{-\tau} ({\sigma}\cdot\nu )  K_{\lambda_\Omega(\tau)} ({\sigma}\cdot\nu) \big) u_\tau, 
		\end{split}
	\end{equation}
	where $\varphi_\tau= (u_\tau, i e^\tau ({\sigma}\cdot\nu) u_\tau)^\intercal$ on $\partial\Omega$. \Cref{Lemma:BoundaryPb} also shows that $\varphi_\tau$ vanishes identically on $\Omega$ if and only if $u_\tau$ vanishes identically on $\partial\Omega$. Hence, by homogeneity we can assume that $\|u_\tau\|_{L^2(\partial\Omega)^2}=1$ for all $\tau\in \R$. Now, 
	using that $\lim_{\tau\downarrow-\infty}\lambda_\Omega(\tau)=m$, that $L_\Omega^\star<+\infty$, and 
	\eqref{Eq:BoundaryPbComm.UV}, we see that there exist 
	$\tau_0\in\R$ and $C_\circ>0$ such that
	\begin{equation}
		\begin{split}
			\|u_\tau\|_{H^{1/2}(\partial\Omega)^2}&\leq 
			C(1+ |\lambda_\Omega(\tau)+m| e^\tau+|\lambda_\Omega(\tau)-m|e^{-\tau})C_{\lambda_\Omega(\tau)}
			\norm{u_\tau}_{H^{-1/2}(\partial \Omega)^2}\\
			&\leq 
			C\norm{u_\tau}_{H^{-1/2}(\partial \Omega)^2}
			\leq C_\circ\norm{u_\tau}_{L^2(\partial \Omega)^2}
			=C_\circ
		\end{split}
	\end{equation}
	for all $\tau<\tau_0$. From this uniform estimate and the compact embedding of 
	$H^{1/2}(\partial\Omega)^2$ into $L^2(\partial\Omega)^2$, we get the existence of a sequence $\{\tau_k\}_{k\in\N}$ with 
	$\lim_{k\uparrow+\infty}\tau_k=-\infty$ for which
	$u_{\tau_k}$ converges in $L^2(\partial\Omega)^2$ as $k\uparrow+\infty$ to some $u_\star\in L^2(\partial\Omega)^2$ with $\|u_\star\|_{L^2(\partial\Omega)^2}=1$. With this limit function at hand, we now consider \eqref{Eq:BoundaryPb.aux1} for $\tau=\tau_k$.
	It is an exercise to show that the operators $W_\lambda$ and $K_\lambda$, as bounded operators in $L^2(\partial\Omega)^2$, depend continuously on the parameter 
	$\lambda\in\R$; recall \eqref{Eq:DefK} and \eqref{kernels.K.W}. Therefore, taking the limit $k\uparrow+\infty$ in \eqref{Eq:BoundaryPb.aux1} with $\tau=\tau_k$, we get that 
	\begin{eqnarray}
		&&u_\star=2iW_m(\sigma\cdot\nu)u_\star, \label{limit-infty:eq1}\\
		&&u_\star=2i(\sigma\cdot\nu)W_m u_\star
		+2L_\Omega^\star( \sigma\cdot\nu)K_m( \sigma\cdot\nu)u_\star
		\label{limit-infty:eq2}
	\end{eqnarray}
	in $L^2(\partial\Omega)^2$. Since \eqref{limit-infty:eq1} is equivalent to $P_-u_\star=0$, and \eqref{limit-infty:eq2} can be rewritten as $(P_+)^*u_\star=
	L_\Omega^\star( \sigma\cdot\nu)K_m( \sigma\cdot\nu)u_\star$, we conclude that $L_\Omega^\star\in\mathcal L_\Omega$. 
\end{proof}

\begin{lemma}\label{Th:Rayleigh.Intro.l2}
	$\mathcal L_\Omega\subset\R$, and 
	$1/\rcal_\Omega\leq L$ for all $L\in\mathcal L_\Omega$.
\end{lemma}

\begin{proof} 
	Let 
	$L\in \mathcal L_\Omega$. By definition of $\mathcal L_\Omega$, there exists $u\in L^2(\partial\Omega)^2\setminus\{0\}$ such that $P_-u=0$ and 
	$(P_+)^*u=L( \sigma\cdot\nu)K_m( \sigma\cdot\nu)u$, that is, 
	\begin{eqnarray}
		&&u=2iW_m(\sigma\cdot\nu)u, \label{limit-infty:eq1.L}\\
		&&u=2i(\sigma\cdot\nu)W_m u
		+2L( \sigma\cdot\nu)K_m( \sigma\cdot\nu)u.
		\label{limit-infty:eq2.L}
	\end{eqnarray}
	Multiplying \eqref{limit-infty:eq2.L} by $\overline u$, integrating on $\partial\Omega$, using that 
	$(\sigma\cdot\nu)$ and $W_m$ are self-adjoint operators on $L^2(\partial\Omega)^2$, and using \eqref{limit-infty:eq1.L} we obtain
	\begin{equation}
		\begin{split}
			\|u\|_{L^2(\partial\Omega)^2}^2
			&=\langle2i(\sigma\cdot\nu)W_m u,u\rangle_{L^2(\partial\Omega)^2}
			+2L\langle(\sigma\cdot\nu)K_m(\sigma\cdot\nu)u,u\rangle_{L^2(\partial\Omega)^2}\\
			&=-\langle u,2iW_m(\sigma\cdot\nu)u\rangle_{L^2(\partial\Omega)^2}
			+2L\langle(\sigma\cdot\nu)K_m(\sigma\cdot\nu)u,u\rangle_{L^2(\partial\Omega)^2}\\
			&=-\|u\|_{L^2(\partial\Omega)^2}^2
			+2L\langle(\sigma\cdot\nu)K_m(\sigma\cdot\nu)u,u\rangle_{L^2(\partial\Omega)^2}.
		\end{split}
	\end{equation}
	From this we deduce that
	\begin{equation}\label{intuition R(u)}
		L=\frac{\|u\|_{L^2(\partial\Omega)^2}^2}
		{\langle(\sigma\cdot\nu)K_m(\sigma\cdot\nu)u,u\rangle_{L^2(\partial\Omega)^2}}=\frac{1}{\rcal(u)}.
	\end{equation}
	On the one hand, using that $(\sigma\cdot\nu)K_m(\sigma\cdot\nu)$ is self-adjoint we get that $L\in\R$, which proves that $\mathcal L_\Omega\subset\R$. On the other hand, using that $P_-u=0$, \eqref{intuition R(u)} also yields
	$L\geq{1}/{\rcal_\Omega}.$
\end{proof}

\begin{lemma}\label{Rstar.attained}
	There exists $u\in L^2(\partial\Omega)^2\setminus\{0\}$ such that $P_-u=0$ and $\rcal(u)=\rcal_\Omega$.
\end{lemma}

\begin{proof}
	We take $u_j\in L^2(\partial\Omega)^2\setminus\{0\}$ such that $P_-u_j=0$ for all 
	$j\in\N$ and 
	$\lim_{j\uparrow+\infty}\rcal(u_j)=\rcal_\Omega$. From the fact that $\rcal$ is homogeneous, we can assume that 
	$\|u_j\|_{L^2(\partial\Omega)^2}=1$ for all $j$. Then, since $K_m$ is a compact operator in $L^2(\partial\Omega)^2$, up to a subsequence, there exists 
	\begin{equation}\label{Rstar.attained:eq5}
		g:=\lim_{j\uparrow+\infty}(\sigma\cdot\nu)K_m
		(\sigma\cdot\nu)u_j\quad\text{in $L^2(\partial\Omega)^2$.}
	\end{equation}
	Also, since 
	$\|u_j\|_{L^2(\partial\Omega)^2}=1$ for all $j$, by Banach-Alaoglu theorem there exists 
	\begin{equation}\label{Rstar.attained:eq6}
		u:=\lim_{j\uparrow+\infty}u_j\quad\text{in the weak$^*$ topology of $(L^2(\partial\Omega)^2)^*\cong L^2(\partial\Omega)^2$},
	\end{equation}
	up to a subsequence.
	This means that $u\in L^2(\partial\Omega)^2$ and
	$\lim_{j\uparrow+\infty}\langle
	u_j,w\rangle_{L^2(\partial\Omega)^2}
	=\langle u,w\rangle_{L^2(\partial\Omega)^2}$
	for all $w\in L^2(\partial\Omega)^2$. In particular, for every $w\in L^2(\partial\Omega)^2$ with 
	$\|w\|_{L^2(\partial\Omega)^2}=1$ we have
	\begin{equation}
		\begin{split}
			1=\lim_{j\uparrow+\infty}\|u_j\|_{L^2(\partial\Omega)^2}=\lim_{j\uparrow+\infty}\sup_{\|v\|_{L^2(\partial\Omega)^2}=1}
			|\langle u_j,v\rangle_{L^2(\partial\Omega)^2}|
			\geq\lim_{j\uparrow+\infty}
			|\langle u_j,w\rangle_{L^2(\partial\Omega)^2}|
			=|\langle u,w\rangle_{L^2(\partial\Omega)^2}|.
		\end{split}
	\end{equation}
	Taking the supremum of $|\langle u,w\rangle_{L^2(\partial\Omega)^2}|$ among all $w\in L^2(\partial\Omega)^2$ with $\|w\|_{L^2(\partial\Omega)^2}=1$, we get
	\begin{equation}\label{Rstar.attained:eq1}
		\|u\|_{L^2(\partial\Omega)^2}\leq1.
	\end{equation}
	
	The next step is to relate $u$ and $g$. Since $K_m$ and $\sigma\cdot\nu$ are self-adjoint in $L^2(\partial\Omega)^2$, for every $w\in L^2(\partial\Omega)^2$ we have
	\begin{equation}
		\begin{split}
			\langle g,w \rangle_{L^2(\partial\Omega)^2}
			&=\lim_{j\uparrow+\infty}
			\langle (\sigma\cdot\nu)K_m(\sigma\cdot\nu)u_j,w \rangle_{L^2(\partial\Omega)^2}
			=\lim_{j\uparrow+\infty}
			\langle u_j,(\sigma\cdot\nu)K_m(\sigma\cdot\nu)w \rangle_{L^2(\partial\Omega)^2}\\
			&=\langle u,(\sigma\cdot\nu)K_m(\sigma\cdot\nu)w \rangle_{L^2(\partial\Omega)^2}
			=\langle (\sigma\cdot\nu)K_m(\sigma\cdot\nu)u,w \rangle_{L^2(\partial\Omega)^2},
		\end{split}
	\end{equation}
	which yields 
	\begin{equation}\label{Rstar.attained:eq2}
		g=(\sigma\cdot\nu)K_m(\sigma\cdot\nu) u\quad\text{in $L^2(\partial\Omega)^2$}.
	\end{equation}
	
	Let us now prove that $P_-u=0$. Using that $P_-u_j=0$ for all $j$ and \eqref{Rstar.attained:eq6} we see that, for every $w\in L^2(\partial\Omega)^2$,
	\begin{equation}
		\begin{split}
			0=\lim_{j\uparrow+\infty}
			\langle P_-u_j,w\rangle_{L^2(\partial\Omega)^2}
			=\lim_{j\uparrow+\infty}
			\langle u_j,(P_-)^*w\rangle_{L^2(\partial\Omega)^2}
			=\langle u,(P_-)^*w\rangle_{L^2(\partial\Omega)^2}
			=\langle P_-u,w\rangle_{L^2(\partial\Omega)^2},
		\end{split}
	\end{equation}
	which leads to
	\begin{equation}\label{Rstar.attained:eq4}
		P_-u=0\quad\text{in $L^2(\partial\Omega)^2$}.
	\end{equation}
	We also claim that $u$ is not identically zero. If it was, we would get $g=0$ by \eqref{Rstar.attained:eq2} and, thus, 
	$\rcal_\Omega=\lim_{j\uparrow+\infty}\rcal(u_j)$ would also be zero by \eqref{Rstar.attained:eq5}. However, this would contradict \eqref{ROmega.pos.finite}.
	
	Finally, since $\|u_j\|_{L^2(\partial\Omega)^2}=1$ for all $j$, we have
	\begin{equation}
		\begin{split}
			|\langle(\sigma\cdot\nu)K_m(\sigma\cdot\nu) u_j-g,u_j \rangle_{L^2(\partial\Omega)^2}|
			\leq\|(\sigma\cdot\nu)K_m(\sigma\cdot\nu) u_j-g\|_{L^2(\partial\Omega)^2},
		\end{split}
	\end{equation}
	which yields $\lim_{j\uparrow+\infty}
	\langle(\sigma\cdot\nu)K_m(\sigma\cdot\nu) u_j-g,u_j \rangle_{L^2(\partial\Omega)^2}=0$ by \eqref{Rstar.attained:eq5}. Combining this with \eqref{Rstar.attained:eq2}, \eqref{Rstar.attained:eq6}, \eqref{Rstar.attained:eq1}, and \eqref{Rstar.attained:eq4}, we conclude that
	\begin{equation}
		\begin{split}
			0<\rcal_\Omega&=\lim_{j\uparrow+\infty}\rcal(u_j)
			=\lim_{j\uparrow+\infty}
			\langle(\sigma\cdot\nu)K_m(\sigma\cdot\nu)u_j,u_j\rangle_{L^2(\partial\Omega)^2}
			=\lim_{j\uparrow+\infty}
			\langle g,u_j\rangle_{L^2(\partial\Omega)^2}\\
			&=\lim_{j\uparrow+\infty}
			\langle (\sigma\cdot\nu)K_m(\sigma\cdot\nu) u,u_j\rangle_{L^2(\partial\Omega)^2}
			=\langle (\sigma\cdot\nu)K_m(\sigma\cdot\nu) u,u\rangle_{L^2(\partial\Omega)^2}\\
			&\leq\frac{\langle (\sigma\cdot\nu)K_m(\sigma\cdot\nu) u,u\rangle_{L^2(\partial\Omega)^2}}{\|u\|_{L^2(\partial\Omega)^2}^2}=\rcal(u)\leq \rcal_\Omega,
		\end{split}
	\end{equation}
	which leads to $\rcal(u)=\rcal_\Omega$. As a byproduct, we also deduce that 
	$\|u\|_{L^2(\partial\Omega)^2}=1$.
\end{proof}

\begin{lemma}\label{Rstar.EL}
	If $u\in L^2(\partial\Omega)^2\setminus\{0\}$ is such that $P_-u=0$ and 
	$\rcal(u)=\rcal_\Omega$, then
	\begin{equation}
		(P_+)^*u=\frac{1}{\rcal_\Omega}( \sigma\cdot\nu)K_m( \sigma\cdot\nu)u.
		\label{limit-infty:eq2_lemma}
	\end{equation}
\end{lemma}

\begin{proof}
	Given $\epsilon\in\R$ and $v\in L^2(\partial\Omega)^2$, set $u_\epsilon=u+\epsilon P_+v$. Then, 
	$P_-u_\epsilon=P_-u+\epsilon P_-P_+v=0$ by \Cref{skew_projections}~$(ii)$. In addition, 
	$\|u_\epsilon\|_{L^2(\partial\Omega)^2}>0$ if $|\epsilon|$ is small enough because 
	$u\in L^2(\partial\Omega)^2\setminus\{0\}$ by assumption. Therefore, $\rcal(u)=\rcal_\Omega\geq \rcal(u_\epsilon)$ for all $|\epsilon|$ small enough, which entails
	\begin{equation}\label{EulerLagr R}
		\frac{d}{d\epsilon}\Big|_{\epsilon=0}\rcal(u_\epsilon)=0.
	\end{equation}
	On the one hand, 
	\begin{equation}
		\begin{split}
			\|u_\epsilon\|_{L^2(\partial\Omega)^2}^2
			&=\langle u+\epsilon P_+v,u+\epsilon P_+v\rangle_{L^2(\partial\Omega)^2}\\
			&=\|u\|_{L^2(\partial\Omega)^2}^2
			+2\epsilon\Re\langle u,P_+v\rangle_{L^2(\partial\Omega)^2}
			+\epsilon^2\|P_+v\|_{L^2(\partial\Omega)^2}^2,
		\end{split}
	\end{equation}
	which yields $\frac{d}{d\epsilon}\big|_{\epsilon=0}
	(\|u_\epsilon\|_{L^2(\partial\Omega)^2}^2)
	=2\Re\langle u,P_+v\rangle_{L^2(\partial\Omega)^2}$. On the other hand, since $K_m$ and $\sigma \cdot \nu$ are self-adjoint,
	\begin{equation}
		\begin{split}
			\langle(\sigma\cdot\nu)K_m(\sigma\cdot\nu)u_\epsilon,u_\epsilon\rangle_{L^2(\partial\Omega)^2}
			&=\langle(\sigma\cdot\nu)K_m(\sigma\cdot\nu)( u+\epsilon P_+v), u+\epsilon P_+v\rangle_{L^2(\partial\Omega)^2}\\
			&=\langle(\sigma\cdot\nu)K_m(\sigma\cdot\nu)u, u\rangle_{L^2(\partial\Omega)^2}\\
			&\quad+2\epsilon\Re\langle(\sigma\cdot\nu)K_m(\sigma\cdot\nu)u,P_+v\rangle_{L^2(\partial\Omega)^2}\\
			&\quad+\epsilon^2\langle(\sigma\cdot\nu)K_m(\sigma\cdot\nu)P_+v,P_+v\rangle_{L^2(\partial\Omega)^2},
		\end{split}
	\end{equation}
	which yields $\frac{d}{d\epsilon}\big|_{\epsilon=0}
	\langle(\sigma\cdot\nu)K_m(\sigma\cdot\nu)u_\epsilon,u_\epsilon\rangle_{L^2(\partial\Omega)^2}
	=2\Re\langle(\sigma\cdot\nu)K_m(\sigma\cdot\nu)u,P_+v\rangle_{L^2(\partial\Omega)^2}$. Therefore, plugging this in \eqref{EulerLagr R} and using that $\rcal(u)=\rcal_\Omega$, we deduce that
	\begin{equation}\label{EulerLagr R comput1}
		\begin{split}
			0&=\frac{1}{2}\|u\|_{L^2(\partial\Omega)^2}^2\frac{d}{d\epsilon}\Big|_{\epsilon=0}\rcal(u_\epsilon)\\
			&=\Re\langle(\sigma\cdot\nu)K_m(\sigma\cdot\nu)u,P_+v\rangle_{L^2(\partial\Omega)^2}
			-\rcal(u)\Re\langle u,P_+v\rangle_{L^2(\partial\Omega)^2}\\
			&=\Re\langle(\sigma\cdot\nu)K_m(\sigma\cdot\nu)u-\rcal_\Omega u,P_+v\rangle_{L^2(\partial\Omega)^2}
		\end{split}
	\end{equation}
	for all $v\in L^2(\partial\Omega)^2$. Replacing $v$ by $iv$ in \eqref{EulerLagr R comput1}, we conclude that
	\begin{equation}
		\begin{split}
			0&=\langle(\sigma\cdot\nu)K_m(\sigma\cdot\nu)u-\rcal_\Omega u,P_+v\rangle_{L^2(\partial\Omega)^2}
		\end{split}
	\end{equation}
	for all $v\in L^2(\partial\Omega)^2$, which leads to
	\begin{equation}\label{EulerLagr R comput2}
		\begin{split}
			(P_+)^*\big((\sigma\cdot\nu)K_m(\sigma\cdot\nu)-\rcal_\Omega\big) u
			=0\quad\text{in $L^2(\partial\Omega)^2$}.
		\end{split}
	\end{equation}
	Since $(P_+)^*(P_+)^*=(P_+)^*$ by 
	\Cref{skew_projections}~$(iii)$, 
	\eqref{EulerLagr R comput2} can be rewritten as
	\begin{equation}\label{EulerLagr R comput3}
		\begin{split}
			(P_+)^*\big((\sigma\cdot\nu)K_m(\sigma\cdot\nu)-\rcal_\Omega(P_+)^*\big) u
			=0\quad\text{in $L^2(\partial\Omega)^2$}.
		\end{split}
	\end{equation}
	
	Let us now show that we also have
	$(P_-)^*\big((\sigma\cdot\nu)K_m(\sigma\cdot\nu)-\rcal_\Omega(P_+)^*\big) u
	=0$. Using that
	$W_m(\sigma\cdot\nu)K_m=
	-K_m(\sigma\cdot\nu)W_m$ by \Cref{l.prop.K.W}~$(iii)$, we have
	\begin{equation}\label{Rstar.attained:eq3}
		\begin{split}
			(P_-)^*(\sigma\cdot\nu)K_m(\sigma\cdot\nu)
			&=({\textstyle\frac{1}{2}}+i(\sigma\cdot\nu)W_m)(\sigma\cdot\nu)K_m(\sigma\cdot\nu)\\
			&=(\sigma\cdot\nu)K_m(\sigma\cdot\nu)({\textstyle\frac{1}{2}}-iW_m(\sigma\cdot\nu))\\
			&=(\sigma\cdot\nu)K_m(\sigma\cdot\nu)P_-.
		\end{split}
	\end{equation}
	Combining \eqref{Rstar.attained:eq3} with the fact that 
	$(P_-)^*(P_+)^*=0$ by \Cref{skew_projections}~$(ii)$, and that $P_-u=0$, we deduce that 
	\begin{equation}\label{EulerLagr R comput4}
		\begin{split}
			(P_-)^*\big((\sigma\cdot\nu)K_m(\sigma\cdot\nu)-\rcal_\Omega(P_+)^*\big) u
			&=(P_-)^*(\sigma\cdot\nu)K_m(\sigma\cdot\nu) u\\
			&=(\sigma\cdot\nu)K_m(\sigma\cdot\nu)P_-u=0
			\quad\text{in $L^2(\partial\Omega)^2$}.
		\end{split}
	\end{equation}
	
	Finally, since $(P_+)^*+(P_-)^*=1$ by 
	\Cref{skew_projections}~$(i)$, summing \eqref{EulerLagr R comput3} and \eqref{EulerLagr R comput4} we arrive to
	\begin{equation}
		\begin{split}
			\big ( (\sigma\cdot\nu)K_m(\sigma\cdot\nu)-\rcal_\Omega(P_+)^* \big ) u
			=0
			\quad\text{in $L^2(\partial\Omega)^2$},
		\end{split}
	\end{equation}
	from which \eqref{limit-infty:eq2_lemma} follows.
\end{proof}

Combining the previous lemmas we establish \Cref{Th:Rayleigh.Intro}.

\begin{proof}[Proof of \Cref{Th:Rayleigh.Intro}]
	\Cref{Th:Rayleigh.Intro.l1} proves $(i)$, and \Cref{Th:Rayleigh.Intro.l2} yields $(ii)$. The proof of 
	$(iii)$ follows from \Cref{Rstar.attained,Rstar.EL}. \Cref{Prop:ParamFirstEigSphere} shows the last statement in the theorem regarding the ball, using \Cref{Th:Rayleigh.Intro.l1,Th:Rayleigh.Intro.l2}.
\end{proof}

\appendix

\section{Properties of the spectrum}
\label{Sec:Formulas}

Here we prove the properties of the spectrum of $\Dirac_\tau$ collected in \Cref{Lemma:SpectrumGenMIT}.
For a shorter notation, we will use 
\begin{equation}
	\mathcal{B} := -i \beta (\alpha \cdot \nu),
\end{equation}
which defines a self-adjoint operator in $L^2(\partial\Omega)^4$ such that $\mathcal{B}^2 = I_4$. For every $\varphi \in \mathrm{Dom}(\Dirac_\tau)$ it holds $\varphi = \sinh \tau \mathcal{B} \beta \varphi + \cosh \tau \mathcal{B} \varphi$ on $\partial \Omega$.

\begin{proof}[Proof of \Cref{Lemma:SpectrumGenMIT}]
	From \cite[Proposition 5.15]{BehrndtHolzmannMas} we know that $\Dirac_\tau$ is self-adjoint in $L^2(\Omega)^4$. The proof of $(i)$ follows from this and the compact embedding of 
	$H^1(\Omega)^4$ into $L^2(\Omega)^4$.
	
	Let us now show $(ii)$. 
	We first claim that for every $\varphi = (u,v)^\intercal \in \mathrm{Dom}(\Dirac_\tau)$ it holds
	\begin{equation}
		\label{Eq:SquareDiracNorm}
		\norm{ \Dirac \varphi}^2_{L^2(\Omega)^4} = \norm{ \alpha \cdot \nabla \varphi }^2_{L^2(\Omega)^4} + m^2 \norm{\varphi}^2_{L^2(\Omega)^4} + m e^\tau \norm{u}^2_{L^2(\partial \Omega)^2} + m e^{-\tau} \norm{v}^2_{L^2(\partial \Omega)^2}.
	\end{equation}
	To prove this formula, note first that expanding the square we have
	\begin{equation}
		\begin{split}
			\norm{ \Dirac \varphi}^2_{L^2(\Omega)^4} 
			&= \langle -i \alpha \cdot \nabla \varphi,  -i \alpha \cdot \nabla \varphi \rangle_{L^2(\Omega)^4} + m^2 \langle \beta \varphi, \beta \varphi \rangle_{L^2(\Omega)^4} + 2 m \Re \langle \beta \varphi, -i \alpha \cdot \nabla \varphi \rangle_{L^2(\Omega)^4} \\
			& = \norm{ \alpha \cdot \nabla \varphi }^2_{L^2(\Omega)^4} + m^2 \norm{\varphi}^2_{L^2(\Omega)^4} + 2 m \Re \langle \beta \varphi, -i \alpha \cdot \nabla \varphi \rangle_{L^2(\Omega)^4} .
		\end{split}	
	\end{equation}
	Integrating by parts we get
	\begin{equation}
		\begin{split}
			\langle \beta \varphi, -i \alpha \cdot \nabla \varphi \rangle_{L^2(\Omega)^4} & = \langle -i \alpha \cdot \nabla  (\beta \varphi), \varphi \rangle_{L^2(\Omega)^4} - \langle -i (\alpha \cdot \nu) \beta \varphi , \varphi \rangle_{L^2(\partial\Omega)^4} \\
			& = - \langle \beta (-i \alpha \cdot \nabla )\varphi, \varphi \rangle_{L^2(\Omega)^4} + \langle \mathcal{B} \varphi , \varphi \rangle_{L^2(\partial\Omega)^4} ,
		\end{split}
	\end{equation}
	and using that $\beta$ is hermitian, we obtain
	$2 \Re \langle \beta \varphi, -i \alpha \cdot \nabla \varphi \rangle_{L^2(\Omega)^4} = \langle \mathcal{B} \varphi , \varphi \rangle_{L^2(\partial\Omega)^4}$.
	Thus,
	\begin{equation}
		\label{Eq:SquareDiracNormPrev}
		\norm{ \Dirac \varphi}^2_{L^2(\Omega)^4} = \norm{ \alpha \cdot \nabla \varphi }^2_{L^2(\Omega)^4} + m^2 \norm{\varphi}^2_{L^2(\Omega)^4} + m \langle \mathcal{B} \varphi , \varphi \rangle_{L^2(\partial\Omega)^4}.
	\end{equation}
	Now, using that $\varphi \in \mathrm{Dom}(\Dirac_{\tau})$, we see that
	\begin{equation}
		\begin{split}
			\langle \mathcal{B} \varphi , \varphi \rangle_{L^2(\partial\Omega)^4} & = \langle \mathcal{B} \varphi , \sinh \tau \mathcal{B} \beta \varphi + \cosh \tau \mathcal{B} \varphi \rangle_{L^2(\partial\Omega)^4} = \langle  \varphi , \sinh \tau \mathcal{B}^2 \beta \varphi + \cosh \tau \mathcal{B}^2 \varphi \rangle_{L^2(\partial\Omega)^4} \\
			&= \langle  \varphi , \sinh \tau \beta \varphi + \cosh \tau \varphi \rangle_{L^2(\partial\Omega)^4},
		\end{split}
	\end{equation}
	where we have used that $\mathcal{B}$ is self-adjoint and that $\mathcal{B}^2 = I_4$.
	From here, and writing $\varphi = (u, v)^\intercal$, we easily see that
	\begin{equation}
		\langle \mathcal{B} \varphi , \varphi \rangle_{L^2(\partial\Omega)^4} = e^\tau \norm{u}^2_{L^2(\partial \Omega)^2} + e^{-\tau} \norm{v}^2_{L^2(\partial \Omega)^2}.
	\end{equation}
	Plugging this into \eqref{Eq:SquareDiracNormPrev} we obtain \eqref{Eq:SquareDiracNorm}, proving the claim.
	Now, let $\varphi\in \mathrm{Dom}(\Dirac_\tau)\setminus\{0\}$ be such that $\Dirac \varphi = \lambda \varphi$ in $\Omega$.	
	Note that, by Lemma~\ref{Lemma:ReproductionFormula}, $\varphi$ cannot vanish identically on $\partial\Omega$.
	Therefore, using \eqref{Eq:SquareDiracNorm} we obtain
	\begin{equation}
		\lambda ^2 \norm{\varphi}^2_{L^2(\Omega)^4} =  \norm{ \Dirac \varphi}^2_{L^2(\Omega)^4} >  m^2 \norm{\varphi}^2_{L^2(\Omega)^4},
	\end{equation}
	which yields $|\lambda | > m$.

	We finally prove $(iii)$ and $(iv)$. First, by the compact embedding of 
	$H^1(\Omega)^4$ into $L^2(\Omega)^4$ we have that the resolvent of $\Dirac_\tau$ is a compact operator, which yields that every eigenvalue has finite multiplicity. Now, given $\psi\in\C^4$, consider the charge conjugation operator
	$$
	\mathcal{C} \psi:= i \beta \alpha_{2} \overline{\psi}
	$$
	and the time reversal-symmetry operator
	\begin{equation}
		\label{Eq:TimeReversalSymmetry}
		T \psi: =-i \gamma_{5} \alpha_{2} \overline{\psi}, \quad \text{where}\quad  \gamma_{5}:= \begin{pmatrix}0&I_2\\I_2&0\end{pmatrix}.
	\end{equation}
	Then, simple computations show that $\Dirac T=T \Dirac$, $\Dirac \mathcal{C}=-\mathcal{C} \Dirac$, and $ T \mathcal{C}=  \mathcal{C} T$.
	In addition, setting
	\begin{equation}
		\mathcal{B}_\tau := \sinh \tau \mathcal{B} \beta + \cosh \tau \mathcal{B} = i (\sinh \tau - \cosh \tau \beta ) (\alpha \cdot \nu),
	\end{equation}
	it is also easy to check that $\mathcal B_{\tau} T=T \mathcal B_{\tau}$ and $ \mathcal B_{\tau} \mathcal{C}=\mathcal{C} \mathcal B_{-\tau}$. 
	Note that for every function $\varphi \in \mathrm{Dom}(\Dirac_{\tau})$ it holds $\varphi = \mathcal{B}_\tau \varphi$ on $\partial \Omega$.
	As a consequence, given an eigenfunction $\varphi$ of $\Dirac_\tau$ with eigenvalue $\lambda$, $T\varphi$ is also an eigenfunction of $\Dirac_\tau$ with eigenvalue $\lambda$. 
	Furthermore, $\mathcal{C}\varphi$ and $T\mathcal{C} \varphi$ are eigenfunctions of $\Dirac_{-\tau}$ with eigenvalue~$-\lambda$. 
\end{proof}

To conclude this section, we establish a formula which relates the $L^2(\Omega)^4$-norms of $\nabla \varphi$ and $\alpha \cdot \nabla \varphi$ for functions $\varphi \in \mathrm{Dom} (\Dirac_\tau)$.
Although we do not use this formula in this article, we think that it has its own interest, and it may be useful to present it here for future reference.
The formula is a generalization of \cite[formula (1.3)]{ALR2017}, in which the case $\tau = 0$ is considered, and the proof follows the same lines.

\begin{lemma}
	\label{Lemma:MeanCurvatureFormula}
	Let $\tau\in\R$ and $\varphi \in \mathrm{Dom} (\Dirac_\tau) \cap H^1(\partial \Omega)^4$.
	Then,
	\begin{equation}
		\label{Eq:MeanCurvatureFormula}
		\norm{ \alpha \cdot \nabla  \varphi}^2_{L^2(\Omega)^4} = \norm{ \nabla  \varphi}^2_{L^2(\Omega)^4} + \dfrac{1}{2} \int_{\partial \Omega} \kappa |\varphi|^2 d \upsigma  + \sinh \tau \langle  \gamma_5  (\alpha \cdot \nu) \varphi,  \alpha \cdot (\nu \times \nabla)  \varphi \rangle_{L^2(\partial\Omega)^4},
	\end{equation}	
	where $\kappa$ denotes the mean curvature of $\partial \Omega$.
\end{lemma}

\begin{proof}
	First, for every $\varphi \in H^2(\Omega)^4$, it holds
	\begin{equation}
		\label{Eq:AlphaGradH1}
		\norm{ \alpha \cdot \nabla  \varphi}^2_{L^2(\Omega)^4} = \norm{ \nabla  \varphi}^2_{L^2(\Omega)^4} - \langle \gamma_5 \varphi, -i \alpha \cdot (\nu \times \nabla) \varphi \rangle_{L^2(\partial\Omega)^4}, 
	\end{equation}	
	where $\gamma_5$ is  defined in \eqref{Eq:TimeReversalSymmetry}.
	This is proved in \cite[Appendix A.2]{ALR2017}.
	By density, it also holds for all $\varphi\in H^1(\Omega)^4 \cap H^1(\partial \Omega)^4$.
	
	Let us now investigate the boundary term in the above expression.	
	The crucial point is to use that the mean curvature of $\partial \Omega$ arises in our context through the formula
	\begin{equation}
		\label{Eq:MeanCurvatureCommutator}
		[-i \alpha \cdot (\nu \times \nabla), \mathcal{B}]= - \kappa \gamma_5 \mathcal{B},
	\end{equation}	
	where $[\cdot, \cdot]$ denotes the commutator of two operators, i.e., $[S,T] := ST - TS$; see \cite[Lemma~A.3]{ALR2017}.
	Using this and the boundary condition for $\varphi \in \mathrm{Dom}(\Dirac_\tau)\cap H^1(\partial \Omega)^4$ we get
	\begin{equation}
		\begin{split}
			\langle \gamma_5 \varphi &,-i \alpha \cdot (\nu \times \nabla) \varphi \rangle_{L^2(\partial\Omega)^4} \\
			& = \sinh \tau \langle \gamma_5 \varphi, -i \alpha \cdot (\nu \times \nabla) \mathcal{B} \beta \varphi \rangle_{L^2(\partial\Omega)^4} + \cosh \tau \langle \gamma_5 \varphi, -i \alpha \cdot (\nu \times \nabla) \mathcal{B}\varphi \rangle_{L^2(\partial\Omega)^4} \\
			&= \sinh \tau \langle \gamma_5 \varphi, [-i \alpha \cdot (\nu \times \nabla), \mathcal{B}] \beta \varphi \rangle_{L^2(\partial\Omega)^4}  + \sinh \tau \langle \gamma_5 \varphi, \mathcal{B} ( -i \alpha \cdot (\nu \times \nabla) ) \beta \varphi \rangle_{L^2(\partial\Omega)^4}
			\\
			& \quad \quad  + \cosh \tau \langle \gamma_5 \varphi, [-i \alpha \cdot (\nu \times \nabla), \mathcal{B}]\varphi \rangle_{L^2(\partial\Omega)^4}  + \cosh \tau \langle \gamma_5 \varphi, \mathcal{B} (-i \alpha \cdot (\nu \times \nabla))\varphi \rangle_{L^2(\partial\Omega)^4} \\
			&= -\sinh \tau \langle \gamma_5 \varphi, \kappa \gamma_5 \mathcal{B} \beta \varphi \rangle_{L^2(\partial\Omega)^4}+ \sinh \tau \langle \mathcal{B} \gamma_5 \varphi, -i \alpha \cdot (\nu \times \nabla) \beta \varphi \rangle_{L^2(\partial\Omega)^4}
			\\
			& \quad \quad  - \cosh \tau \langle \gamma_5 \varphi, \kappa \gamma_5 \mathcal{B} \varphi \rangle_{L^2(\partial\Omega)^4}  + \cosh \tau \langle \mathcal{B} \gamma_5 \varphi,  -i \alpha \cdot (\nu \times \nabla) \varphi \rangle_{L^2(\partial\Omega)^4} \\
			&= - \langle \gamma_5 \varphi, \kappa \gamma_5 (\sinh \tau \mathcal{B} \beta + \cosh \tau \mathcal{B} )\varphi \rangle_{L^2(\partial\Omega)^4} + \sinh \tau \langle \mathcal{B} \gamma_5 \varphi, -i \alpha \cdot (\nu \times \nabla) \beta \varphi \rangle_{L^2(\partial\Omega)^4}
			\\
			& \quad \quad  + \cosh \tau  \langle \mathcal{B} \gamma_5 \varphi,  -i \alpha \cdot (\nu \times \nabla) \varphi \rangle_{L^2(\partial\Omega)^4} \\
			&= - \langle \gamma_5 \varphi, \kappa \gamma_5 \varphi \rangle_{L^2(\partial\Omega)^4}  - \sinh \tau \langle \gamma_5 \mathcal{B}  \varphi, -i \alpha \cdot (\nu \times \nabla) \beta \varphi \rangle_{L^2(\partial\Omega)^4}
			\\
			& \quad \quad - \cosh \tau  \langle  \gamma_5 \mathcal{B}\varphi,  -i \alpha \cdot (\nu \times \nabla) \varphi \rangle_{L^2(\partial\Omega)^4} .\\
		\end{split}
	\end{equation}
	Here we have used that $(\alpha \cdot x) \gamma_5 =  \gamma_5 (\alpha \cdot x)$ for all $x\in \R^3$ (see \cite[Lemma~A.1]{ALR2017}) and that $\beta$ anticommutes with $\gamma_5$, thus $ \mathcal{B} \gamma_5 = - \gamma_5 \mathcal{B} $. 
	Now, using that  $(\alpha \cdot x) \beta = - \beta (\alpha \cdot x)$ for every $x\in \R^3$, and that $\beta$ anticommutes with $\mathcal{B}$ and $\gamma_5$, we have
	\begin{equation}
		\begin{split}
			\langle \gamma_5 \mathcal{B}  \varphi, -i \alpha \cdot (\nu & \times \nabla) \beta  \varphi \rangle_{L^2(\partial\Omega)^4}   = - \langle \gamma_5 \mathcal{B}  \varphi, \beta (-i \alpha \cdot (\nu \times \nabla) ) \varphi \rangle_{L^2(\partial\Omega)^4} \\
			&= - \langle \beta \gamma_5\mathcal{B}  \varphi, -i \alpha \cdot (\nu \times \nabla) \varphi \rangle_{L^2(\partial\Omega)^4} = -\langle  \gamma_5 \mathcal{B}  \beta \varphi, -i \alpha \cdot (\nu \times \nabla) \varphi \rangle_{L^2(\partial\Omega)^4}.
		\end{split}
	\end{equation}
	Hence,
	\begin{equation}
		\begin{split}
			\langle \gamma_5 \varphi, -i \alpha \cdot (\nu \times \nabla) \varphi \rangle_{L^2(\partial\Omega)^4} 
			&= - \langle \gamma_5 \varphi, \kappa \gamma_5 \varphi \rangle_{L^2(\partial\Omega)^4}  + \sinh \tau \langle  \gamma_5 \mathcal{B}  \beta \varphi, -i \alpha \cdot (\nu \times \nabla) \varphi \rangle_{L^2(\partial\Omega)^4}
			\\
			& \quad \quad - \cosh \tau  \langle  \gamma_5 \mathcal{B}\varphi,  -i \alpha \cdot (\nu \times \nabla) \varphi \rangle_{L^2(\partial\Omega)^4} \\
			&= - \langle \varphi, \kappa \varphi \rangle_{L^2(\partial\Omega)^4}  + 2\sinh \tau \langle  \gamma_5 \mathcal{B}  \beta \varphi, -i \alpha \cdot (\nu \times \nabla)  \varphi \rangle_{L^2(\partial\Omega)^4}
			\\
			& \quad \quad -  \langle  \gamma_5 (\sinh \tau \mathcal{B} \beta + \cosh \tau  \mathcal{B})\varphi,  -i \alpha \cdot (\nu \times \nabla) \varphi \rangle_{L^2(\partial\Omega)^4}
		\end{split}
	\end{equation}
	and thus, using the boundary condition, we get
	\begin{equation}
		\langle \gamma_5 \varphi, -i \alpha \cdot (\nu \times \nabla) \varphi \rangle_{L^2(\partial\Omega)^4} = - \dfrac{1}{2} \int_{\partial \Omega} \kappa |\varphi|^2 d \upsigma  + \sinh \tau \langle  \gamma_5 \mathcal{B}  \beta \varphi, -i \alpha \cdot (\nu \times \nabla)  \varphi \rangle_{L^2(\partial\Omega)^4},
	\end{equation}
	which combined with \eqref{Eq:AlphaGradH1} gives
	\begin{equation}
		\norm{ \alpha \cdot \nabla  \varphi}^2_{L^2(\Omega)^4} = \norm{ \nabla  \varphi}^2_{L^2(\Omega)^4} + \dfrac{1}{2} \int_{\partial \Omega} \kappa |\varphi|^2 d \upsigma  - \sinh \tau \langle  \gamma_5 \mathcal{B}  \beta \varphi, -i \alpha \cdot (\nu \times \nabla)  \varphi \rangle_{L^2(\partial\Omega)^4}.
	\end{equation}	
	Finally, using again that $(\alpha \cdot x) \beta = - \beta (\alpha \cdot x)$ for every $x\in \R^3$, we have $\mathcal{B}  \beta = i \alpha \cdot \nu$, and inserting this into the above identity we conclude the proof.
\end{proof}

\section{The ball}
\label{Sec:ExplicitComputationsSphere}
In this appendix we present a more explicit spectral analysis in the case that $\Omega\subset\R^3$ is a ball of radius $R>0$ centered at the origin, which will be denoted by $B_R$.
To study this radially symmetric case we introduce spherical coordinates: if $x\in \R^3$ we write $x=r\theta$ with $r = |x| \in [0,+\infty)$ and $\theta = x/|x| \in \Sph^2$. Using separation of variables and the spherical harmonic spinors, we give the explicit equations for the eigenvalues and eigenfunctions of $\Dirac_\tau$.

\addtocontents{toc}{\SkipTocEntry}
\subsection{Decomposition using spherical harmonic spinors}
\label{Subsec:DecompositionSph}

Let $Y^\ell_n$ be the usual spherical harmonics on $\Sph^2$; here $n = 0, 1, 2, \ldots$ and $\ell = -n, -n+1, \ldots, n + 1, n$. They satisfy {$\Delta_{\Sph^2} Y^\ell_n = -n(n + 1)Y^\ell_n$}, where $\Delta_{\Sph^2}$ denotes the usual spherical Laplacian. 
Moreover,  $Y^\ell_n$ form a complete orthonormal set in $L^2(\Sph^2)$.

Following \cite[Section 4.6.4]{Thaller}, the spherical harmonic spinors are defined as follows: for $j=1/2, 3/2, \ldots$ and $\mu_j = -j, -j+1, \ldots, j-1, j$, set
$$
\psi_{j-1 / 2}^{\mu_j}=\frac{1}{\sqrt{2 j}}\begin{pmatrix}
	\sqrt{j+\mu_j} Y_{j-1 / 2}^{\mu_j -1 / 2} \\
	\sqrt{j-\mu_j} Y_{j-1 / 2}^{\mu_j +1 / 2}
\end{pmatrix} \quad \text{ and } \quad \psi_{j+1 / 2}^{\mu_j}=\frac{1}{\sqrt{2 j+2}}\begin{pmatrix}
	\sqrt{j+1-\mu_j} Y_{j+1 / 2}^{\mu_j -1 / 2} \\
	-\sqrt{j+1+\mu_j} Y_{j+1 / 2}^{\mu_j +1 / 2}
\end{pmatrix}.
$$
As shown in \cite[Section 4.6.5]{Thaller}, one can decompose the space $L^2(\R^3)^4$ ---and analogously $L^2(B_R)^4$--- as
$$
L^2(\R^3)^4 = \bigoplus_{j=1/2}^{+\infty } \, \bigoplus_{\mu_j=-j}^j L_{j, \mu_j}^+ \oplus  L_{j, \mu_j}^-,
$$
where
$$
L_{j, \mu_j}^\pm := \left \{ \varphi  \in L^2(\R^3)^4  : \, \varphi(r\theta) = \begin{pmatrix}
	i \tilde{f}(r) \psi_{j \pm 1/2}^{\mu_j} (\theta) \vspace{3pt} \\
	\tilde{g}(r) \psi_{j \mp 1 / 2}^{\mu_j} (\theta)
\end{pmatrix} \text{ with } \tilde{f}, \tilde{g} \in L^2(\R_+, r^2 d r) \right\}.
$$
In each subspace define the mapping $U_{j,\mu_j}^\pm: L^\pm_{j, \mu_j} \to L^2(\R_+)^2$ by
$$
(U_{j,\mu_j}^\pm \varphi) (r) = \begin{pmatrix}
	r \tilde{f}(r)\\
	r \tilde{g}(r) 
\end{pmatrix} 
=: \begin{pmatrix}
	f(r)\\
	g(r) 
\end{pmatrix},
$$
and also define the differential operator (see \cite[equation (4.129)]{Thaller})
$$
\widehat{H}_{j,\pm} := \begin{pmatrix}
	m & - \partial_r + \kappa_{j,\pm}/r\\
	\partial_r + \kappa_{j,\pm}/r & -m
\end{pmatrix}, \quad \text{where } \kappa_{j,\pm} := \pm (j + 1/2).
$$
Then, the differential operator $\Dirac = -i \alpha \cdot \nabla + m \beta $ decomposes into the orthogonal sum of the operators $(U_{j,\mu_j}^\pm)^{-1} \widehat{H}_{j,\pm} U_{j,\mu_j}^\pm$.
In particular, if $\phi = (f,g)^\intercal$ satisfies $\widehat{H}_{j,\pm} \phi = \lambda \phi$ in $(0,R)$, then
\begin{equation}
	\label{Eq:SeparationOfVariables}
	\varphi =  \begin{pmatrix}
		u\\
		v 
	\end{pmatrix} = \begin{pmatrix}
		\dfrac{ i f(r)}{r} \psi^{\mu_j}_{j \pm 1/2}\vspace{3pt}\\ \dfrac{g(r)}{r} \psi^{\mu_j}_{j \mp 1/2}
	\end{pmatrix}
\end{equation}
satisfies $\Dirac\varphi = \lambda \varphi$ in $B_R\setminus \{0\}$.
As we will see, by further imposing that $f(0)$ is finite, we can guarantee that $\Dirac\varphi = \lambda \varphi$ holds across the origin. 

\addtocontents{toc}{\SkipTocEntry}
\subsection{Eigenvalue equations}
\label{Subsec:EigenvalueEquationsBall}

Our first goal is to find solutions to $\widehat{H}_{j,\pm} (f, g)^\intercal = \lambda (f,g)^\intercal$. This equation rewrites as the system of ODE
\begin{equation}
	\label{Eq:EigSphereODE}
	\beqc{\PDEsystem}
	- g' + \frac{\kappa}{r} g & = &(\lambda - m) f, \vspace{2pt}\\
	f' + \frac{\kappa}{r} f & = &(\lambda + m) g,
	\eeqc
\end{equation}
where $\kappa := \kappa_{j,\pm} := \pm (j + 1/2)$.
For simplicity, let us assume first that $\kappa = j + 1/2$.
To solve the system, note that from the second ODE we get
\begin{equation}
	\label{Eq:gIsolated}
	g = \dfrac{1}{\lambda + m} \Big (f' + \frac{\kappa}{r} f \Big )
\end{equation}
and, thus, inserting this into the first one we get the Bessel-type ODE
$$
f'' + \Big ( \lambda^2 - m^2 - \frac{\kappa^2 + \kappa}{r^2} \Big ) f = 0.
$$
Therefore, $f$ is of the form
$$
f(r) = b_1 \sqrt{r} J_{\kappa + 1/2}(\sqrt{\lambda^2 - m^2} r) + b_2 \sqrt{r} Y_{\kappa + 1/2}( \sqrt{\lambda^2 - m^2} r),
$$
where $b_1,b_2\in\C$, and $J_{\kappa + 1/2}$ and $Y_{\kappa + 1/2}$ denote the Bessel functions of the first and second kind of order ${\kappa + 1/2}$; see \cite[Chapters 9 and 10]{AbramowitzStegun}.
Since the eigenfunctions are not allowed to be singular at $r=0$ (as the corresponding $\varphi$ given by \eqref{Eq:SeparationOfVariables} must solve an elliptic equation across the origin), we deduce that $b_2=0$, and thus $f$ is of the form
\begin{equation}
	f(r) = b_1 \sqrt{r} J_{\kappa + 1/2}(\sqrt{\lambda^2 - m^2} r).
\end{equation}
Now, note that for every real index $p$, one has the relation 
$$
\partial_r [J_{p}(\sqrt{\lambda^2 - m^2} r) ]= \sqrt{\lambda^2 - m^2}  J_{p - 1}(\sqrt{\lambda^2 - m^2} r) - \dfrac{p}{r} J_{p}(\sqrt{\lambda^2 - m^2} r);
$$
see \cite[formula (9.1.27)]{AbramowitzStegun}.
Using this and \eqref{Eq:gIsolated}, we see that
\begin{equation}
	g(r) = b_1 \dfrac{\sqrt{\lambda^2 - m^2}}{\lambda + m}\sqrt{r} J_{\kappa - 1/2}(\sqrt{\lambda^2 - m^2} r).
\end{equation}

The case $\kappa = - j - 1/2$ follows by similar arguments.
One isolates $f$ instead of $g$ and uses that, for a positive integer $p$, $J_{-(p + 1/2)} = (-1)^{p+1} Y_{p + 1/2}$ and $Y_{-(p+ 1/2)} = (-1)^{p} J_{p + 1/2}$.

As a conclusion, we obtain that every eigenfunction of $\widehat{H}_{j,\pm}$ with eigenvalue $\lambda$ is, up to a multiplicative constant, of the form
\begin{equation}
	\label{Eq:EigenvalueGeneralForm}
	\begin{pmatrix}
		f(r)\\
		g(r)
	\end{pmatrix} =  \sqrt{r} \begin{pmatrix}
		J_{\ell + 1/2}(\sqrt{\lambda^2 - m^2} r) \vspace{2pt}\\
		\pm \dfrac{\sqrt{\lambda^2 - m^2}}{\lambda + m} J_{\ell'+1/2}(\sqrt{\lambda^2 - m^2} r) 
	\end{pmatrix},
\end{equation}
where $\ell = j \pm 1/2$ and $\ell' = j \mp 1/2$.

To obtain the equation \eqref{Eq:IntroEigEquationSphere} that relates $\lambda$ and 
$\tau$ by means of Bessel functions, it only remains to impose the boundary condition $v = i e^\tau (\sigma \cdot \nu) u$ on $\partial B_R$ for $\varphi=(u,v)^\intercal$ as in \eqref{Eq:SeparationOfVariables} and satisfying 
$\Dirac\varphi=\lambda\varphi$ in $B_R$.
Since 
\begin{equation}
	\label{Eq:SphSpinorsSymmetrySigma}
	(\sigma \cdot \nu) \psi_{j \pm 1/2}^{\mu_j} =   \psi_{j \mp 1/2}^{\mu_j},
\end{equation}
by \cite[equation (4.121)]{Thaller}, it follows from \eqref{Eq:SeparationOfVariables} that the boundary condition relating $f$ and $g$ is 
\begin{equation}
	\label{Eq:BoundaryConditionfg}
	g(R) = - e^\tau f(R).
\end{equation}
Therefore, for each $j = 1/2, 3/2, \ldots$, each $\mu_j = -j, -j + 1,\ldots, j$, and each subspace $L^\pm_{j, \mu_j}$, we obtain the  eigenvalue equation
\begin{equation}
	\label{Eq:EigenvalueEquationBall}
	e^\tau J_{\ell + 1/2}(\sqrt{\lambda^2 - m^2} R) \pm \dfrac{\sqrt{\lambda^2 - m^2}}{\lambda + m} J_{\ell'+1/2}(\sqrt{\lambda^2 - m^2} R) = 0,
\end{equation}
where $\ell = j \pm 1/2$ and $\ell' = j \mp 1/2$. This  corresponds to \eqref{Eq:IntroEigEquationSphere}.
Note that the equation is independent of the indexes $\mu_j$, accounting for the multiplicity of the eigenvalues.

\addtocontents{toc}{\SkipTocEntry}
\subsection{Parametrization of the eigenvalues}
\label{Subsec:ParamEigBall}

Our goal now is to exploit the eigenvalue equations given by \eqref{Eq:EigenvalueEquationBall} to prove that the eigenvalues of $\Dirac_{\tau}$ can be parametrized in terms of $\tau\in \R$, obtaining a family of increasing curves whose limits as $\tau\to \pm \infty$ are related with the zeroes of the Bessel functions (and thus with the eigenvalues of the Dirichlet Laplacian).
In the following lemma we collect the results on the Bessel functions that we will use.

\begin{lemma}
	\label{Lemma:Bessel}
	Let $J_p$ be the Bessel function of the first kind of order $p >0$, and denote the $k$-th positive zero of this function by $z_{p, k}$.
	
	Then,
	\begin{enumerate}
		\item[$(i)$] the positive zeroes of $J_p$ are simple and form an infinite increasing sequence,
		\item[$(ii)$] the zeroes of two consecutive Bessel functions are interlaced, meaning that
		\begin{equation}
			0< z_{p, 1} < z_{p + 1, 1} < z_{p, 2} < z_{p + 1, 2} < \ldots,
		\end{equation}
		\item[$(iii)$] the quotient of two consecutive Bessel functions can be expressed as
		\begin{equation}
			\label{Eq:QuotientBessel}
			\dfrac{J_{p + 1}(x)}{J_{p}(x)} = \sum_{k \geq 1} \dfrac{2x}{z_{p, k}^2 - x^2} \quad \text{ for } x \in \R \setminus \{z_{p, k}\}_{k\in \N}.
		\end{equation}
		As a consequence, $J_{p + 1}/J_{p}$ is odd, strictly increasing in each interval contained in $\R \setminus \{z_{p, k}\}_{k\in \N}$, it is positive in the intervals $(0,z_{p, 1})$ and $(z_{p + 1, k},z_{p, k + 1})$ for $k\geq 1$, and negative in the intervals $(z_{p, k},z_{p + 1, k})$ for $k\geq 1$.

	\end{enumerate}
\end{lemma}

\begin{proof}
	The first two statements are shown in \cite[Chapter XV]{Watson1995TreatiseBessel}.
	Note that for $p \geq -1$ the zeroes of $J_p$ are real, and thus we can order them; see also \cite[p. 372]{AbramowitzStegun}. 	
	Last, \eqref{Eq:QuotientBessel} follows from formula (1) in \cite[p.\,498]{Watson1995TreatiseBessel}.
	Note that this yields that $J_{p + 1}/J_{p}$ is an infinite sum of functions with singularities at $\pm z_{p, k}$ which have strictly positive derivative in their domain of definition. 
	As a consequence, in each interval of the form $(z_{p, k},z_{p, k + 1})$ for $k\in \N$, the function $J_{p + 1}/J_{p}$ is well defined, smooth, and strictly increasing, and therefore has a unique zero which necessarily is $z_{p + 1, k}$.
\end{proof}

With the help of the previous lemma we can now establish the following result on the parametrization of the eigenvalues.

\begin{proposition}
	\label{Prop:ParamEgienvaluesSphere}
	For each index $j = 1/2, 3/2, \ldots$ there exists an infinite number of smooth and strictly increasing functions $\{\tau \mapsto \uplambda_{j, k}^\pm(\tau) \}_{k \in \Z}$ such that, for each $\tau\in \R$, the real number $\uplambda_{j, k}^\pm(\tau)$ is an eigenvalue of $\Dirac_{\tau}$.
	The functions $\uplambda_{j, k}^\pm $ are surjectively defined by
	\begin{equation}
		\begin{matrix}
			\uplambda_{j, k}^\pm : &\R &\to &I_{j, k}^\pm \\
			&\tau &\mapsto& \uplambda_{j, k}^\pm(\tau),
		\end{matrix}
	\end{equation}
	with 
	\begin{itemize}
		\item $I_{j,0}^-= \big(m, \sqrt{(z_{j,1}/R)^2 + m^2}\,\big)$,
		\item $I_{j,k}^- = \big(\sqrt{(z_{j + 1,k}/R)^2 + m^2}, \sqrt{(z_{j,k + 1}/R)^2 + m^2}\,\big)$ for $k = 1, 2, \ldots$,
		\item $I_{j,k}^- = \big(\!-\!\sqrt{(z_{j + 1,|k|}/R)^2 + m^2}, - \sqrt{(z_{j,|k|}/R)^2 + m^2}\,\big)$ for $k=-1, -2, \ldots$,
	\end{itemize}
	and 
	\begin{itemize}
		\item $I_{j,0}^+= \big(\!-\!\sqrt{(z_{j,1}/R)^2 + m^2}, -m\big)$,
		\item $I_{j,k}^+ = \big(\!-\!\sqrt{(z_{j,k + 1}/R)^2 + m^2}, -\sqrt{(z_{j + 1,k}/R)^2 + m^2}\,\big)$ for $k = 1, 2, \ldots$,
		\item $I_{j,k}^+ = \big(\sqrt{(z_{j,|k|}/R)^2 + m^2}, \sqrt{(z_{j + 1,|k|}/R)^2 + m^2}\,\big)$ for $k=-1, -2, \ldots$,
	\end{itemize}
	where $z_{p, k}$ denotes the $k$-th positive zero of $J_p$, the Bessel function of order $p$. 
	As a consequence, for every $\tau \in \R$, the function
	\begin{equation}
		\label{Eq:EigExplicitSphere}
		\varphi_\tau :=  \begin{pmatrix}
			u_\tau\\
			v_\tau
		\end{pmatrix} = \dfrac{1}{\sqrt{r}} \begin{pmatrix}
			i J_{\ell + 1/2}\left (\sqrt{\uplambda_{j, k}^\pm(\tau)^2 - m^2} r \right ) \, \psi^{\mu_j}_{j \pm 1/2}(\theta) \vspace{3pt}\\  \pm \dfrac{\sqrt{\uplambda_{j, k}^\pm(\tau)^2 - m^2}}{\uplambda_{j, k}^\pm(\tau) + m} J_{\ell'+1/2} \left(\sqrt{\uplambda_{j, k}^\pm(\tau)^2 - m^2} r\right ) \, \psi^{\mu_j}_{j \mp 1/2} (\theta)
		\end{pmatrix}
	\end{equation}
	with $j = 1/2, 3/2, \ldots$, $\ell = j \pm 1/2$, $\ell' = j \mp 1/2$, $\mu_j = -j, -j + 1,\ldots, j$, and $k\in \Z$, 	belongs to $L^\pm_{j, \mu_j}$ and is an eigenfunction of $\Dirac_{\tau}$ with eigenvalue $\uplambda_{j, k}^\pm(\tau)$.	
\end{proposition}

Note that the superindex in $\uplambda_{j, k}^\pm$ indicates to which invariant subspace belongs the associated eigenfunction.
It should not be confused with the sign of the eigenvalue (as the superindex in $\lambda^\pm_1$ denotes, hence the different typography $\uplambda$ vs. $\lambda$).

\begin{proof}[Proof of \Cref{Prop:ParamEgienvaluesSphere}]
	From the arguments already presented in \Cref{Subsec:DecompositionSph} it only remains to show that from \eqref{Eq:EigenvalueEquationBall} we can obtain the aforementioned infinite number of parametrizations of $\lambda$ in terms of $\tau\in \R$.
	To do it, the idea is to rewrite the eigenvalue equation \eqref{Eq:EigenvalueEquationBall} as 
	\begin{equation}
		\label{Eq:EigenvalueEquationBallParamGeneral}
		e^\tau =  \mp \dfrac{\sqrt{\lambda^2-m^2}}{\lambda+m} \dfrac{ J_{\ell' + 1/2} ( \sqrt{\lambda^2-m^2} R ) } {J_{\ell + 1/2} ( \sqrt{\lambda^2-m^2} R ) } =: h(\lambda).
	\end{equation}
	Then, our goal will be to invert $h$ in suitable intervals to get 
	$\lambda=\lambda(\tau):=h^{-1}(e^\tau)$. 
	
	First, note that we can restrict ourselves to the subspaces $L^-_{j, \mu_j}$, thanks to the odd symmetry mentioned in \Cref{Rem:SymmetrySphere}.
	In this case the eigenvalue equation is written as
	\begin{equation}
		\label{Eq:EigEquationSphereParamBis}
		e^\tau =  \sign (\lambda+m ) \sqrt{\dfrac{\lambda-m}{\lambda+m}} \dfrac{ J_{j + 1} ( \sqrt{\lambda^2-m^2} R ) } {J_{j} ( \sqrt{\lambda^2-m^2} R ) } = h(\lambda).
	\end{equation}
	Due to the fact that $e^\tau>0$ for all $\tau\in\R$, we are forced to work with $h$ only on intervals $I\subset\R$ such that
	$h(I)\subset(0,+\infty)$ and where $h$ is invertible.
	Since $\frac{\lambda-m}{\lambda+m}$ is positive and strictly increasing for $\lambda \in (m, +\infty)$ and for $\lambda \in (- \infty, -m)$, $I$ must be such that 
	\begin{equation}
		\sign (\lambda+m ) \dfrac{ J_{j + 1} ( \sqrt{\lambda^2-m^2} R ) } {J_{j} ( \sqrt{\lambda^2-m^2} R ) } \text{ is positive and strictly increasing for } \lambda \in I. 
	\end{equation}
	Then, \Cref{Lemma:Bessel} yields that all the possible intervals $I$ are:
	\begin{itemize}
		\item $I= (m, \sqrt{(z_{j,1}/R)^2 + m^2})$,
		\item $I = (\sqrt{(z_{j + 1,k}/R)^2 + m^2}, \sqrt{(z_{j,k + 1}/R)^2 + m^2})$ for $k\geq 1$,
		\item $I = (-\sqrt{(z_{j + 1,k}/R)^2 + m^2}, - \sqrt{(z_{j,k}/R)^2 + m^2})$ for $k\geq 1$.
	\end{itemize}
	In each of these intervals $I$ the function $h: I \to (0,+\infty)$ is of class $C^\infty$, strictly increasing, and surjective.
	Therefore, it can be inverted, obtaining a $C^\infty$ function $\tau \mapsto \lambda(\tau) := h^{-1}(e^\tau)$ which maps $\R$ into $I$ surjectively and corresponds, for each $\tau\in \R$, to an eigenvalue of $\Dirac_\tau$.
	In addition, the monotonicity of $\lambda\mapsto\tau=\tau(\lambda):=\log(h(\lambda))$ yields that $\tau\mapsto\lambda(\tau)$ is also strictly increasing.
\end{proof}

The previous result yields that, for any given eigenvalue curve $\tau \mapsto \lambda(\tau)$,  $\lim_{\tau \to \pm \infty}|\lambda(\tau)|$ is either $m$ or a positive zero of the function $J_{k + 1/2} ( \sqrt{(\cdot)^2-m^2} R )$ for some $k=0,1,2,\ldots$; note that each of these zeroes corresponds to the square root of a Dirichlet eigenvalue of $-\Delta + m^2$ in $B_R$. 
The monotonicity and limiting values of $\tau\mapsto\lambda(\tau)$ was observed in the curves plotted in \Cref{Fig:ALL_eigenvalues}.
Furthermore, the alternation (with respect to the zeroes of the Bessel function) between positive and negative eigenvalue curves, which is given by the  intervals $I_{j,k}^\pm$ defined in \Cref{Prop:ParamEgienvaluesSphere}, was already shown numerically in \Cref{Fig:j12_eigenvalues,Fig:j12BOTH_eigenvalues}.

\addtocontents{toc}{\SkipTocEntry}
\subsection{The first positive eigenvalue}
\label{Subsec:FirstEigenvalueBall}

In this section we focus on the first positive eigenvalue of $\Dirac_{\tau}$ when $\Omega = B_R$. We provide a fine description of the associated eigenvalue curve, whose main properties are summarized in the following proposition. The reader may compare it with \Cref{Th:FirstEigTom}, which is the analogous result for general domains; see also \Cref{Th:Rayleigh.Intro} regarding \eqref{Eq:LStarBall}.

\begin{proposition}
	\label{Prop:ParamFirstEigSphere}
	The function $\tau \mapsto \lambda_1^+(\tau)=\min(\sigma(\Dirac_\tau)\cap(m,+\infty))$ is of class $C^\infty$ in $\R$, and satisfies
	\begin{equation}
		\label{Eq:LimitsFirstEigBall}
		\lim_{\tau \downarrow - \infty} \lambda_1^+ (\tau) = m \quad \text{and} \quad  \lim_{\tau \uparrow + \infty} \lambda_1^+ (\tau) = \sqrt{\pi^2/R^2 + m^2} = \sqrt{ \min\sigma(-\Delta_D) + m^2},
	\end{equation}
	where $\min\sigma(-\Delta_D)$ denotes the first Dirichlet eigenvalue of $-\Delta$ in $B_R$.
	In addition, the corresponding eigenspace associated to $\lambda_1^+(\tau)$ has always dimension $2$.
	
	Furthermore,
	\begin{equation}
		\label{Eq:LStarBall}
		L^\star_{B_R} := \lim_{\tau \downarrow -\infty} (\lambda_1^+(\tau)-m)e^{-\tau} = \frac{3}{R}=\frac{1}{\rcal_{B_R}},
	\end{equation}
	where $\rcal_{B_R}$ is defined in \eqref{form.rayleigh.r.omega}.
\end{proposition}

\begin{proof}

	Let $\lambda : \R \to \big(m, \sqrt{\pi^2/R^2 + m^2}\big)$ be the eigenvalue curve corresponding to $\uplambda_{1/2, 0}^-$ in the notation of \Cref{Prop:ParamEgienvaluesSphere}.
	This eigenfunction is associated to the two subspaces  $L^-_{1/2, \mu_{1/2}}$, with either  $\mu_{1/2} =  1/2$ or $\mu_{1/2} = -1/2$, and solves the implicit equation
	\begin{equation}
		\label{Eq:EigenvalueEquationBallParamFirst}
		e^\tau =   \dfrac{\sqrt{\lambda^2-m^2}}{\lambda+m} \dfrac{ J_{3/2} ( \sqrt{\lambda^2-m^2} R ) } {J_{1/2} ( \sqrt{\lambda^2-m^2} R ) }.
	\end{equation}
	We will show that $\lambda = \lambda_1^+$. 
	The upper bound for $\lambda(\tau)$ (and limit as $\tau \to +\infty$) is given by the fact that $z_{1/2, 1} = \pi$, which follows from the explicit expression of the Bessel functions involved in the above equation:	
	\begin{equation}
		\label{Eq:BesselJ1232}
		J_{1/2} (x) = \sqrt{\dfrac{2}{\pi x}} \sin(x)  \quad \text{ and } \quad J_{3/2} (x) = \sqrt{\dfrac{2}{\pi x}} \left( \dfrac{\sin(x)}{x} - \cos(x) \right).
	\end{equation}
	As a matter of fact, these expressions can be used to show ---after a tedious computation and using that $x>\sin(x)$ for all $x\in(0,\pi)$--- that the right-hand side of \eqref{Eq:EigenvalueEquationBallParamFirst} is a strictly increasing function of $\lambda$ for $\lambda \in \big(m, \sqrt{\pi^2/R^2 + m^2}\big)$ without making use of \Cref{Lemma:Bessel} (as done in the proof of \Cref{Prop:ParamEgienvaluesSphere}).

	Since $\pi = z_{1/2, 1}$ is the smallest positive zero among all the positive zeroes of the Bessel functions of half-integer index ---as shown by \Cref{Lemma:Bessel}~($ii$)---, it follows that $\lambda(\tau) $ coincides with $ \lambda^+_1(\tau)$ at least for big enough values of $\tau$.
	To show that indeed $\lambda(\tau)$ is the first positive eigenvalue $\lambda_1^+ (\tau)$ for all $\tau \in \R$, it suffices to show that $\lambda$ cannot cross any other eigenvalue curve. 
	On the one hand, taking into account the possible intervals $I_{j,k}^+$ given in \Cref{Prop:ParamEgienvaluesSphere}, it follows that any positive eigenvalue curve associated to the spaces $L^+_{j, \mu_{j}}$ must lie above $\sqrt{\pi^2/R^2 + m^2}$, and thus it cannot cross $\lambda(\tau)$.
	On the other hand, if there was a crossing between $\lambda(\tau)$ and another positive eigenvalue curve associated to $L^-_{j_\circ, \mu_{j_\circ}}$ for some half-integer  $j_\circ > 1/2$, then by \eqref{Eq:EigenvalueEquationBallParamFirst} there would exist a point $x_\circ \in (0, z_{1/2, 1})$ such that
	\begin{equation}
		\label{Eq:CrossingBessel}
		\dfrac{ J_{3/2} (x_\circ) } {J_{1/2} (x_\circ ) } = \dfrac{ J_{j_\circ+ 1} (x_\circ) } {J_{j_\circ} (x_\circ ) }.
	\end{equation}
	However, since the zeroes of the Bessel functions are ordered (see \Cref{Lemma:Bessel}), for every half-integer $j\geq 1/2$ it follows that $z_{j, k} < z_{j + 1,k}$ for all $k \geq 1$, and therefore
	\begin{equation}
		\dfrac{2x}{z_{j, k}^2 - x^2}  > \dfrac{2x}{z_{j + 1, k}^2 - x^2} \quad \text{ for every } x \in (0, z_{1/2, 1}).
	\end{equation}
	Hence, by  \eqref{Eq:QuotientBessel} it follows that 
	\begin{equation}
		\dfrac{ J_{j + 1} (x) } {J_{j} (x ) } > \dfrac{ J_{j+2} (x) } {J_{j+1} (x ) } \quad \text{ for every } x \in (0, z_{1/2, 1}) \text{ and for every half-integer } j \geq 1/2,
	\end{equation}
	thus there cannot exist such a $x_\circ \in (0, z_{1/2, 1})$ satisfying \eqref{Eq:CrossingBessel}.
	In conclusion, $\lambda$ does not cross any other eigenvalue curve. Therefore,  $\lambda(\tau)$ is the first positive eigenvalue of $\Dirac_\tau$ for all $\tau \in \R$.
	As a byproduct, since $\lambda(\tau)$ is associated to $L^-_{1/2, \mu_{1/2}}$ with either $\mu_{1/2} = 1/2$ or $\mu_{1/2} = -1/2$, it follows that the first positive eigenvalue of $\Dirac_\tau$ has multiplicity $2$ for all $\tau \in \R$.

	To conclude the proof, we are only left to show that 
	$L_{B_R}^* = 3/R=1/\rcal_{B_R}$.
	First, note that from the rescaling properties of the operators $K_m$ and $W_m$ defined in \eqref{Eq:DefK}, it follows readily that $\rcal_{B_R} = R\rcal_{B_1}$.
	Moreover, $L_{B_R}^\star= L_{B_1}^\star/R$. 
	To show this second equality, one notices that if $u_1$ and $u_R$ denote the boundary values of the upper component of the first eigenfunction in $B_1$ and $B_R$ respectively, both associated to the same subspace $L^-_{1/2, \mu_{1/2}}$, then after a normalization one can choose them in such a way that $u_R(\cdot) = u_1(\cdot/R)$. 
	Hence, from \eqref{Eq:BoundaryPbComm} and taking the limit $\tau \downarrow -\infty$, using again the scaling of $K_m$, and that $\{W_m, \sigma\cdot \nu \} = 0$ by \Cref{Lemma:EquivalencesPBall}, it follows that $L_{B_R}^\star= L_{B_1}^\star/R$; see the argument to get to \eqref{Eq:BoundaryPbSphLimit} below for more details.  
	As a consequence of all this, it is enough to prove the result for $R=1$.
	
	Recall that $\lambda(\tau)$ is associated, for all $\tau \in \R$, to the two subspaces $L^-_{1/2, \mu_{1/2}}$ with either $\mu_{1/2}= 1/2$ or  $\mu_{1/2}= -1/2$.
	We will work in one of these two subspaces (the precise choice will be completely irrelevant for the rest of the argument), and therefore, after a normalization, we can take an eigenfunction $\varphi_\tau = (u_\tau, v_\tau)^\intercal$ associated to $\lambda(\tau)$	such that $u_\tau$ at the boundary of $B_1$ is given by 
	$ \psi^{\mu_{1/2}}_0$ for all $\tau \in \R$.
	By \Cref{Lemma:BoundaryPb}, it holds
	\begin{equation}
		\label{Eq:BoundaryPbSph}
		\begin{split}
			\Big ( \frac{1}{2} -i W_\lambda  ({\sigma}\cdot\nu) \Big )\psi^{\mu_{1/2}}_0 & =   - (\lambda+m) e^\tau K_\lambda \psi^{\mu_{1/2}}_0\quad\text{and}\\
			\Big ( \frac{1}{2} -i ({\sigma}\cdot\nu) W_\lambda   \Big )\psi^{\mu_{1/2}}_0 & =  (\lambda-m)e^{-\tau} ({\sigma}\cdot\nu ) \, K_\lambda ({\sigma}\cdot\nu)  \psi^{\mu_{1/2}}_0
		\end{split}
	\end{equation}
	in $L^2(\Sph^2)^2$.
	We will take the limit $\tau \downarrow -\infty$ in \eqref{Eq:BoundaryPbSph}, taking into account that $\lambda \downarrow m$ as $\tau \downarrow -\infty$.
	Letting  $\tau \downarrow -\infty$ in \eqref{Eq:BoundaryPbSph}, and using that then $K_\lambda \to K_m$ and $W_\lambda \to W_m$ as bounded operators in $L^2(\Sph^2)^2$ (which follows from the explicit expressions of the operators),  we obtain
	\begin{equation}
		\label{Eq:BoundaryPbSphLimit}
		\begin{split}
			P_- \, \psi^{\mu_{1/2}}_0  =  0 \quad\text{and} \quad (P_+)^*\,  \psi^{\mu_{1/2}}_0  =   L_{B_1}^\star ({\sigma}\cdot\nu ) \, K_m ({\sigma}\cdot\nu) \psi^{\mu_{1/2}}_0 \quad \text{ in } L^2(\Sph^2)^2,
		\end{split}
	\end{equation}
	where  $L^\star_{B_1} =\lim_{\tau \downarrow -\infty} (\lambda(\tau)-m)e^{-\tau}$, and $P_- =  \frac{1}{2} -i W_m  ({\sigma}\cdot\nu)$ and $(P_+)^*= \frac{1}{2} -i ({\sigma}\cdot\nu) W_m $, as defined in \Cref{Subsec:Proj.Hardy}.
	In particular, \eqref{Eq:BoundaryPbSphLimit} shows that $L^\star_{B_1}$ is finite, since all the involved operators  are bounded and $K_m$ is injective.
	It is worth pointing out that this last argument leading to $L_{B_1}^\star< +\infty$ works thanks to the fact that, on $\partial \Omega$, the eigenfunction $u_\tau = \psi^{\mu_{1/2}}_0$ is indeed independent of $\tau$, something that may not be guaranteed on a general domain $\Omega$.

	We will use \eqref{Eq:BoundaryPbSphLimit} to establish \eqref{Eq:LStarBall} for $R=1$.
	First, let $L\in \mathcal{L}_{B_1}$ ---recall that $\mathcal{L}_{B_1}$ is defined in \eqref{form.rayleigh.l.omega.setL}, see also \eqref{form.rayleigh.l.omega.setL.true}. 
	We claim that $1/L$ is an eigenvalue of $K_m$.
	This claim can be proven by adding the two equations from which $\mathcal{L}_{B_1}$ is defined and using that $\{W_m, \sigma\cdot \nu\} = 0$ as operators in $L^2(\Sph^2)^2$; see \Cref{Lemma:EquivalencesPBall}.
	As a consequence, and since $1/\rcal_\Omega = \min \lcal_{\Omega}$ for every $\Omega \subset \R^3$ by \Cref{Th:Rayleigh.Intro}, $ \rcal_{B_1}$ is the maximum of the eigenvalues of $K_m$ among eigenfunctions of the form $(\sigma\cdot\nu)u$ with $u\in P_+(L^2(\Sph^2)^2)$, where $P_+=\frac{1}{2}+ iW_m(\sigma\cdot\nu)$.
	
	Let us now compute explicitly the eigenvalues of $K_m$ as an operator in $L^2(\Sph^2)^2$.
	Using \cite[Lemma~4.3]{AMV2} and 	\cite[Theorem~3.6]{MasPizzichillo-Sphere} on $K_a$ for $a>0$ and letting\footnote{Note that the kernel $k_m$ of $K_m$ does not depend on $m$. Hence, we can assume here that $m>0$ to cover as well the case $m=0$.} $a \uparrow m$, it follows that the spectrum of $K_m$ is given by a sequence $\{ d_{j\pm 1/2} \}_{j = 1/2, \ 3/2, \ldots}$ whose associated eigenfunctions are $\psi_{j \pm 1 / 2}^{\mu_j}$, i.e., 
	\begin{equation}\label{Eq:djExpression.diagonal}
		K_m \psi_{j \pm 1 / 2}^{\mu_j} = d_{j\pm 1/2}\, \psi_{j \pm 1 / 2}^{\mu_j},
	\end{equation}
	and the eigenvalues $ d_{j\pm 1/2}$ are given by the expression
	\begin{equation}
		\label{Eq:djExpression}
		d_{j\pm 1/2} = \lim_{t\downarrow 0} \ical_{j \pm 1/2 +1/2}(t) \mathcal{K}_{j \pm 1/2 +1/2}(t).
	\end{equation}
	Here $\ical_\kappa$ and $\mathcal{K}_\kappa$ are the modified Bessel functions of the first and second kind of order $\kappa$, respectively.
	By \cite[formulas 9.6.7 and 9.6.9]{AbramowitzStegun} we have
	\begin{equation}
		\lim_{t\downarrow 0} \ical_{\kappa}(t) \mathcal{K}_{\kappa}(t) = \dfrac{\Gamma (\kappa)}{2 \Gamma (\kappa + 1)} = \dfrac{1}{2 \kappa } \quad \text{for } \kappa > 0.
	\end{equation}
	Hence, for $k= 0, 1, \ldots,$ we have
	\begin{equation}
		\label{Eq:d_kExpression}
		d_{k} = \dfrac{1}{2k +1}.
	\end{equation}

	We claim that $(\sigma \cdot \nu)  \psi_0^{\mu_j} = \psi_1^{\mu_j} $ is orthogonal to $P_+(L^2(\Sph^2)^2)$. 
	Once this is proved, it follows that $1/d_0 \notin \lcal_{B_1}$ and, as a consequence, $\rcal_{B_1} \leq  d_1 = 1/3$.
	To prove the claim, notice that for all $\tau\in \R$, after a normalization,  $\psi_1^{\mu_j}$ is the upper component of the eigenfunction of $\Dirac_{\tau}$ associated to $-\lambda(-\tau)$, i.e., the first (larger) negative eigenvalue of $\Dirac_{\tau}$.
	Indeed, this can be shown with exactly the same arguments as we did at the beginning of the proof of the proposition for the first positive eigenvalue, considering in this case the subspace $L^+_{1/2, \mu_{1/2}}$; see also \Cref{Rem:SymmetrySphere}.
	Therefore, since such eigenvalue converges to $-m$ as $\tau \uparrow +\infty$, using the second equation in \eqref{Eq:BoundaryPb} with $u= \psi_1^{\mu_j}$	and taking the limit $\tau \uparrow +\infty$ ---analogously as how we proceeded before to show \eqref{Eq:BoundaryPbSphLimit}---, it follows that $(P_+)^* \psi_1^{\mu_j} = 0$, establishing our claim and, as a byproduct, the inequality $\rcal_{B_1} \leq  d_1 = 1/3$.
	
	We shall now prove that $L_{B_1}^\star = 1/ d_1 = 3$, which will finish the proof of \eqref{Eq:LStarBall} since $1/\rcal_{B_1} \leq L_{B_1}^*$ by \Cref{Th:Rayleigh.Intro.l1,Th:Rayleigh.Intro.l2}.
	By adding the two equations in \eqref{Eq:BoundaryPbSphLimit} and using again that $\{W_m,\sigma \cdot \nu  \}  = 0$ (or equivalently, that $(P_+)^* = P_+$ by \Cref{Lemma:EquivalencesPBall} since the underlying domain is a ball), we obtain
	\begin{equation}
		\psi^{\mu_{1/2}}_0  =   L_{B_1}^\star ({\sigma}\cdot\nu ) \, K_m ({\sigma}\cdot\nu) \psi^{\mu_{1/2}}_0 \quad \text{in } L^2(\Sph^2)^2.
	\end{equation}
	Finally, using that $({\sigma}\cdot\nu) \psi^{\mu_{1/2}}_{j\pm 1/2} = \psi^{\mu_{1/2}}_{j \mp 1/2}$ for all half-integers $j$ (see \cite[equation (4.121)]{Thaller}), and that $\psi^{\mu_{1/2}}_1$ is an eigenfunction of $K_m$ with eigenvalue $d_1$ by \eqref{Eq:djExpression.diagonal}, we get
	\begin{equation}
		\psi^{\mu_{1/2}}_0  = L_{B_1}^\star d_1 \psi^{\mu_{1/2}}_0 \quad \text{ in } L^2(\Sph^2)^2.
	\end{equation}
	Therefore, we have $L_{B_1}^\star = 1/d_1 = 3$ by \eqref{Eq:d_kExpression}. This concludes the proof.
\end{proof}

\bibliographystyle{siam}
\bibliography{biblio}

\end{document}